\documentclass[14pt]{amsart}
\usepackage{cases}
\usepackage{bm}
\usepackage{amsfonts,amsmath,amssymb,amscd,bbm,amsthm,mathrsfs,dsfont}
\usepackage{mathrsfs}
\usepackage{color}
\usepackage{xypic}
\usepackage{pifont}
\usepackage[all]{xy}
\usepackage{breqn}
\usepackage{graphicx}
\usepackage{epstopdf}
\usepackage{float}

\newtheorem{Theorem}{Theorem}[section]
\newtheorem{Lemma}[Theorem]{Lemma}
\newtheorem{Definition}[Theorem]{Definition}
\newtheorem{Corollary}[Theorem]{Corollary}
\newtheorem{Proposition}[Theorem]{Proposition}
\newtheorem{Example}[Theorem]{Example}
\newtheorem{Remark}[Theorem]{Remark}
\newtheorem{Conjecture}[Theorem]{Conjecture}
\newtheorem{Problem}[Theorem]{Problem}

\newtheorem{Construction}[Theorem]{Construction}
\newtheorem*{Theorem*}{Theorem}

\newtheorem{Condition}[Theorem]{Condition}

\title [Recurrence formula, positivity and polytope basis in cluster algebras ]
{Recurrence formula, positivity and polytope basis in  cluster algebras via Newton polytopes}

\author{Fang Li}\author{Jie Pan}

\address{Fang Li
\newline Department of Mathematics, Zhejiang University (Yuquan Campus), Hangzhou, Zhejiang 310027, P. R. China}
\email{fangli@zju.edu.cn}

\address{Jie Pan
\newline Department of Mathematics, Zhejiang University (Yuquan Campus), Hangzhou, Zhejiang 310027, P. R. China}
\email{11635015@zju.edu.cn}


\newenvironment{Proof}[1][Proof]{\begin{trivlist}
\item[\hskip \labelsep {\bfseries #1}]}{\flushright
$\Box$\end{trivlist}}

\date{version of \today}

\newcommand{\lra}{\longrightarrow}

\newcommand{\ra}{\rightarrow}
\newcommand{\sdp}{\times\kern-.2em\vrule height1.1ex depth-.05ex}
\newcommand{\epi}{\lra \kern-.8em\ra}

\renewcommand{\P}{{\mathbb P}}
\newcommand{\Q}{{\mathbb Q}}
\newcommand{\N}{{\mathbb N}}
\newcommand{\R}{{\mathbb R}}
\newcommand{\F}{{\mathcal F}}
\newcommand{\A}{{\mathcal A}}

\newcommand{\T}{{\mathbb T}}

\newcommand{\Z}{{\mathbb Z}}
\renewcommand{\'}{{^{\prime}}}

 \normalbaselines
\allowdisplaybreaks
\begin{document}

\renewcommand{\thefootnote}{\alph{footnote}}
\maketitle
\bigskip
 \begin{abstract}
  In this paper, we study the Newton polytopes of $F$-polynomials in a totally sign-skew-symmetric cluster algebra $\mathcal A$ and generalize them to a larger set consisting of polytopes $N_{h}$ associated to vectors $h\in\Z^{n}$ as well as $\widehat{\mathcal{P}}$ consisting of polytope functions $\rho_{h}$ corresponding to $N_{h}$.

  The main contribution contains that
  (i)\; obtaining a {\em recurrence construction} of the Laurent expression of a cluster variable in a cluster from its $g$-vector;
  (ii)\; proving the subset $\mathcal{P}$ of $\widehat{\mathcal{P}}$ consisting of Laurent polynomials in $\widehat{\mathcal{P}}$ is a strongly positive $\Z Trop(Y)$-basis for $\mathcal{U}(\A)$ consisting of certain universally indecomposable Laurent polynomials when $\A$ is a cluster algebra with principal coefficients. For a cluster algebra $\mathcal A$ over arbitrary semifield $\mathbb P$ in general, $\mathcal{P}$ is a strongly positive $\Z\P$-basis for a subalgebra  $\mathcal{I_P(A)}$ (called the intermediate cluster algebra of $\mathcal A$) of $\mathcal{U(A)}$.  We call $\mathcal P$ the {\em polytope basis};
  (iii)\; constructing some explicit maps among corresponding $F$-polynomials, $g$-vectors, $d$-vectors and cluster variables to characterize their relationship.

  As an application of (i), we give an affirmation to the positivity conjecture of cluster variables in a totally sign-skew-symmetric cluster algebra, which in particular  provides a new method different from that given by Gross-Hacking-Keel-Kontsevich in \cite{GHKK} to present the positivity of cluster variables in the skew-symmetrizable case. As another application, a conjecture on Newton polytopes posed by Fei is answered affirmatively.

  For (ii), we know that in rank 2 case, $\mathcal{P}$ coincides with the greedy basis introduced by Lee-Li-Zelevinsky in \cite{LLZ}. Hence, we can regard the polytope basis $\mathcal{P}$ as a natural generalization of the greedy basis in general rank.

  As an application of (iii),  the positivity of denominator vectors associated to non-initial cluster variables, which was a conjecture raised in \cite{FZ4}, is proved in a totally sign-skew-symmetric cluster algebra.
\end{abstract}

\tableofcontents

\section{Introduction and preliminaries}
\subsection{Introduction}\quad

Cluster algebras are first constructed by Fomin and Zelevinsky in \cite{FZ1}. Generally speaking, it is a commutative algebra with so-called exchange relations given by an extra combinatorial structure. Later, researchers found many relationships from the theory of cluster algebras to other topics, such as Lie theory, quantum groups, representation theory, Riemann surfaces with triangulation, number theory, tropical geometry and Grassmanian theory as well as many interesting properties. The most significant properties among them are the Laurent phenomenon and positivity of varieties which claim that each cluster variables can be expressed as a Laurent polynomial in any cluster over $\N\P$. However, the calculation of the Laurent expression of a cluster variable in a given cluster is in general difficult. One of our aims in this paper is to provide recurrence formulas as a program to make the above calculation easier.

Two bases related to an upper cluster algebra $\mathcal{U}(\A)$, called the greedy basis and the theta basis respectively, are constructed in \cite{LLZ} and \cite{GHKK}, which both contain coefficient free cluster monomials. It is known that each element in the above two bases satisfies the Laurent phenomenon and positivity, that is, its expression in every cluster is a Laurent polynomial over $\N\P$. So in some sense such element can be seen as a generalization of cluster monomials. Moreover, the constant coefficients of the Laurent expression  in the initial cluster are related to counting of some combinatorial objects. However, the greedy basis is only constructed for rank $2$ case, while the theta basis relies on the cluster scattering diagram. Another goal of this paper is to directly construct a basis of $\mathcal{U}(\A)$ consisting of some universally indecomposable Laurent polynomials as a generalization of cluster monomials in general case. In order to achieve it, one useful tool we will apply is the Nowton polytopes of $F$-polynomials associated to cluster variables.

In \cite{F}, Jiarui Fei defined the Newton polytope of an $F$-polynomial associated to representations of a finite-dimensional basic algebra, as well as showed some interesting combinatorial properties of such Newton polytopes. On the other hand, the authors of \cite{LLZ} and \cite{LLS} focused on Newton polytopes of cluster variables in  cluster algebras of rank $2$ and rank $3$ respectively. By definitions, up to a translation, the Newton polytope of a cluster variable can be obtained from that of the related $F$-polynomial by a transformation induced by its exchange matrix $B$ since $\hat{y}_{j;t}=\prod\limits_{i=1}^{m}x_{i;t}^{b_{ji}}$ in the case of geometric type. However, on the other hand, when $B$ is not invertible, it is not apparent to obtain the latter from the former. In this a aspect, it seems that the Nowton polytopes of $F$-polynomials keeps more information. So in this paper, {\bf we mainly focus on the Nowton polytopes of $F$-polynomials in initial $Y$-variables associated to cluster variables.}

Based on the study of the Newton polytope $N_{l;t}$ of $F_{l;t}$, we introduce the polytope $N_{h}$ associated to a vector $h\in\Z^n$ and the polytope functions $\rho_{h}$ as a generalization of $N_{l;t}$ and $x_{l;t}$ respectively. Then the properties of $N_{h}$ and $\rho_{h}$ naturally induce those of $N_{l;t}$ and $x_{l;t}$. Moreover, it will also be proved that the polytope functions compose a strongly positive basis of $\mathcal{U}(\A)$ for a cluster algebras with principal coefficients as well as certain cluster algebra over arbitrary semifield. As applications, several conjectures are confirmed for TSSS cluster algebras, including the positivity conjectures of cluster variables and of $d$-vectors respectively.

\subsection{Notions and notations on cluster algebras and Laurent polynomials}\quad

In this section, we recall some preliminaries of cluster algebras, $F$-polynomials and $d$-vectors mainly based on \cite{FZ4}.

We would like to introduce the following notations for convenience: for any $n\in\N$, $x\in\Z$,
\[[1,n]=\{1,2,\cdots,n\},\;\;\;\;\;\;\;
sgn(x)=\left\{\begin{array}{cc}
                  0 & x=0 \\
                  \frac{|x|}{x} & otherwise
                \end{array}\right.,\;\;\;\;\;\;\;
[x]_{+}=max\{x,0\}.\]
And for a vector $\alpha=(\alpha_{1},\cdots,\alpha_{r})\in\Z^{r}$, $[\alpha]_{+}=([\alpha_{1}]_{+},\cdots,[\alpha_{r}]_{+})$. We always represent the elements in $\R^{n}$ as row vectors unless otherwise specified.

An $n\times n$ integer matrix $B=(b_{ij})$ is called {\bf sign-skew-symmetric} if either $b_{ij}=b_{ji}=0$ or $b_{ij}b_{ji}<0$ for any $i,j\in[1,n]$. A \textbf{skew-symmetric} matrix is a sign-skew-symmetric matrix with $b_{ij}=-b_{ji}$ for any $i,j\in[1,n]$. Moreover, a \textbf{skew-symmetrizable} matrix is a sign-skew-symmetric matrix such that there is a positive diagonal integer matrix $D$ satisfying that $DB$ is skew-symmetric.

For a sign-skew-symmetric matrix $B$, we define another $n\times n$ matrix $B'=(b_{ij}')$ satisfying that for any $k, i,j\in[1,n]$,
  \begin{equation}\label{equation: mutation of B}
    b_{ij}'=\left\{\begin{array}{ll}
                             -b_{ij} & i=k\text{ or }j=k; \\
                             b_{ij}+sgn(b_{ik})[b_{ik}b_{kj}]_{+} & otherwise.
                           \end{array}\right.
  \end{equation}
We call the formula (\ref{equation: mutation of B}) {\bf the exchange relation for sign-skew-symmetric matrices}. Denote by $B'=\mu_k(B)$ the mutation of $B$ in direction $k$.

For $k_1, k_2\in[1,n]$, if $B'=\mu_{k_1}(B)$ is also sign-skew-symmetric, then we can mutate $B'$ in direction $k_2$ to obtain $B''=\mu_{k_2}\mu_{k_1}(B)$.

\begin{Definition}
  For a sign-skew-symmetric matrix $B$,  if $B^{(i)}=\mu_{k_i}\cdots\mu_{k_1}(B)$ are always sign-skew-symmetric for all $i\in[1,s]$ and any sequences of mutations $\mu_{k_1}, \cdots, \mu_{k_s}$, then $B$ is called a {\bf totally sign-skew-symmetric matrix}.
\end{Definition}

The notion of totally sign-skew-symmetric matrices was introduced in \cite{BFZ}. It is well-known that skew-symmetric and skew-symmetrizable matrices are totally sign-skew-symmetric matrices. An example of a $3\times 3$ sign-skew-symmetric matrix which is not skew-symmetrizable was given in \cite{BFZ}. In that paper, Berenstein etc. conjectured that acyclic sign-skew-symmetric matrices are always totally sign-skew-symmetric. In \cite{HL}, Ming Huang and Fang Li proved this conjecture.

Hence, on sign-skew-symmetric matrices, one of the most important remaining problems is the condition under which sign-skew-symmetric matrices are total. In this paper, we always assume the involved sign-skew-symmetric matrices are totally sign-skew-symmetric.

For convenience, we will denote a totally sign-skew-symmetric matrix (respectively, cluster algebra defined subsequently) briefly as a {\bf TSSS matrix} (respectively, {\bf TSSS cluster algebra}).

Let $(\P,\oplus,\cdot)$ be a semifield, i.e., a free abelian multiplicative group endowed with a binary operation of (auxiliary) addition $\oplus$ which is commutative, associative and distributive with respect to the multiplication in $\P$. And $\F$ is the field of rational functions in $n$ independent variables with coefficients in $\Q\P$.
\begin{Definition}
  A \textbf{seed} in $\F$ is a triple $\Sigma=(X,Y,B)$ such that
  \begin{itemize}
    \item $X=(x_{1},x_{2},\cdots,x_{n})$ is an $n$-tuple whose components form a free generating set of $\F$;
    \item $Y=(y_{1},y_{2},\cdots,y_{n})$ is an $n$-tuple of elements in $\P$;
    \item $B$ is an $n\times n$ totally sign-skew-symmetric integer matrix.
  \end{itemize}
\end{Definition}
$X$ defined above is called a \textbf{cluster} with \textbf{cluster variables} $x_{i}$, $y_{i}$ is called a \textbf{$Y$-variable} and $B$ is called an \textbf{exchange matrix}.

\begin{Definition}
  For any seed $\Sigma=(X,Y,B)$ in $\F$ and $k\in[1,n]$, $\Sigma\'=(X\',Y\',B\')$ is obtained from $\Sigma$ by \textbf{mutation} in direction $k$ if
  \begin{equation}\label{equation: mutation of x}
            x_{j}\'=\left\{\begin{array}{ll}
                             \frac{y_{k}\prod\limits_{i=1}^{n}x_{i}^{[b_{ik}]_{+}}+\prod\limits_{i=1}^{n}x_{i}^{[-b_{ik}]_{+}}}{(y_{k}\oplus 1)x_{k}} & \text{if}\; j=k;\\
                             x_{j} & \text{otherwise}.
                           \end{array}\right.
  \end{equation}
  \begin{equation}\label{equation: mutation of y}
    y_{j}\'=\left\{\begin{array}{ll}
                             y_{k}^{-1} & \text{if}\; j=k; \\
                             y_{j}y_{k}^{[b_{kj}]_{+}}(y_{k}\oplus 1)^{-b_{kj}} & \text{otherwise}.
                           \end{array}\right.
  \end{equation}
  and $B\'=\mu_{k}(B)$. In this case, we write $\Sigma\'=\mu_{k}(\Sigma)$.
\end{Definition}
It can be easily checked that $\Sigma\'$ is a seed and the seed mutation $\mu_{k}$ ia an involution.
\begin{Definition}
  Let $\T_{n}$ be the $n$-regular tree whose $n$ edges emanating from the same vertex are labeled bijectively by $[1,n]$. We assign a seed to each vertex of $\T_{n}$ such that if two vertices are connected by an edge labeled $k$, then the seeds assigned to them are obtained from each other by the mutation in direction $k$. This assignment is called a \textbf{cluster pattern}.
\end{Definition}
In this paper, the seed assigned to a vertex $t$ is denoted by $\Sigma_{t}=(X_{t},Y_{t},B_{t})$ with
\[X_{t}=(x_{1;t},x_{2;t}\cdots,x_{n;t}),\quad Y_{t}=(y_{1;t},y_{2;t}\cdots,y_{n;t})\quad and\quad B_{t}=(b_{ij}^{t})_{i,j\in[1,n]}.\]

Now we are ready to introduce the definition of cluster algebras.
\begin{Definition}
  Given a cluster pattern, let $\mathcal{S}=\{x_{i;t}\in\F\mid i\in[1,n],t\in\T_{n}\}$. The ({\bf totally sign-skew-symmetric}) \textbf{cluster algebra} $\mathcal{A}$ associated with the given cluster pattern is the $\Z\P$-subalgebra of $\F$ generated by $\mathcal{S}$.
\end{Definition}

If there is a skew-symmetrizable (respectively, skew-symmetric) exchange matrix in a cluster algebra $\A$, then all exchange matrices of $\A$ are skew-symmetrizable (respectively, skew-symmetric). So, in this case we call $\A$ a {\bf skew-symmetrizable} (respectively, {\bf skew-symmetric}) {\bf cluster algebra}.

In this paper, when saying a cluster algebra,  we always mean a TSSS cluster algebra.

It can be seen from the definition that the cluster algebra $\A$ is related to the choice of semifield $\P$. There are two special semifields which play important
roles.
\begin{Definition}
    (i)\; The universal semifield $\Q_{sf}(u_{1},u_{2},\cdots,u_{l})$ is the semifield of all rational functions which have subtraction-free rational expressions in independent variables $u_{1},u_{2},\cdots,u_{l}$, with usual multiplication and addition.

    (ii)\; The tropical semifield $Trop(u_{1},u_{2},\cdots,u_{l})$ is the free abelian multiplicative group generated by $u_{1},u_{2},\cdots,u_{l}$ with addition defined by
        $\prod\limits_{j=1}^{l}u_{j}^{a_{j}}\oplus\prod\limits_{j=1}^{l}u_{j}^{b_{j}}=\prod\limits_{j=1}^{l}u_{j}^{min(a_{j},b_{j})}.$
  \end{Definition}
In particular, we say a cluster algebra $\A$ is of \textbf{geometry type} if $\P$ is a tropical semifield. In this case, we can also denote $\P$ as $Trop(x_{n+1},x_{n+2},\cdots,x_{m})$. Then according to the definition, $y_{j;t}$ is a Laurent monomial of $x_{n+1},x_{n+2},\cdots,x_{m}$ for any $j\in[1,n],t\in\T_{n}$. Hence we can define $b_{ij}^{t}$ for $i\in[n,m],j\in[1,n]$ satisfying
\[y_{j;t}=\prod\limits_{i=n+1}^{m}x_{i}^{b_{ij}^{t}}.\]
Let $\tilde{B}_{t}$ be the $m\times n$ matrix $\tilde{B}_{t}=(b_{ij}^{t})_{i\in[1,m],j\in[1,n]}$ and $\tilde{X}_{t}=(x_{1;t},\cdots,x_{n;t},x_{n+1},\cdots,x_{m})$. Then the seed assigned to $t$ can be represented as $(\tilde{X}_{t},\tilde{B}_{t})$. The mutation formulas are the same for $\tilde{B}$ while those of $\tilde{X}$ at direction $k$ become
\[x_{j}'=\left\{\begin{array}{ll}
                             \frac{\prod\limits_{i=1}^{m}x_{i}^{[b_{ik}]_{+}}+\prod\limits_{i=1}^{m}x_{i}^{[-b_{ik}]_{+}}}{x_{k}} & j=k;\\
                             x_{j} &  otherwise.
                           \end{array}\right.\]

\begin{Definition}
  A cluster algebra is said to have {\bf  principal coefficients} at a vertex $t_{0}$ if $\P=Trop(y_{1},y_{2},\cdots,y_{n})$ and $Y_{t_{0}}=(y_{1},y_{2},\cdots,y_{n})$.
\end{Definition}
Hence a cluster algebra having principal coefficients at some vertex is of geometric type. Then if we use $(\tilde{X},\tilde{B})$ to represent a seed, the definition is equivalent to that there is a seed $(\tilde{X}_{t_{0}},\tilde{B}_{t_{0}})$ at vertex $t_{0}$ satisfying $\tilde{B}_{t_{0}}=\left(\begin{array}{c}
                                                                                                                                                  B_{t_{0}} \\
                                                                                                                                                  I
                                                                                                                                               \end{array}\right),$
where $I$ is a $n\times n$ identity matrix.

Given a cluster algebra $\A$ with initial seed $\Sigma_{t_{0}}=(X_{t_{0}},Y_{t_{0}},B_{t_{0}})$, we denote by $\A_{prin}$ the cluster algebra with principal coefficients associated to $B_{t_{0}}$, which is called the {\bf principal coefficients cluster algebra corresponding to} $\A$ since it is unique up to cluster isomorphisms.

The \textbf{Laurent phenomenon}, given in \cite{FZ1, FZ2}, is the most fundamental result in cluster theory, which says that for a cluster algebra $\A$ and its fixed seed $(X,Y,B)$, every cluster variable of $\A$ is a Laurent polynomial over $\Z\P$ in cluster variables in $X$.

Thus for a seed $(X_{t_{0}},Y_{t_{0}},B_{t_{0}})$ and any cluster variable $x_{l;t}$ in $\A$, we can express it as a Laurent polynomial in cluster $X_{t_{0}}$:
\[x_{l;t}=\frac{P_{l;t}^{t_{0}}}{\prod\limits_{i=1}^{n}x_{i;t_{0}}^{d_{i}^{t_{0}}(x_{l;t})}}.\]
such that $P_{l;t}^{t_{0}}$ is a polynomial with no non-constant monomial factor. Here and in the following, we call $P_{l;t}^{t_{0}}$ the {\bf absolute numerator} of $x_{l;t}$ with respect to $X_{t_{0}}$. The denominator vector $d_{l;t}^{t_{0}}=(d_{1}^{t_{0}}(x_{l;t}),d_{2}^{t_{0}}(x_{l;t}),\cdots,d_{n}^{t_{0}}(x_{l;t}))$ is called the \textbf{$d$-vector} of $x_{l;t}$ with respect to the cluster $X_{t_{0}}$. Moreover, if $\A$ has principal coefficients at $t_{0}$, then $P_{l;t}^{t_{0}}$ belongs to $\Z[x_{1,t_{0}},\cdots,x_{n,t_{0}};y_{1,t_{0}},\cdots,y_{n,t_{0}}]$. $F_{l;t}^{t_{0}}=P_{l;t}^{t_{0}}|_{x_{i;t_{0}}\rightarrow1,\forall i\in[1,n]}$ is a polynomial in $y_{1,t_{0}},\cdots,y_{n,t_{0}}$ called the \textbf{$F$-polynomial} of $x_{l;t}$ with respect to $X_{t_{0}}$. Under the canonical $\Z^{n}$-grading given by $deg(x_{i;t_0})=e_{i},deg(y_{i;t_0})=-(b_{i}^{t_{0}})^{\top}$ for any $i\in[1,n]$ where $e_{1},\cdots,e_{n}$ are standard basis in $\Z^{n}$, and $b_{i}^{t_{0}}$ is the $i$-th column of $B_{t_{0}}$, the Laurent expression of $x_{l;t}$ in $X_{t_{0}}$ is homogeneous with degree $g_{l;t}^{t_{0}}$, which is called the \textbf{$g$-vector}\footnote{Usually, $d$-vectors and $g$-vectors are written as column vectors. But in this paper, because we will write the coordinates specifically in some discussion, it is more convenient for us to use row vectors. So for the sake of consistency we always write vectors in $\R^{n}$, including $d$-vectors and $g$-vectors, as row vectors.} of $x_{l;t}$ corresponding to $X_{t_{0}}$. Or generally in a cluster algebra of geometric type, $g$-vectors can also be defined recurrently as follows: $g_{j;t_{0}}^{t_{0}}=e_{j}$, and
\[g_{j;t\'}^{t_{0}}=\left\{\begin{array}{ll}
                     -g_{k;t}^{t_{0}}+\sum\limits_{i=1}^{n}[b_{ik}^{t}]_{+}g_{i;t}^{t_{0}}-\sum\limits_{i=1}^{n}[b_{n+i;k}^{t}]_{+}(b_{i}^{t_{0}})^{\top} & \text{if }j=k; \\
                     g_{j;t}^{t_{0}} & otherwise.
                   \end{array}\right.\]
when $t$ and $t\'$ are connected by an edge labeled $k$ in $\T_{n}$.

Next Theorem shows the importance of principal coefficients case in the study of cluster algebras.
\begin{Theorem}\label{expression of a cluster variable}
  \cite{FZ4}For any cluster algebra $\A$ and any vertices $t$ and $t\'$ in $\T_{n}$, the cluster variable $x_{l;t}$ can be expressed as
  \begin{equation*}
    x_{l;t}=\frac{F_{l;t}^{t\'}|_{\F}(\hat{y}_{1;t\'},\cdots,\hat{y}_{n;t\'})}{F_{l;t}^{t\'}|_{\P}(y_{1;t\'},\cdots,y_{n;t\'})}\prod\limits_{i=1}^{n}x_{i;t\'}^{g_{i}},
  \end{equation*}
  where
 \begin{equation}\label{hatandg}
  \hat{y}_{j;t\'}=y_{j;t\'}\prod\limits_{i=1}^{n}x_{i;t\'}^{b\'_{ij}}\;\;\;\; and\;\;\;\; g_{l;t}^{t\'}=(g_{1},\cdots,g_{n}).
  \end{equation}
\end{Theorem}

We denote $\hat{Y}_{t}=\{\hat{y}_{1;t},\cdots,\hat{y}_{n;t}\}$ for any $t\in\T_{n}$.

When $\A$ is a cluster algebra of geometric type with initial seed $\Sigma_{t_{0}}=(\tilde{X}_{t_{0}},\tilde{B}_{t_{0}})$, we also denote its seed as $\Sigma_{t}=(X_{t},X^{fr}_{t},\tilde{B}_{t})$ for any $t\in\T_{n}$, where $X_{t}=\{x_{1;t},\cdots,x_{n;t}\}$ and $X^{fr}_{t}=\{x_{n+1},\cdots,x_{m}\}$ to distinguish two kinds of variables.

%

\begin{Definition}
  For any seed $\Sigma_{t}$ associated to $t\in\T_{n}$, we denote by $\mathcal{U}(\Sigma_{t})$ the $\Z\P$-subalgebra of $\F$ given by
  \begin{equation*}
    \mathcal{U}(\Sigma_{t})=\Z\P[X_{t}^{\pm 1}]\cap\Z\P[X_{t_{1}}^{\pm 1}]\cap\cdots\cap\Z\P[X_{t_{n}}^{\pm 1}],
  \end{equation*}
  where $t_{i}$ is the vertex connected to $t$ by an edge labeled $i$ in $\T_{n}$ for any $i\in[1,n]$. $\mathcal{U}(\Sigma_{t})$ is called the {\bf upper bound} associated with the seed $\Sigma_{t}$. And $\mathcal{U}(\A)=\bigcap\limits_{t\in\T_{n}}\mathcal{U}(\Sigma_{t})$ is called the \textbf{upper cluster algebra} associated to $\A$.
\end{Definition}

In this paper, we will construct $\rho_{h}\in\N[[\hat{Y}]]X^{h}$ for each $h\in\Z^{n}$, that is, $\rho_{h}$ is of the form $\sum\limits_{p\in\N^{n}}a_{p}\hat{Y}^{p}X^{h}$ and $\rho_{h}|_{x_{i}\rightarrow1,\forall i\in[1,n]}$ is a power series of $Y$, where $a_p\in\N$. In order to emphasis that we regard $X$ as variables while $Y$ as coefficients, we will slightly abuse the notation to denote $\rho_{h}\in\N[Y][[X^{\pm1}]]$ and call it a {\bf formal Laurent polynomial} in $X$ with coefficients in $\N[Y]$.

For any $t,t\'\in\T_{n}$ connected by an edge labeled $k$ and any homogeneous Laurent polynomial
$$f\in\Z[Y_{t\'}][X_{t\'}^{\pm 1}]\cap\Z[Y_{t}][X_{t}^{\pm 1}]\subseteq\Z Trop(Y_{t\'})[X_{t\'}^{\pm 1}]$$
with grading $h$, we naturally have $f=F|_{\F}(\hat{Y}_{t\'})X_{t'}^{h}$, where $F$ is obtained from the Laurent expression of $f$ in $X_{t\'}$ by specilizing $x_{i;t\'}$ to 1 for any $i\in[1,n]$. We modify it into the Laurent polynomial
\begin{equation}\label{modify}
\frac{F|_{\F}(\hat{Y}_{t\'})}{y_{k;t\'}^{[h_{k}]_{+}}}X_{t\'}^{h},
\end{equation}
and
\begin{equation}\label{L(f)}
 \text{\em denote by}\; L^{t}(f)\; \text{\em the Laurent expression of this modified form (\ref{modify}) in}\; X_{t}
\end{equation}
with coefficients in $Y_{t}$ belonging to $\Z Trop(Y_{t})$ (So $\oplus$ is now defined in $\Z Trop(Y_{t})$ rather than in $\Z Trop(Y_{t\'})$. Hence $1\oplus y_{k;t}$ now equals $1$ rather than $y_{k;t}$). The motivation of such definition is as follows.

Applying $L^{t}$ on $f$ changes the semifield from $Trop(Y_{t\'})$ to $Trop(Y_{t})$, and compensating for it with dividing $y_{k;t\'}^{[h_{k}]_{+}}$ as in (\ref{modify}),  we will show in Remark \ref{remark after the theorem} that $L^{t}(\rho_{h}^{t\'})$ equals the Laurent expression of $\frac{F_{h}^{t\'}|_{\F}(\hat{Y}_{t\'})}{F_{h}^{t\'}|_{Trop(Y_{t})}(Y_{t\'})}X_{t\'}^{h}$ in $X_{t}$ with coefficients in $Y_{t}$ belonging to $\Z Trop(Y_{t})$, which realizes $\rho_{h}^{t\'}$ and $L^{t}$ as a generalization of a coefficient free cluster monomial and its mutation respectively in some aspect and thus justifies
the compensation. On the other hand, later we will use $L^{t}$ to introduce a strong restriction on the support of homogeneous Laurent polynomial (see $\mathcal{U}^{+}(\Sigma_{t})$ or $\widehat{\mathcal{U}}^{+}(\Sigma_{t})$) so as to make $\rho_{h}$ which we are going to construct a special element under this restriction.

Then we can define $L^{t;\gamma}(f)=L^{t}\circ L^{t^{(1)}}\circ\cdots\circ L^{t^{(s)}}(f)$ for any path $\gamma=t-t^{(1)}-\cdots-t^{(s)}-t\'$ in $\T_{n}$ if $L^{t}\circ L^{t^{(1)}}\circ\cdots\circ L^{t^{(j)}}(f)\in\Z[Y_{t^{(j)}}][X_{t^{(j)}}^{\pm 1}]$ for $j\in[0,s]$.

Later we will show that $L^{t;\gamma}(f)$ only depends on the endpoints $t$ and $t\'$ in Remark \ref{remark after the theorem}, so we usually omit the path in the superscript. Naturally, $L^t$ also makes sense for formal Laurent polynomials.

For any $t\in\T_{n}$ denote
\[\mathcal{U}_{\geqslant0}(\Sigma_{t})=\N\P[X_{t}^{\pm 1}]\cap\N\P[X_{t_{1}}^{\pm 1}]\cap\cdots\cap\N\P[X_{t_{n}}^{\pm 1}]\subseteq\mathcal{U}(\Sigma_{t}),\]
\[\mathcal{U}^{+}(\Sigma_{t})=\{f\in \mathcal{U}_{\geqslant0}(\Sigma_{t})\;|\; L^{t}(f)\in\N[Y_{t}][X_{t}^{\pm 1}]\text{ and }L^{t_{i}}(f)\in\N[Y_{t_{i}}][X_{t_{i}}^{\pm 1}],\forall i\in[1,n]\}\]
and
$$\mathcal{U}^{+}_{\geqslant0}(\Sigma_{t})= \mathcal{U}^{+}(\Sigma_{t})\cap\mathcal{U}_{\geqslant0}(\Sigma_{t_{1}})\cap\cdots\cap\mathcal{U}_{\geqslant0}(\Sigma_{t_{n}}),$$
where $t_{i}\in\T_{n}$ is the vertex connected to $t$ by an edge labeled $i$. We say an element in $\mathcal{U}^{+}_{\geqslant0}(\Sigma_{t})$ to be {\bf indecomposable} if it can not be written as a sum of two nonzero elements in $\mathcal{U}^{+}_{\geqslant0}(\Sigma_{t})$.

And we add the hat ``\;$\widehat\quad\;$'' to represent their completion. That is, denote
\[\widehat{\mathcal{U}}_{\geqslant0}(\Sigma_{t})=\N\P[[X_{t}^{\pm 1}]]\cap\N\P[[X_{t_{1}}^{\pm 1}]]\cap\cdots\cap\N\P[[X_{t_{n}}^{\pm 1}]],\]
\[\widehat{\mathcal{U}}^{+}(\Sigma_{t})=\{f\in\widehat{\mathcal{U}}_{\geqslant0}(\Sigma_{t})|L^{t}(f)\in\N[Y_{t}][[X_{t}^{\pm 1}]]\text{ and }L^{t_{i}}(f)\in\N[Y_{t_{i}}][[X_{t_{i}}^{\pm 1}]],\forall i\in[1,n]\}\]
and
$$\widehat{\mathcal{U}}^{+}_{\geqslant0}(\Sigma_{t})= \widehat{\mathcal{U}}^{+}(\Sigma_{t})\cap\widehat{\mathcal{U}}_{\geqslant0}(\Sigma_{t_{1}})\cap\cdots\cap\widehat{\mathcal{U}}_{\geqslant0}(\Sigma_{t_{n}}),$$
where $t_{i}\in\T_{n}$ is the vertex connected to $t$ by an edge labeled $i$.

In the sequel, for a cluster algebra $\A$, we will always {\em denote by $t_0$ the vertex of the initial seed} unless otherwise specified. And when a vertex is not written explicitly, we always mean the initial vertex $t_0$. For example, we use $X$, $x_{i}$, $P_{l;t}$ to denote $X_{t_{0}}$, $x_{i;t_{0}}$, $P_{l;t}^{t_{0}}$ respectively. For any cluster $X_{t}$ and any vector $\alpha=(\alpha_{1},\cdots,\alpha_{n})\in \Z^n$, we denote $X_{t}^{\alpha}=\prod\limits_{i=1}^{n}x_{i;t}^{\alpha_{i}}$.
\vspace{2mm}

Let $\A$ be a cluster algebra over $\P$ with initial seed $(X,Y,B)$.

For any $k\in[1,n]$, $t\in\T_{n}$ and Laurent polynomial $P\in\Z\P[X_t^{\pm 1}]$, we will always denote by $$M_{k;t}=x_{k;t}\mu_{k}(x_{k;t})$$ the \textbf{exchange binomial} in direction $k$ at $t$ and by $deg_{x_{k;t}}(P)$ the $x_{k;t}$-\textbf{degree} of $P$. Trivially, $M_{k;t}$ is a polynomial in $\Z\P[x_{1;t},\cdots,x_{k-1;t},x_{k+1;t},\cdots,x_{n;t}]$.

A \textbf{cluster monomial} in $\A$ is a monomial in a cluster $X_{t}$ for some $t\in\T_{n}$. In the following, when we mention a cluster monomial, it is of the form $aY_{t}^{p}X_{t}^{q}$ with $a\in\Z,p,q\in\Z^n$ and $t\in\T_{n}$. For such a cluster monomial $f=aY^{p}X^{q}$, we call $a$ (respectively, $aY^{p}$) the {\bf constant coefficient} (respectively, {\bf coefficient}) of $f$ and say $f$ is {\bf constant coefficient free} (respectively, {\bf coefficient free}) if $a=1$ (respectively,$aY^{p}=1$). Similarly, we define a \textbf{cluster polynomial} to be a polynomial in a cluster $X_{t}$.

\begin{Definition}
  (i)\;  For a Laurent polynomial $P\in\Z\P[X^{\pm 1}]$ and a constant coefficient free Laurent monomial $p=Y^{\alpha}X^{\beta}$ with $\alpha,\beta\in\Z^{n}$, we denote by $co_{p}(P)$ the constant coefficient of $p$ in $P$.

  (ii)\;  For any Laurent polynomial $P$, $P\'$ is called a \textbf{summand} of $P$ if for any Laurent monomial $p$ with constant coefficient $1$, either $0\leqslant co_{p}(P\')\leqslant co_{p}(P)$ or $co_{p}(P)\leqslant co_{p}(P\')\leqslant 0$. $P\'$ is called a \textbf{monomial summand} of $P$ if it is moreover a Laurent monomial.
\end{Definition}

For a variable $x$, we say a polynomial $P$ is  {\bf $x$-homogeneous} if $deg_{x}(p)$ are the same for all monomial summands $p$ of $P$.

\begin{Definition}
  For any $t\in\T_{n}$, $k\in[1,n]$ and $x_{k;t}$-homogeneous polynomial $P$ in $X_{t}$ with exchange binomial $M_{k;t}$, denote by $\widetilde{deg}_{k}^{t}(P):=deg_{x_{k;t}}(P)+max\{s\in\N:\; M_{k;t}^{s}|P\text{ in }\Z\P[X_{t}^{\pm 1}]\}$ the \textbf{general degree} of $P$ in $x_{k;t}$.   Moreover, for any polynomial $P=\sum\limits_{i}P_{i}$, where $P_{i}$ is a $x_{k;t}$-homogeneous polynomial in $X_{t}$ satisfying $deg_{x_k}(P_i)\neq deg_{x_k}(P_j)$ when $i\neq j$, define $\widetilde{deg}_{k}^{t}(P):=\min\limits_{i}\{\widetilde{deg}_{k}^{t}(P_{i})\}$.
\end{Definition}
According to the mutation formula (\ref{equation: mutation of x}), $\widetilde{deg}_{k}^{t}(P)$ is the maximal integer $a$ such that $\frac{P}{x_{k;t}^{a}}$ can be expressed as a Laurent polynomial in $X_{t_{k}}$, where $t_{k}\in\T_{n}$ is the vertex connected to $t$ by an edge labeled $k$.

Following the definitions in \cite{LLZ}, a Laurent polynomial $p$ in $X$ is called \textbf{universally positive} if $p\in\N\P[X_{t}^{\pm 1}]$ for any $t\in\T_{n}$. And a universally positive Laurent  polynomial is said to be \textbf{universally indecomposable} if it cannot be expressed as a sum of two nonzero universally positive Laurent  polynomials. Universal indecomposability can be regarded as the ``minimalism'' in the set of universally positive Laurent polynomials. Since the above two definitions are given for all $t\in T_{n}$, they are naturally mutation invariants.

A $\Z\P$-basis $\{\alpha_{s}\}_{s\in I}$ of $\mathcal{U(A)}$ is called \textbf{strongly positive} if for any $i,j\in I$, $\alpha_{i}\alpha_{j}=\sum\limits_{s\in I}a_{ij}^{s}\alpha_{s}$, where $a_{ij}^{s}\in\N\P$ for any $s\in I$.

\subsection{Notions and notations on polytopes}\quad

Next we briefly introduce some concepts and notations about polytopes mainly from \cite{Z}, which will be used in this paper with slight modification.

In this paper, unless otherwise specified, we always fix the following notations and notions: \\
(i) Denote by $z_{1},\cdots,z_{n}$ the coordinates of $\R^{n}$;
\\
(ii) {\bf Points} imply lattice points in $\Z^{n}$;
\\
(iii) {\bf Polytopes} imply those whose vertices are lattice points;
\\
(iv)  The {\bf partial order} ``$\leqslant$" in $\Z^{n}$ is defined as $a\leqslant b$ for any $a=(a_1, \cdots, a_n), b=(b_1, \cdots, b_n)\in\Z^{n}$ if $a_{i}\leqslant b_{i}$ for all $i\in[1,n]$.

\begin{Definition}\label{convex}
  (i)\;The \textbf{convex hull} of a finite set $V=\{\alpha_{1},\cdots,\alpha_{r}\}\subseteq \R^{n}$ is
  \[conv(V)=\{\sum\limits_{i=1}^{r}a_{i}\alpha_{i}\; |\; a_{i}\geqslant 0,\sum\limits_{i=1}^{r}a_{i}=1\},\]
  while the \textbf{affine hull} of $V$ is
  \[\text{\bf aff}(V)=\{\sum\limits_{i=1}^{r}a_{i}\alpha_{i}\; |\; a_i\in\R, \sum\limits_{i=1}^{r}a_{i}=1\}.\]

  (ii)\;An (unweighted) \textbf{polytope} is the convex hull of a certain finite set of points in  $\R^{n}$ for some $n\in \N$, or equivalently, a polytope is a bounded intersection of finitely many closed halfspaces in $\R^{n}$ for $n\in\N$. The \textbf{dimension of a polytope} is the dimension of its affine hull.

  (iii)\;Let $N\subseteq \R^{n}$ be a polytope. For some chosen $w\in\R^{n}$ and $c\in\R$, a linear inequality $wp^{\top}\leqslant c$ is called {\bf valid} for $N$ if it is satisfied for all points $p\in N$. A \textbf{face} of $N$ is a set of the form
  \begin{equation}\label{face}
  S = N\cap \{p\in\R^{n}\; |\; wp^{\top} = c\},
  \end{equation}
  where $wp^{\top}\leqslant c$ is a valid inequality for $N$. The {\bf dimension of a face} is the dimension of its affine hull.

  (iv)\;The \textbf{vertices}, \textbf{edges} and \textbf{facets} of a polytope $N$ are its faces with dimension $0$, $1$, and $dim N - 1$ respectively. Denote by $V(N)$ and $E(N)$ the set consisting of vertices and edges of $N$ respectively.

  (v)\;The {\bf sum} $N+N\'$ of two polytopes $N$ and $N\'$ is the convex hull of $N\cup N\'$.

  (vi)\;The {\bf Minkowski sum} $N\oplus N\'$ of two polytopes $N$ and $N\'$ is the polytope consisting of all points $p+q$ for points $p\in N$ and $q\in N\'$.
\end{Definition}

Here we modify the original definition of polytopes by associating weight to each lattice point in it.
\begin{Definition}\label{weight}
  For a polytope $N$, the {\bf weight} of a point $p\in N$ is the integer placed on this point, denoted as $co_p(N)$, or simply $co_{p}$ when the polytope $N$ is known clearly. A polytope $N$ equipped with weights is called a {\bf weighted polytope} if $N=conv(supp(N))$, where $supp(N)=\{p\in N|co_{p}(N)\neq0\}$ is called the support of $N$.
\end{Definition}

Two summations introduced in Definition \ref{convex} (v) and (vi) can be extended to polytopes with weights. For two polytopes $N$ and $N\'$, we define the weights of $N+N\'$ as follows:
\begin{equation}\label{cosum}
co_p(N+N\')=co_p(N)+co_{p}(N\')\triangleq\left\{\begin{array}{lr}
                        co_{p}(N)+co_{p}(N\')  & \text{if}\; p\in N\cap N\'; \\
                         co_p(N) & \text{if}\; p\in N\backslash N\';\\
                         co_p(N\') & \text{if}\; p\in N\'\backslash N;\\
                         0 & \text{if}\; p\not\in N\cup N\'.
                       \end{array}\right.
\end{equation}
Then $N+N\'$ is the polytope $conv(supp(N+N\'))$ equipped with the above weights. This summation is induced from that of Laurent polynomials with respect to the correspondence between polytopes and Laurent polynomials. Note that in general, $conv(supp(N+N\'))$ does not equal to $conv(N\cup N\')$. While for the Minkowski sum $\oplus$, let
\begin{equation}\label{minsum}
co_{q}(N\oplus N\')=\sum\limits_{p+p\'=q}co_{p}(N)co_{p\'}(N\')
\end{equation}
for any point $q\in N\oplus N\'$. This is induced from the multiplication of Laurent polynomials. It can easily verified that both summations are commutative and associative.

\begin{Example}
  Let $p=(0,0)$, and $q=(1,0)$ in $\R^2$. Define polytopes $N$, $N'$ and $N''$ to be the segment connecting $p$ and $q$ equipped with weights
  $$co_{p}(N)=co_{p}(N')=co_{q}(N)=1,\quad co_{p}(N'')=2,\quad co_{q}(N\')=-1 \text{ and }co_{q}(N'')=0.$$
  Then by the definition of weighted polytopes, it can be check directly $N$ and $N'$ are weighted polytopes but $N''$ is not. Instead the weighted polytope $N+N'$ should be the single point $p$ equipped with weight 2. Since $supp(N+N\')=\{p\}$ while $N\cup N'$ equals he segment connecting $p$ and $q$, $conv(supp(N+N\'))\neq conv(N\cup N\')$.
\end{Example}

For convenience we assume that for a polytope $N$,
\begin{equation*}
  co_{p}(N)=0\;\;\; \text{if}\;\; p\notin N.
\end{equation*}

In the sequel, since all polytopes concerned about are weighted, we will omit the word ``weighted'' and simply say as polytopes. Also, the sum $N+N'$ and the Minkowski sum $N\oplus N'$ always mean those of weighed polytopes according to the formulae (\ref{cosum}) and (\ref{minsum}) respectively.

Under the meaning of the sum $``+"$, we define the subtraction $N-N\'$ to be the polytope $N^{\prime\prime}$ such that $N=N\'+N^{\prime\prime}$. It is not hard to see such $N^{\prime\prime}$ is well-defined and unique.

For any polytope $N$ in $\R^{n}$ and $w\in\Z^{n}$, we denote by $N[w]$ the polytope obtained from $N$ by a translation along $w$.

\begin{Definition}
  For two polytopes $N$ and $N\'$ in $\R^{n}$, $N\'$ is a {\bf sub-polytope} of $N$ if there is $w\in\Z^{n}$ such that $0\leqslant co_{p}(N\'[w])\leqslant co_{p}(N)$ or $0\geqslant co_{p}(N\'[w])\geqslant co_{p}(N)$ for any point $p\in N$, which is equivalent to that the $Y$-polynomial corresponding to $N\'$ is a summand of that corresponding to $N$ up to multiplying a Laurent monomial in $Y$. In this case, denote $N\'\leqslant N$. This relation $\leqslant$ defines a partial order in the set of polytopes.
\end{Definition}

In the above definition, we always have $dim(N\')\leqslant dim(N)$ for a sub-polytope $N\'$ of $N$, where the strict inequality may hold.

It is easy to see when the weights are all non-negative for polytopes $N$ and $N\'$, they are both sub-polytopes of $N+N\'$ and $N\oplus N\'$.

Given a polytope $N$ with a vertex $p$, a \textbf{lattice generating set} of $N$ based on $p$ is a minimal set of vectors $V=\{v_1,\cdots,v_r\}$ satisfying that for any $i\in[1,n]$, there are two points $q_i$ and $q_i\'$ with non-zero weights on an edge of $N$ such that $q_i\'-q_i=v_i$ and for any point $p\'\in N$ with non-zero weight, there is unique $a_i\in\N$ such that
\[p\'=p+\sum\limits_{i=1}^{r}a_iv_i.\]
Denote
$$ldim(N)=\min\{|V|\mid V\text{ is a lattice generating set of }N\}.$$
Apparently, $ldim(N)\geqslant dim(N)$.

Given a polytope $N$ with a point $v$ in it and a sequence $i_{1},\cdots,i_{r}\in[1,n]$, define $\{i_{1},\cdots,i_{r}\}$-\textbf{section} at $v$ of $N$ to be the convex hull of lattice points in $N$ whose $j$-th coordinates are equal to that of $v$ respectively for any $j\in[1,n]\setminus\{i_{1},\cdots,i_{r}\}$.

For any two points $p, q$, denote by $l(\overline{pq})$ the length of the segment $\overline{pq}$ connecting $p$ and $q$.

\begin{Definition}
(i)\; A map $\eta:\R^r\rightarrow\R^s$ is called \textbf{affine} if $\eta(p)=Ap+w$ for any $p\in\R^r$, where $A\in Mat_{s\times r}(\R)$ and $w\in\R^s$.

(ii)\; A \textbf{projection} $\tau$ of two polytopes $N$ and $N\'$ is a restriction of an affine map $\eta:\R^{dim(N)}\rightarrow\R^{dim(N\')}$ satisfying that $\tau(N)=N\'$ and the weights associated to $p$ in $N$ and to $\tau(p)$ in $N\'$ are the same for any (not necessary lattice) points $p\in N$, where the weights of non-lattice points are set to be zero. A projection is an {\bf isomorphism} if it is bijective.
\end{Definition}

An {\bf isomorphism $\tau$ of two (weighted) polytopes} $N$ and $N\'$ is a bijection of two sets:
\[\tau:\quad\{\text{(not necessary lattice) points in } N\}\quad  \longrightarrow  \quad\{\text{(not necessary lattice) points in } N\'\}\]
satisfying that $\tau(ap+bq)=a\tau(p)+b\tau(q)$ for any (not necessary lattice) points $p,q\in N$ and any $a,b\in\R_{\geqslant 0}$ with $a+b=1$, and the weights associated to $p$ and $\tau(p)$ respectively are the same, where the weights of non-lattice points are set to be zero.

Assume the dimensions of $N$ and $N\'$ are $r$ and $s$ respectively. Then they can be embedded in $\R^r$ and $\R^s$ respectively. For each $i\in[1,r]$, there are two different (not necessary lattice) points $p,p\'\in N$ such that $p-p\'=l(\overline{pp\'})e_{i}$. Then a linear map $\tilde\tau$ is induced by a projection $\tau: N\rightarrow N\'$ as
\begin{equation}\label{tildetau}
  \begin{array}{ccc}
    \tilde{\tau}:\quad\R^{r}\quad&\longrightarrow&\quad\R^{s} \\
    \qquad\qquad e_{i}\qquad & \mapsto & \qquad \frac{\tau(p)-\tau(p\')}{l(\overline{pp\'})}
  \end{array}
\end{equation}
In the later discussion, $N\'$ is often a face of some polytope with higher dimension $n$, so we usually slightly abuse the notation to use $\tilde{\tau}$ as the linear map:
$$\tilde{\tau}: \; \R^{r}\quad\longrightarrow\quad\R^{s}\quad\hookrightarrow\quad\R^{n}.$$
A polytope projection is called \textbf{non-negative} if $\tilde{\tau}(e_i)$ is a non-negative vector for any $i$.

Following Definition \ref{weight}, we can obtain the correspondence from polytopes to Laurent polynomials in the following way:

(i)\; To a Laurent monomial $a_{v}Y^{v}$ in $y_{1},\cdots,y_{n}$, where $a_{v}\neq 0\in\Z,v\in\Z^{n}$, we associate a vector $v$ together with integer $a_v$. Hence a Laurent polynomial $f(Y)=\sum\limits_{v\in\Z^{n}}a_{v}Y^{v}$ with $a_{v}\neq 0\in\Z$ corresponds to a set consisting of vectors $v$, which is called the \textbf{support} of $f(Y)$, together with integers $a_{v}\neq 0$.

(ii)\; Each integer vector $v$ of dimension $n$ corresponds to a lattice point in $\R^{n}$. Denote by $N$ the convex hull of lattice points corresponding to
the above vectors $v$ from $f(Y)$ with integers $a_{v}$ placed at lattice points. Then, we set up the following bijection:

$$\upsilon:\;\; \{\text{Laurent polynomials}\; f(Y) \}\longleftrightarrow \{\text{weighted polytopes}\; N\}$$
\\
In particular, polynomials in $Y$ correspond to polytopes lying in the non-negative quadrant.

For a principal coefficients cluster algebra of rank $n$ with initial cluster $X$, when a vector $h\in\Z^{n}$ is given, the above bijection $\upsilon$ induces a bijection $\tilde{\upsilon}$ from homogeneous Laurent polynomials $f(\hat{Y})X^{h}$ of degree $h$ to the weighted polytopes $N$ which corresponds to $f(Y)=f(\hat{Y})X^{h}|_{x_{i}\rightarrow 1,\forall i}$. That is,
\begin{equation}\label{polypoly}
  \tilde{\upsilon}(f(\hat{Y})X^{h})=\upsilon(f(Y))=N.
\end{equation}

In the sequel, we will call $N$ {\bf the Newton polytope} corresponding to the Laurent polynomial $f(Y)$ or to the Laurent polynomial $f(\hat{Y})X^{h}$ (with respect to $h$). For convenience, when discussing a polytope $N$, we sometimes use point $p$ to represent its corresponding Laurent monomial $\hat{Y}^pX^h$.

Denote by $N_{l;t}$ the Newton polytope of $F_{l;t}$ associated to the cluster variable $x_{l;t}$ for any $l\in[1,n]$ and $t\in\T_n$, that is,
$$\tilde{\upsilon}(x_{l;t})=\upsilon(F_{l;t})=N_{l;t}.$$

The support of a Laurnt polynomial $f(Y)$ is called {\bf saturated} if any lattice point in the Newton polytope $N$ corresponds to a nonzero monomial summand of $f(Y)$, i.e., if the weight of any lattice point in $N$ is nonzero.

Because of this correspondence, we can deal with weighted polytopes when discussing Laurent polynomials, which contains all elements in a cluster algebra according to Laurent phenomenon. In this paper, {\em our strategy is to  reduce problems into a simpler case by dividing polytopes into a sum of sub-polytopes inductively}.

\subsection{Main contents}\quad

The paper is organized as follows.

In Section 2, we introduce some results in cluster algebras and take primary discussion.

In Section 3, we first construct $N_{h}$ and $\rho_{h}$ for every $h\in\Z^{2}$ and show that $\{\rho_{h}|h\in\Z^{2}\}$ coincides with the greedy basis (Proposition \ref{polytope basis equals greedy basis}). Then we furthermore define essential skeleton in any rank as the generalization of that in the case of rank 2 and provide a program to construct $N_{h}$ as well as $\rho_{h}$ for any integer vector $h\in\Z^n$ (Construction \ref{construction}). After that, we have the following theorem.
\vspace{2mm}\\
$\spadesuit$ \textbf{(Theorem \ref{properties in case tsss})}\;
  Let $\A$ be a TSSS cluster algebra having principal coefficients and $h\in\Z^{n}$. Then,

  (i)\;For any $i\in[1,n]$, there is a decomposition
  \[x_{i}\rho_{h}=\sum\limits_{w,\alpha} c_{w,\alpha}Y^{w}\rho_{\alpha},\]
  where $w\in\N^{n},\alpha\in\Z^{n}$ and $c_{w,\alpha}\in\N$.

  (ii)\;The polytope function $\rho_{h}$ is the unique indecomposable formal Laurent polynomial in $X$ in $\widehat{\mathcal{U}}_{\geqslant0}(\Sigma_{t_{0}})$ which has $X^{h}$ as a summand and whose support is contained in $supp(N_{h})$.

  (iii)\; For any $h\in\Z^{n}$ and any $k\in[1,n]$, there is
  \[h^{t_{k}}=h-2h_{k}e_{k}+h_{k}[(b_{k})^{\top}]_{+}+[-h_{k}]_{+}(b_{k})^{\top}\]
  such that $L^{t_{k}}(\rho_{h})=\rho^{t_{k}}_{h^{t_{k}}}$, where $t_{k}\in\T_{n}$ is the vertex connected to $t_{0}$ by an edge labeled $k$ and $h_{k}$ is the $k$-th entry of $h$.

  (iv)\;For any $p,p\'\in N_{h}$, if the segment $l$ connecting $p$ and $p\'$ is parallel to the $k$-th coordinate axis for some $k\in[1,n]$ and $m_{k}(p),m_{k}(p\')>0$, then $m_{k}(p^{\prime\prime})>0$ for any point $p^{\prime\prime}\in l$.

  (v)\; Let $S$ be an $r$-dimensional face of $N_{h}$ for $h\in\Z^{n}$ such that $\rho_{h}\in \mathcal{U}_{\geqslant0}(\Sigma)$. Then there is a vector $h\'\in\Z^{ldim(S)}$ and a cluster algebra $\A\'$ with principal coefficients of rank $ldim(S)$ and a non-negative polytope projection $\tau: N_{h\'}|_{\A\'}\rightarrow S$. In particular, $\tau$ is an isomorphism when $ldim(S)=r$.
\vspace{1mm}

In Section 4, we explain that since each cluster variable $x_{l;t}$ equals $\rho_{g_{l;t}}$, cluster variables and the Newton polytopes inherit all properties shown in the last section. In particular, the definitions of $N_{h}$ and $\rho_{h}$ present a recursive way to calculate the Laurent expression of any cluster variable in a given cluster from its $g$-vector.
\vspace{2mm}\\
$\spadesuit$ {\bf (Theorem \ref{from general to cluster variables})}\;
(\textbf{Recurrence formula})  Let $\A$ be a TSSS cluster algebra having principal coefficients, then $x_{l;t}$ as well as $N_{l;t}$ can be calculated via a recurrence formula induced from the constructions of $N_h$ and $\rho_h$ for $h\in\Z^n$.
\vspace{1mm}

Based on Laurent phenomenon, in \cite{FZ1}, the  positivity conjecture for cluster variables is suggested, that is,
\begin{Conjecture}[\cite{FZ1}]\label{positive property}
  Every cluster variable of a cluster algebra $\A$  is a Laurent polynomial in cluster variables from an initial cluster $X$ with positive coefficients.
\end{Conjecture}

So far, the recent advance on the positivity conjecture is a proof in skew-symmetrizable case given in \cite{GHKK}. For totally sign-skew-symmetric cluster algebras, it was only proved in acyclic case in \cite{HL}.

As a harvest of this polytope method, a natural conclusion  of Theorem \ref{from general to cluster variables} is the following corollary, which actually completely confirms  Conjecture \ref{positive property} in the most general case:\\\\
$\spadesuit$ {\bf (Corollary \ref{TSSS positivity})}\;
 The positivity conjecture for cluster variables holds for TSSS cluster algebras.
\vspace{1mm}

Moreover, as a class of special elements in $\mathcal{P}$, they admit extra properties as the following theorem claims.
\vspace{2mm}\\
$\spadesuit$ {\bf (Theorem \ref{properties for cluster variable case})}\;
Let $\A$ be a TSSS cluster algebra having principal coefficients, $l\in[1,n],t\in\T_{n}$. Then the support of $F$-polynomial $F_{l;t}$ is saturated and for any $p\in N_{l;t}$, $co_{p}(N_{l;t})=1$ if and only if $p\in V(N_{l;t})$.
\vspace{2mm}

As a conclusion, in Corollary \ref{answer to fei}, we provide a positive answer to Conjecture \ref{F} posed in \cite{F} by Jiarui Fei.

In Section 5 we present another application.
\vspace{2mm}\\
$\spadesuit$ {\bf (Theorem \ref{positivity of d-vectors})}\; The positivity conjecture of $d$-vectors of non-initial cluster variables holds. More precisely, $d$-vector of a cluster variable can be expressed as a vector composed by general degrees of the absolute numerator of this cluster variable.

\noindent$\spadesuit$ {\bf (Corollary \ref{$F$-polynomial uniquely determines cluater variable})}\;
   Let $\A$ be a TSSS cluster algebra. Then a non-initial cluster variable is uniquely determined by its corresponding $F$-polynomial.
\vspace{2mm}

Moreover, in summary, we set the following relationship:

\vspace{2mm}
\noindent$\spadesuit$ {\bf (Theorem \ref{maps from $F$-polynomials} and Theorem \ref{maps from $g$-vectors})}\;
   For a cluster algebra with principal coefficients, there are some bijections among non-initial $F$-polynomials, $g$-vectors and cluster variables as well as surjections from non-initial $F$-polynomials, $g$-vectors or cluster variables to $d$-vectors.
\vspace{2mm}

In Section 6, we show that $\mathcal{P}$ is a strongly positive basis of the upper cluster algebra $\mathcal{U}(\A)$.
\vspace{2mm}\\
$\spadesuit$ {\bf (Theorem \ref{positive} and Theorem \ref{positive over arbitray semifield})}\;
   (i)\;For a TSSS cluster algebra $\A$ with principal coefficients, the set  $\mathcal{P}=\{\rho_{h}\in\N Trop(Y)[X^{\pm 1}]|h\in\Z^{n}\}$ is a strongly positive $\Z Trop(Y)$-basis for the upper cluster algebra $\mathcal{U}(\A)$ which we call the {\bf polytope basis}.

   (ii)\;Let $\A$ be a cluster algebra over a semifield $\P$. Then $\mathcal{P}$ is a strongly positive $\Z\P$-basis for the intermediate cluster algebra $\mathcal{I_P(A)}$ (see its definition in Page 60).\vspace{2mm}

In Section 7, when $\A$ is in particular skew-symmetrizable, we can calculate the cluster algebra associated to each face $S$ by the following result:
\vspace{2mm}\\
$\spadesuit$ {\bf (Theorem \ref{properties of N_h for skew-symmetrizable})}\;
In Theorem \ref{properties in case tsss} (v), if $\A$ is a skew-symmetrizable cluster algebra with principal coefficients whose initial exchange matrix is $B$, and denote by $B\'$ the initial exchange matrix of the cluster algebra $\A\'$, then the relation between $B\'$ and $B$ is showed by an equation.
\vspace{2mm}

\section{Some needful conclusions}
Then we introduce some important conclusions in cluster algebras and several lemmas for further discussion in the next sections.

In order to prove Laurent phenomenon of a cluster algebra, it is first proved in \cite{BFZ} that
\begin{Theorem}\label{theorem of upper cluster algebras} \cite{BFZ}
  For any vertices $t,t\'\in\T_{n}$ connected by an edge labeled $k\in[1,n]$, assume $M_{i;s}$ and $M_{j;s}$ are coprime for any $i\neq j\in[1,n]$,$s=t\text{ or }t\'$. Then their corresponding upper bounds coincide, that is, $\mathcal{U}(\Sigma_{t})=\mathcal{U}(\Sigma_{t\'})$.
\end{Theorem}

In particular, when $\A$ is a cluster algebra having principal coefficients, $M_{i;t}$ and $M_{j;t}$ are coprime for any $i\neq j\in[1,n]$, $t\in\T_{n}$. So we can get from Theorem \ref{theorem of upper cluster algebras} that $\mathcal{U}(\Sigma_{t})=\mathcal{U}(\Sigma_{t\'})$ for any $t,t\'\in\T_{n}$.

In this paper, we will use $A\mid B$ to imply that a Laurent polynomial $A$ can divide another Laurent polynomial $B$. $P|_{a\rightarrow b}$ means all $a$ in a Laurent polynomial $P$ is replaced by $b$.

\begin{Lemma}\label{generalized Laurent}
  In a cluster algebra $\A$ with principal coefficients, let $P^{t}$ be a polynomial over $\Z\P$ in $X_{t}$ and $\alpha=\frac{P^{t}}{X_{t}^{d^{t}}}$ be a Laurent polynomial in $X_{t}$ with $d^t\in \N^{n}$, where $t\in\T_{n}$. Then the following statements are equivalent:

(i)\;  $\mathcal{U}(\Sigma_{t})=\mathcal{U}(\Sigma_{t\'})$, where $t\'$ is connected to $t$ by an edge in $\T_{n}$;

(ii)\;
  $\alpha$ is a Laurent polynomial in expression of any $X_{t\'},t\'\in\T_{n}$ if and only if $M_{k;t}^{d^{t}_{k}}\mid (P^{t}|_{x_{k;t}\rightarrow M_{k;t}})$ for any $k\in[1,n]$.
\end{Lemma}
\begin{Proof}
 (i) $\Longrightarrow$ (ii):\; Firstly, we prove the necessity. Since $\alpha$ is a Laurent polynomial in expression of any $X_{t\'}$ for any $t\'\in\T_{n}$, in particular this holds when $t\'$ is the vertex connected to $t$ by an edge labelled $k$ in $\T_{n}$. By the definition of mutations, we have that $\alpha=\frac{P^{t}}{X_{t}^{d^{t}}}|_{x_{k;t}\rightarrow x_{k;t\'}^{-1}M_{k;t}}$ and it is a Laurent polynomial. So, $M_{k;t}^{d^{t}_{k}}\mid (P^{t}|_{x_{k;t}\rightarrow M_{k;t}})$ for any $k\in[1,n]$ as $d^t\in\N^{n}$.

  Secondly, we prove the sufficiency. $M_{k;t}^{d^{t}_{k}}\mid (P^{t}|_{x_{k;t}\rightarrow M_{k;t}})$ for any $k\in[1,n]$ ensures that $\alpha=\frac{P^{t}}{X_{t}^{d^{t}}}|_{x_{k;t}\rightarrow x_{k;t\'}^{-1}M_{k;t}}$ is a Laurent polynomial for any $k$, i.e., $\alpha\in \mathcal{U}(\Sigma_{t})$. Then by statement (i), we have that $\alpha\in \mathcal{U}(\Sigma_{t\'})\subseteq\Z\P[X_{t\'}^{\pm 1}]$ for any $t\'\in\T_{n}$.

 (ii) $\Longrightarrow$ (i):\; If $M_{k;t}^{d^{t}_{k}}\mid (P^{t}|_{x_{k;t}\rightarrow M_{k;t}})$ for any $k\in[1,n]$ can lead to that $\alpha$ is a Laurent polynomial in expression of any $X_{t\'},t\'\in\T_{n}$, then $\mathcal{U}(\Sigma_{t})\subseteq\Z\P[X_{t\'}^{\pm 1}]$ for any $t\'\in\T_{n}$. Hence $\mathcal{U}(\Sigma_{t})\subseteq\mathcal{U}(\Sigma_{t\'})$ for any $t\'\in\T_{n}$. Therefore $\mathcal{U}(\Sigma_{t})=\mathcal{U}(\Sigma_{t\'})$ because of the arbitrary choice of $t$.
\end{Proof}

Hence Lemma \ref{generalized Laurent} (ii) gives an equivalent statement of Theorem \ref{theorem of upper cluster algebras}.

Until now, there are so many researchers studying about cluster algebras and many important properties are found. Here we would like to list some of them which are helpful in our research. Although we may not use these results directly in this paper, they help us to understand cluster algebras better and inspire our construction of $\rho_{h}$.

\begin{Theorem}\cite{GLS}\label{P is irreducible}
  For any skew-symmetrizable cluster algebra $\A$, $l\in [1, n]$ and $t,t\'\in\T_{n}$, $P_{l;t}^{t\'}$ is irreducible as a polynomial in $\Z\P[X_{t\'}]$.
\end{Theorem}

Theorem \ref{theorem of upper cluster algebras}, Lemma \ref{generalized Laurent} and Theorem \ref{P is irreducible} can lead to the result that non-initial cluster variable is uniquely determined by its corresponding $F$-polynomial (Corollary \ref{$F$-polynomial uniquely determines cluater variable} for skew-symmetrizable case). But in the sequel, we will give the proof of this theorem in another way as an application of Newton polytope. In fact, we will provide a stronger result showing how a non-initial $F$-polynomial determines its corresponding $d$-vector specifically, which will lead to Corollary \ref{$F$-polynomial uniquely determines cluater variable} directly.

\begin{Theorem}\label{max term and constant term}
  \cite{GHKK}For any skew-symmetrizable cluster algebra $\A$, each $F$-polynomial $F_{l;t}^{t\'}$ has constant term $1$ and a unique monomial of maximal degree. Furthermore, this monomial has coefficient $1$, and it is divisible by all the other occurring monomials.
\end{Theorem}

As we introduced above, Laurent Phenomenon ensures that any cluster variable $x_{l;t}$ can be expressed as a Laurent polynomial of any cluster $X_{t^{\prime}}$:
\begin{equation*}
  x_{l;t}=\frac{P_{l;t}^{t^{\prime}}}{\prod\limits_{i=1}^{n}x_{i;t^{\prime}}^{d_{i}^{t^{\prime}}(x_{l;t})}},
\end{equation*}
where $P_{l;t}$ is a polynomial in $\Z\P[x_{1;t^{\prime}},x_{2;t^{\prime}},\cdots,x_{n;t^{\prime}}]$ which is not divisible by $x_{1;t^{\prime}},x_{2;t^{\prime}},\cdots,x_{n;t^{\prime}}$.

For any $l,k\in[1,n]$, $t,t^{\prime}\in\T_{n}$, we can express $P_{l;t}^{t\'}$ as
\begin{equation} \label{P decomp}
  P_{l;t}^{t^{\prime}}=\sum\limits_{s=0}^{deg_{x_{k;t^{\prime}}}(P_{l;t}^{t\'})}x_{k;t^{\prime}}^{s}P_{s}(k)=\sum\limits_{s=d_{k}^{t\'}(x_{l;t})}^{deg_{x_{k;t\'}}(P_{l;t}^{t\'})}x_{k;t\'}^{s}P_{s}(k)+ \sum\limits_{s=0}^{d_{k}^{t\'}(x_{l;t})-1}x_{k;t\'}^{s}P_{s}(k),
\end{equation}
where $P_{s}(k)$ is a polynomial in $\Z\P[x_{1;t^{\prime}},\cdots,x_{k-1;t^{\prime}},x_{k+1;t^{\prime}},\cdots,x_{n;t^{\prime}}]$ for any $s$. Note that $P_{0}(k)\neq 0$ and $\sum\limits_{s=d_{k}^{t\'}(x_{l;t})}^{deg_{x_{k;t\'}}(P_{l;t}^{t\'})}x_{k;t\'}^{s}P_{s}(k)=0$ if $deg_{x_{k;t^{\prime}}}(P_{l;t}^{t\'})<d_{k}^{t\'}(x_{l;t})$.

\begin{Lemma}\label{mutations of $d$-vector}
  For any $l,k\in[1,n]$, $t,t_{1},t_{2}\in\T_{n}$, if $t_{1}$ and $t_{2}$ are connected by an edge labeled j, then
  \begin{equation*}
    d_{k}^{t_{2}}(x_{l;t})=\left\{
    \begin{array}{lr}
      d_{k}^{t_{1}}(x_{l;t}) & \text{if }k\neq j; \\
      deg_{x_{j;t_{1}}}(P_{l;t}^{t_{1}})-d_{j}^{t_{1}}(x_{l;t}) & \text{if }k=j.
    \end{array}
    \right .
  \end{equation*}
\end{Lemma}
\begin{proof}
  By the mutation formula (\ref{equation: mutation of x}),
  $x_{i;t_{1}}=\left\{
    \begin{array}{lr}
      x_{i;t_{2}} & i\neq j; \\
      \frac{M_{j;t_{2}}}{x_{j;t_{2}}} & i=j.
    \end{array}
    \right.$
  Assume $x_{l;t}$ can be expressed as a Laurent polynomial of $X_{t_{1}}$ as
  \begin{equation*}
      x_{l;t}=\frac{\sum\limits_{s=0}^{deg_{x_{j;t_{1}}}(P_{l;t}^{t_{1}})}x_{j;t_{1}}^{s}P_{s}(j)}{\prod\limits_{i=1}^{n}x_{i;t_{1}}^{d_{i}^{t_{1}}(x_{l;t})}},
  \end{equation*}
  then we can get the expression of $x_{l;t}$ by $X_{t_{2}}$ as
  \begin{equation*}
      x_{l;t}=\frac{x_{j;t_{2}}^{d_{j}^{t_{1}}(x_{l;t})}(\sum\limits_{s=0}^{deg_{x_{j;t_{1}}}(P_{l;t}^{t_{1}})}(\frac{M_{j;t_{2}}}{x_{j;t_{2}}})^{s}P_{s}(j))}{\prod\limits_{i\neq j}x_{i;t_{2}}^{d_{i}^{t_{1}}(x_{l;t})}M_{j;t_{2}}^{d_{j}^{t_{1}}(x_{l;t})}}
      =\frac{M_{j;t_{2}}^{-d_{j}^{t_{1}}(x_{l;t})}(\sum\limits_{s=0}^{deg_{x_{j;t_{1}}}(P_{l;t}^{t_{1}})}M_{j;t_{2}}^{s}P_{s}(j)x_{j;t_{2}}^{deg_{x_{j;t_{1}}}(P_{l;t}^{t_{1}})-s})}{\prod\limits_{i\neq j}x_{i;t_{2}}^{d_{i}^{t_{1}}(x_{l;t})}x_{j;t_{2}}^{deg_{x_{j;t_{1}}}(P_{l;t}^{t_{1}})-d_{j}^{t_{1}}(x_{l;t})}},
  \end{equation*}
  which completes the proof.
\end{proof}

\begin{Lemma}\label{general deg > d}
  For any $l,k\in[1,n]$, $t,t\'\in\T_{n}$ and cluster variable $x_{l;t}$, $\widetilde{deg}_{k}^{t\'}(P^{t\'}_{l;t})\geqslant d_{k}^{t\'}(x_{l;t})$.
\end{Lemma}
\begin{proof}
  First, when $d_{k}^{t\'}(x_{l;t})\leqslant 0$, this is true as $\widetilde{deg}_{k}^{t\'}(P^{t\'}_{l;t})\geqslant 0$.

  When $d_{k}^{t\'}(x_{l;t})>0$, let $X_{t_{k}}=\mu_{k}(X_{t\'})$. Then $x_{i;t_{k}}=x_{i;t^{\prime}}$ for $i\neq k$ and $x_{k;t^{\prime}}=\frac{M_{k;t_{k}}}{x_{k;t_{k}}}$. $x_{l;t}$ can be expressed as a Laurent polynomial of $X_{t\'}$ and $X_{t_{k}}$ respectively as
  \begin{equation*}
      x_{l;t}=\frac{\sum\limits_{s=0}^{deg_{x_{k;t^{\prime}}}(P_{l;t}^{t\'})}x_{k;t^{\prime}}^{s}P_{s}(k)}{\prod\limits_{i=1}^{n}x_{i;t^{\prime}}^{d_{i}^{t^{\prime}}(x_{l;t})}},
  \end{equation*}
  and
  \begin{equation*}
    x_{l;t}=\prod\limits_{i\neq k}x_{i;t_{k}}^{-d_{i}^{t\'}(x_{l;t})}(\sum\limits_{s=d_{k}^{t\'}(x_{l;t})}^{deg_{x_{k;t\'}}(P_{l;t}^{t\'})}(\frac{M_{k;t_{k}}}{x_{k;t_{k}}})^{s-d_{k}^{t\'}(x_{l;t})}P_{s}(k)+ \frac{\sum\limits_{s=0}^{d_{k}^{t\'}(x_{l;t})-1}M_{k;t_{k}}^{s}P_{s}(k)x_{k;t_{k}}^{d_{k}^{t\'}(x_{l;t})-s}}{M_{k;t_{k}}^{d_{k}^{t\'}(x_{l;t})}}).
  \end{equation*}
  Therefore, $M_{k;t_{k}}^{d_{k}^{t\'}(x_{l;t})}|\sum\limits_{s=0}^{d_{k}^{t\'}(x_{l;t})-1}M_{k;t_{k}}^{s}P_{s}(k)x_{k;t_{k}}^{d_{k}^{t\'}(x_{l;t})-s}$. Then for every $s\in[0,d_{k}^{t\'}(x_{l;t})-1]$, because $x_{k;t_{k}}$ does not appear in $M_{k;t_{k}}$ or $P_{s}(k)$, it follows that $M_{k;t_{k}}^{d_{k}^{t\'}(x_{l;t})-s}|P_{s}(k)$.

  Hence,  by the definition of general degree, $$\widetilde{deg}_{k}^{t\'}(x_{k;t^{\prime}}^{s}P_{s}(k))=deg_{x_{k;t\'}}(x_{k;t^{\prime}}^{s}P_{s}(k))+max\{s\in\N:M_{k;t\'}^{s}|x_{k;t^{\prime}}^{s}P_{s}(k)\}\geqslant s+d_{k}^{t\'}(x_{l;t})-s=d_{k}^{t\'}(x_{l;t})$$ for any $s\in[0,deg_{x_{k;t^{\prime}}}(P_{l;t}^{t\'})]$. So $\widetilde{deg}_{k}^{t\'}(P^{t\'}_{l;t})\geqslant d_{k}^{t\'}(x_{l;t})$.
\end{proof}

\begin{Remark}\label{mutation of P}
  Due to the above lemma, $M_{k;t\'}^{d^{t\'}_{k}(x_{l;t})}\mid (P_{l;t}^{t\'}|_{x_{k;t\'}\rightarrow M_{k;t\'}})$ for any $k\in[1,n]$ and we can express $x_{l;t}=P_{l;t}^{t\'}X_{t\'}^{-d_{l;t}^{t\'}}$ more explicitly by
  \begin{equation}\label{likemu1}
    x_{l;t}=\sum\limits_{s=0}^{deg_{x_{k;t\'}}(P_{l;t}^{t\'})-d_{k}^{t\'}(x_{l;t})}x_{k;t\'}^{s}P_{s}(k)+ \sum\limits_{s=-d_{k}^{t\'}(x_{l;t})}^{-1}x_{k;t\'}^{s}M_{k;t\'}^{-s}P_{s}(k).
  \end{equation}
  where $P_{s}(k)$ is a Laurent polynomial in $\Z\P[x_{1;t\'}^{\pm1},\cdots,x_{k-1;t\'}^{\pm1},x_{k+1;t\'}^{\pm1},\cdots,x_{n;t\'}^{\pm1}]$ for any $s$.

  According to the mutation formula (\ref{equation: mutation of x}), for $t^{\prime\prime}\in\T_{n}$ connected to $t\'$ by an edge labeled $k\in[1,n]$, the Laurent expression of $x_{l;t}$ in $X_{t^{\prime\prime}}$ is obtained from that in $X_{t\'}$ by the following way:

  The $x_{k;t\'}$-homogeneous term $x_{k;t\'}^{s}M_{k;t\'}^{[-s]_{+}}P_{s}(k)$ with $x_{k;t\'}$-degree $s$ is changed to $x_{k;t^{\prime\prime}}$-homogeneous term $x_{k;t^{\prime\prime}}^{-s}M_{k;t^{\prime\prime}}^{[s]_{+}}P_{s}(k)$ with $x_{k;t^{\prime\prime}}$-degree $-s$. Therefore, the Laurent expression of $x_{l;t}$ in $X_{t^{\prime\prime}}$ equals
  \begin{equation}\label{likemu2}
  \sum\limits_{s=0}^{deg_{x_{k;t\'}}(P_{l;t}^{t\'})-d_{k}^{t\'}(x_{l;t})}x_{k;t^{\prime\prime}}^{-s}M_{k;t^{\prime\prime}}^{s}P_{s}(k)+ \sum\limits_{s=-d_{k}^{t\'}(x_{l;t})}^{-1}x_{k;t^{\prime\prime}}^{-s}P_{s}(k).
  \end{equation}
  Here the Laurent expressions are for the same cluster algebra. However in the following, we always deal with the cluster algebra with principal coefficients associated to $B_{t\'}$ when we express $x_{l;t}$ in $X_{t\'}$. Therefore, we should also take the change of semifield into consideration. According to Theorem \ref{properties in case tsss} (iii), the expression of $X_{t^{\prime\prime}}$ under semifield change
  \[\sum\limits_{s=0}^{deg_{x_{k;t\'}}(P_{l;t}^{t\'})-d_{k}^{t\'}(x_{l;t})}x_{k;t^{\prime\prime}}^{-s}M_{k;t^{\prime\prime}}^{s}P\'_{s}(k)+ \sum\limits_{s=-d_{k}^{t\'}(x_{l;t})}^{-1}x_{k;t^{\prime\prime}}^{-s}P\'_{s}(k)\]
  is obtained from (\ref{likemu2}) by dividing $y_{k;t\'}^{[g_{kl;t}^{t\'}]_{+}}$ and then substituting Y-variables $y_{i;t\'}$ by Laurent monomials in $Y_{t^{\prime\prime}}$ according to (\ref{equation: mutation of y}), where $g_{kl;t}^{t\'}$ is the $k$-th entry of $g_{l;t}^{t\'}$.

  In the sequel, we say that under the mutation in direction $k$, $x_{k;t\'}^{s}M_{k;t\'}^{[-s]_{+}}p$ and $x_{k;t^{\prime\prime}}^{-s}M_{k;t^{\prime\prime}}^{[s]_{+}}p\'$ \textbf{correlate} to each other for any non-zero monomial summands $p$ of $P_{s}(k)$ and $p\'$ of $P\'_s(k)$ such that $p\'$ can be obtained from $p$ by dividing $y_{k}^{[h_{k}]_{+}}$ and then substituting Y-variables $y_{i;t\'}$ by Laurent monomials in $Y_{t^{\prime\prime}}$ according to (\ref{equation: mutation of y}), and we also say two non-zero monomial summands of $x_{k}^{s}M_{k}^{[-s]_{+}}p$ and $x_{k;t_k}^{-s}M_{k;t_k}^{[s]_{+}}p\'$ respectively correlate to each other under the mutation in direction $k$.

  In this sense, we also say two faces $S_1$ and $S_2$ of $N_{l;t}^{t\'}$ and $N_{l;t}^{t^{\prime\prime}}$ respectively correlate to each other under the mutation in direction $k$ if for any point $q$ in $S_i$ with non-zero weight, there is $q\'$ in $S_{3-i}$ with non-zero weight correlated to $q$ under the mutation in direction $k$ for $i=1,2$.

  These notions can be naturally extended to polytope functions $\rho_h$ and its Newton polytope $N_h$ defined later for any $h\in\Z^n$.
\end{Remark}

\vspace{4mm}

\section{Polytope associated to an integer vector and the relevant polytope functions}

We will construct a collection of polytopes as well as their corresponding Laurent polynomials associated to vectors and show that they admit some interesting properties. In this section, assume $\A$ is a totally sign-skew-symmetric cluster algebra with principal coefficients.

Before introducing the construction, we would like to explain some notations first.

In this paper, for any $i\in\Z$, $j\in\N$, we denote binomial coefficients
\begin{equation*}
  \begin{pmatrix}
    i \\
    j
  \end{pmatrix}\triangleq\left\{\begin{array}{lr}
                         \frac{i(i-1)\cdots(i-j+1)}{j!} & \text{if}\; j> 0; \\
                         1 & \text{if}\; j=0.
                       \end{array}\right.
\end{equation*}
and denote
\begin{equation*}
  \tilde{C}_{i}^{j}\triangleq\left\{\begin{array}{lr}
            \begin{pmatrix}
              i \\
              j
            \end{pmatrix} & \text{if}\; i>0; \\
            0 & \text{if}\; i\leqslant 0.
          \end{array}\right.
\end{equation*}
as modified binomial coefficients.

In this paper we denote the canonical projections and embeddings respectively as
\[\pi_{i}:\quad\R^{n}\quad\longrightarrow\quad\R^{n-1}\quad\quad\quad\qquad\]
\[(\alpha_{1},\cdots,\alpha_{i},\cdots,\alpha_{n})\quad\mapsto\quad (\alpha_{1},\cdots,\alpha_{i-1},\alpha_{i+1}\cdots,\alpha_{n}), \;\]
and
\[\gamma_{i;j}:\quad\R^{n-1}\quad\longrightarrow\quad\R^{n}\quad\quad\quad\quad\quad\quad\quad\quad\quad\quad\]
\[(\alpha_{1},\cdots,\alpha_{n-1})\quad\mapsto\quad (\alpha_{1},\cdots,\alpha_{i-1},j,\alpha_{i},\cdots,\alpha_{n-1}),\]
where $i\in[1,n]$ and $j\in\R$. We extend $\gamma_{i;j}$ to be a map from the set of polytopes in $(n-1)$-dimensional real vector space to that of polytopes in $n$-dimensional real vector space, which is also denoted as $\gamma_{i;j}$, that is, $\gamma_{i;j}(N)=\{\gamma_{i;j}(p)\;|\; \forall p\in N\}\subset\R^n$ for any polytope $N\subseteq\R^{n-1}$. It is easy to see $\gamma_{i;j}(N)$ is a polytope in $\R^n$.

\subsection{Outline of the idea of the polytope function $\rho_h$ associated to $h\in\Z^n$} \label{2dim}\quad

Before introducing our construction, we would like to briefly explain our idea to make it sightly easier to understand our main objects $N_h$ as well as $\rho_h$ in this paper.

As a generalization of cluster monomials, we want to construct $\rho_h$ from $X^h$ satisfying that it can be expressed as a homogeneous (formal) Laurent polynomial in $X_t$ with coefficients in $\N[Y_t]$ under $L^t$ for any $t\in\T_n$, and includes $X^h$ as a summand. Such global conditions are too complicated to deal with directly, so we first try to construct a (formal) Laurent polynomial in $\widehat{\mathcal{U}}^+_{\geqslant0}(\Sigma_{t_0})$ from $X^h$ and then prove it satisfies the above global conditions.

Based on the above idea, $\rho_h$ can be constructed in the following three steps, while more details will be given in the sequel.

(i)\; Given a vector $h=(h_1,\cdots,h_n)\in\Z^n$, we get a coefficient free Laurent monomial $X^h=x_1^{h_1}\cdots x_n^{h_n}$ in $X$.
In general, it can not be expressed as a Laurent polynomial with positive coefficients in any cluster. For example, when $h_k<0$ for some $k\in[1,n]$, the expression of $X^h$ in $X_{t_{k}}$ via mutation $\mu_k$ equals to
\[(\frac{M_{k;t_k}}{x_{k;t_k}})^{h_k}\prod\limits_{i\neq k}x_{i;t_{k}}^{h_{i}}=\frac{x_{k;t_{k}}^{-h_k}\prod\limits_{i\neq k}x_{i;t_{k}}^{h_{i}}}{M_{k;t_k}^{-h_k}},\]
which is not a Laurent polynomial in $X_{t_k}$, where $t_k\in\T_n$ is the vertex connected to $t_0$ by an edge labeled $k$. Our method is to add some Laurent polynomial in $X$ to make the summation also a Laurent polynomial in $X_{t_k}$. Concretely, we find $(\hat{y}_k+1)^{-h_k}X^h=x_{k}^{h_k}M_{k;t_k}^{-h_{k}}\prod\limits_{i\neq k}x_{i}^{h_{i}+[-b_{ik}]_+h_k}$ having $X^h$ as a summand, which can be expressed as a Laurent polynomial in $X_{t_k}$.

(ii)\;If there is $k\'\in[1,n]$ such that $(\hat{y}_k+1)^{-h_k}X^h$ can not be expressed as a Laurent polynomial in $X_{t_{k\'}}$, then there is a summand
$x_{k\'}^{-a}p$, where $a\in\Z_{>0}$ and $p$ is some Laurent monomial in $\N[Y][x_{1}^{\pm 1},\cdots,x_{k\'-1}^{\pm 1},x_{k\'+1}^{\pm 1},\cdots,x_{n}^{\pm
1}]$, causing the expression non-Laurent polynomial similar as $X^h$ we deal with above.

Again we need to find an appropriate Laurent polynomial $x_{k\'}^{-a}M_{k\',t_{k\'}}^{a}q\in\N[Y][X^{\pm 1}]$ which has $x_{k\'}^{-a}p$ as a summand (here the ``appropriate'' refers to the condition induced by our aim $\rho_h\in\widehat{\mathcal{U}}^+_{\geqslant0}(\Sigma_{t_0})$, which restricts the support of $N_h$ in certain region), where $q$ is a Laurent monomial in $\N[Y][x_{1}^{\pm 1},\cdots,x_{k\'-1}^{\pm 1},x_{k\'+1}^{\pm 1},\cdots,x_{n}^{\pm 1}]$..

In the above process we call $x_{k}^{-a}M_{k,t_k}^{a}q$ a {\bf complement} of $x_{k}^{-a}p$ in direction $k$.

(iii)\;Then we focus on the minimal Laurent polynomial having both $(\hat{y}_k+1)^{-h_k}X^h$ and $x_{k\'}^{-a}M_{k\',t_{k\'}}^{a}q$ as summands and look for $k^{\prime\prime}\in[1,n]$ if it exists such that the minimal Laurent polynomial can not be expressed as a Laurent polynomial in $X_{t_{k^{\prime\prime}}}$ and to repeat step (ii) for $k^{\prime\prime}$. Such construction keeps on until the final (formal) Laurent polynomial can be expressed as a (formal) Laurent polynomial in any $X_{t_k}$ for $k\in[1,n]$, and we denote it by $\rho_h$. Note that $\rho_h$ is a Laurent polynomial if the construction ends in finitely many steps, otherwise it is a formal Laurent polynomial.

In summary, the construction is achieved by inductively adding a complement of some monomial summand in certain direction $k^{(s)}$ with negative exponent of $x_{k^{(s)}}$. In this way we construct a (formal) Laurent polynomial $\rho_h$ in $\widehat{\mathcal{U}}^+_{\geqslant0}(\Sigma_{t_0})$ having $X^h$ as a summand. And it will turn out $\rho_h$ is moreover universally positive. Hence it is really the object we search for. In the above process, we also keep $\rho_{h}$ ``minimal'' to make it universally indecomposable by avoiding unnecessary summands.

{\bf Supplemental illustration on the case of higher ranks ---}

For a cluster algebra of rank 2, we will construct polytope function $\rho_h$ by the above three steps. Although such polytope functions can also be constructed similarly for higher ranks, however, in general, it seems inefficient to build polytopes of higher ranks by segments. So with the help of the decomposition $x_{i}\rho_{h}=\sum\limits_{w,\alpha} c_{w,\alpha}Y^{w}\rho_{\alpha} $ for $\rho_h$ (see (\ref{equation: statement for decomposition})), we achieve our construction through induction on the partial order induced by sub-polytopes. The form of this construction may seem different in general rank from the above three steps, but it still comes from the idea we just explain.

During the construction of a polytope function $\rho_h$, the Newton polytope $N_{h}$ associated to $\rho_{h}$ is constructed with the order induced by sub-polytopes and some combinatorial structures. Thus,  it is more convenient for us to construct and study $\rho_h$ via $N_h$.

\subsection{Polytope $N_h$ and polytope function $\rho_h$ in rank 2 case}\quad

Calculation under the above idea leads us to the following definition of $N_{h}$ as well as $\rho_{h}$.

When $\A$ is a cluster algebra with principal coefficients of rank $2$,  without loss of generality, assume  that the initial exchange matrix is
\[B=\begin{pmatrix}
      0 & b \\
      -c & 0
    \end{pmatrix},\]
where $b,c\in\Z_{>0}$.

For $h=(h_{1},h_{2})\in\Z^{2}$, as explained in the last subsection,

(1)\; Starting from $X^h$, we have $(1+\hat{y}_1)^{[-h_1]_+}X^h$ as the complement of $X^h$ in direction 1 according to (i) in the last page. Since $(1+\hat{y}_1)^{[-h_1]_+}X^h=\sum\limits_{i=0}^{[-h_1]_+}\begin{pmatrix}
                                                                                       [-h_{1}]_+ \\
                                                                                       i
                                                                                     \end{pmatrix}\hat{y}_{1}^iX^h$,
it corresponds to a segment with vertices $v_1$ and $v_2$ via the bijection $\tilde{v}$ defined in (\ref{polypoly}) of the first section, where
\begin{equation*}
  \begin{array}{l}
     v_1=(0,0), \\
   v_2=([-h_{1}]_{+}, 0),
  \end{array}
\end{equation*}

(2)\; Keep on doing (ii) in the last page, we have $(1+\hat{y}_2)^{[-h_{2}+c[-h_{1}]_{+}]_{+}}\hat{y}_1^{[-h_1]_+}X^h$ as the complement of $\hat{y}_1^{[-h_1]_+}X^h$ in direction 2, which corresponds to a segment with vertices $v_2$ and $v_3$, where
$$v_3=([-h_{1}]_{+}, [-h_{2}+c[-h_{1}]_{+}]_{+}).$$
Similarly, we have $(1+\hat{y}_2)^{[-h_2]_+}X^h$ as the complement of $X^h$ in direction 2, which corresponds to a segment with vertices $v_1$ and $v_4$, where
\begin{equation*}
v_4=(0, [-h_{2}]_{+}).
\end{equation*}
And we have
$$(1+\hat{y}_1)^{[-h_1]_+}\hat{y}_1^{[-h_{1}]_{+}-[[-h_{1}]_{+}-b[c[-h_{1}]_{+}-h_{2}]_{+}]_{+}}\hat{y}_2^{[-h_{2}+c[-h_{1}]_{+}]_{+}}X^h$$
as the complement of $\hat{y}_1^{[-h_{1}]_{+}}\hat{y}_2^{[-h_{2}+c[-h_{1}]_{+}]_{+}}X^h$ in direction 1, which corresponds to a segment with vertices $v_3$ and $v_5$, where
\begin{equation*}
v_5=([-h_{1}]_{+}-[[-h_{1}]_{+}-b[c[-h_{1}]_{+}-h_{2}]_{+}]_{+}, [-h_{2}+c[-h_{1}]_{+}]_{+}).
\end{equation*}

Note that the next vertex $v_6$ calculated in this process is in the convex hull of $\{v_1,v_2,v_3,v_4,v_5\}$, so it is enough for us to construct the polytope we want inductively from these five vertices.

The calculation goes on to find all complements, this can be summarized as following inductively based on the above three points $v_1,v_2,v_3$.

For any point $p_0=(u_0,v_0)$ on $p_{1}p_{2}$, where $p_{1}p_{2}$ is either $v_{1}v_{2}$ or $v_{2}v_{3}$, define the weight $co_{p_0}=\tilde{C}_{l(\overline{p_{1}p_{2}})}^{l(\overline{p_0p_{2}})}$, and denote
\begin{equation}\label{equation: m_j for E_h}
  m_{1}(p_0)=\left\{\begin{array}{rl}
                      co_{p_0}, & \text{if } u_{0}=-h_{1}; \\
                      0, & otherwise.
                    \end{array}\right. \;\;\;\;
  \text{and}\;\;\;\;
  m_{2}(p_0)=\left\{\begin{array}{rl}
                      co_{p_0}, & \text{if } v_{0}=0; \\
                      0, & otherwise.
                    \end{array}\right.
\end{equation}
For any other point $p=(u,v)$, define $co_{p}$ inductively as follows: $$co_{p}=m_{1}(p)=m_{2}(p)=0\;\;\; \text{if}\;\;\; u>[-h_{1}]_{+}\; \text{or}\; v<0;$$
otherwise,
\begin{equation}\label{equation: weight of rank2}
co_{p}=max\{\sum\limits_{i=1}^{[-h_{1}]_{+}-u}m_{1}((u+i,v))\tilde{C}_{-h_{1}-bv}^{i},\;\; \sum\limits_{i=1}^{v}m_{2}((u,v-i))\tilde{C}_{-h_{2}+cu}^{i}\}
\end{equation}
while
\begin{equation}\label{equation: m(p)}
  m_{1}(p)=co_{p}-\sum\limits_{i=1}^{[-h_{1}]_{+}-u}m_{1}((u+i,v))\tilde{C}_{-h_{1}-bv}^{i},\;\; m_{2}(p)=co_{p}-\sum\limits_{i=1}^{v}m_{2}((u,v-i))\tilde{C}_{-h_{2}+cu}^{i}.
\end{equation}
Note that by induction $co_p,m_{1}(p),m_{2}(p)\geqslant0$ always holds according to (\ref{equation: weight of rank2}) and (\ref{equation: m(p)}). Then, we denote by {\bf $N_{h}$  the convex hull of the set $\{p\in\N^{2}\mid co_{p}\neq 0\}$ with weight $co_{p}$ for each $p\in \N^2$}.

(3)\;According to the definition of $N_{h}$ for $h\in\Z^{2}$, it is easy to see that its support is finite. Hence we can associate a Laurent polynomial
\begin{equation}\label{equation: rho_h in rank 2}
  \rho_{h}=\sum\limits_{p\in N_{h}}co_{p}\hat{Y}^{p}X^{h}
\end{equation}
to each $N_{h}$. Such $\rho_{h}$ is homogeneous with grading $h$.

Let $V_{h}=\{v_1,v_2,v_3,v_4,v_5\}$ and the \textbf{essential skeleton} $E_{h}$ of $N_h$ be the set consisting of edges connecting points in $V_{h}$ and parallel to $e_{1}$ or $e_{2}$.

We call $\rho_h$ the {\bf polytope function} associated to vector $h$. As mentioned before, in this case $N_{h}$ is the Newton polytope of $\rho_{h}|_{x_{i}\rightarrow1}$.

Next we will show that $N_h$ and $\rho_h$ are exactly the ones we are looking for in the last subsection.

When $\A$ is a cluster algebra without coefficients of rank $2$, we know in this case $\A=\mathcal U(\A)$ since $\A$ is acyclic. A $\Z$-basis $\{x[d]\mid d\in\Z^{2}\}$ for $\mathcal U(\A)$ was found in \cite{LLZ} called the {\bf greedy basis}, where $x[d]=X^{-d}\sum\limits_{u,v\in\N}c(u,v)x_{1}^{bu}x_{2}^{cv}$ with $c(0,0)=1$ and
\begin{equation}\label{equation: greedy basis}
  c(u,v)=max\{\sum\limits_{k=1}^{u}(-1)^{k-1}c(u-k,v)\begin{pmatrix}d_{2}-cv+k-1 \\k\end{pmatrix},
\sum\limits_{k=1}^{v}(-1)^{k-1}c(u,v-k)\begin{pmatrix}d_{1}-bu+k-1 \\k\end{pmatrix}\}
\end{equation}
for each $(u,v)\in\N^{2}\setminus\{(0,0)\}$.

On the other hand, we also have a set of Laurent polynomials $\{\rho_{h}|_{y_{i}\rightarrow 1,\forall i\in[1,2]} | h\in\Z^{2}\}$, where $\rho_{h}$ is defined for the principal coefficients cluster algebra corresponding to $\A$ as above. Here we modify $\rho_h$ for $\A$ by setting $y_{i}$ to be 1 for $i=1,2$.

The following result claims that the greedy basis is in fact the same as $\{\rho_{h}|_{y_{i}\rightarrow 1,\forall i\in[1,2]}|h\in\Z^{2}\}$ in the above case for rank 2.

\begin{Proposition}\label{polytope basis equals greedy basis}
  Let $\A$ be a cluster algebra without coefficients of rank $2$. Then $\{\rho_{h}|_{y_{i}\rightarrow 1,\forall i\in[1,2]}\mid h\in\Z^{2}\}$ and the greedy basis $\{x[d]\mid d\in\Z^{2}\}$ are the same. More precisely, $\rho_{h}|_{y_{i}\rightarrow 1,\forall i\in[1,2]}=x[d]$ for any $h=(h_1,h_2)\in\Z^{2}$, where $d=(d_{1},d_{2})=(-h_{1},-h_{2}+c[-h_{1}]_{+})$.
\end{Proposition}
\begin{Proof}
  In this proof, we define a partial order ``$\prec$" on $\Z^{2}$ as
  \begin{equation}\label{partial order in proof}
    (u,v)\prec(u\',v\')\;\; \text{if} \;\;(-h_{1}-u,v)<(-h_{1}-u\',v\')
  \end{equation}

  We claim that $co_{(u,v)}=c(v,[-h_{1}]_{+}-u)$ for $h=(h_1,h_2)\in\Z^{2}$ and $d=(d_{1},d_{2})=(-h_{1},-h_{2}+c[-h_{1}]_{+})$, which will be proved by induction on $\N^2$ with respect to $\prec$ as follows (In fact, since $co_{(u,v)}=0$ when $u>[-h_{1}]_{+}$, we only need to focus on those with $0\leqslant u\leqslant [-h_{1}]_{+}$). More precisely, we will first show the claim for $(u,v)$ in $E_{h}$ holds. Then, we prove the claim for any points $(u,v)\in N_{h}$ by verifying that the recurrence relations (\ref{equation: weight of rank2}) and (\ref{equation: greedy basis}) for $co_{(u,v)}$ and $c(v,[-h_{1}]_{+}-u)$ respectively are the same based on the induction assumption that the claim holds for points $(u\',v\')$ satisfying $(u\',v\')\prec(u,v)$.

  According to the definition of $\rho_{h}$ for any $h\in\Z^{2}$, we have $V_{h}=\{v_1, v_2, v_3, v_4, v_5\}$, where

   $v_1=(0, 0)$,

   $v_2=([-h_{1}]_{+}, 0)$,

   $v_3=(0, [-h_{2}]_{+})$,

   $v_4=([-h_{1}]_{+}, [-h_{2}+c[-h_{1}]_{+}]_{+})$,

   $v_5=([-h_{1}]_{+}-[[-h_{1}]_{+}-b[c[-h_{1}]_{+}-h_{2}]_{+}]_{+},[-h_{2}+c[-h_{1}]_{+}]_{+})$.\\
   So $E_{h}$ is as shown in Figure \ref{figure of essential skeleton}, where two red points connected by an edge may be coincident.
  \begin{figure}[H]
    \centering
    \includegraphics[width=60mm]{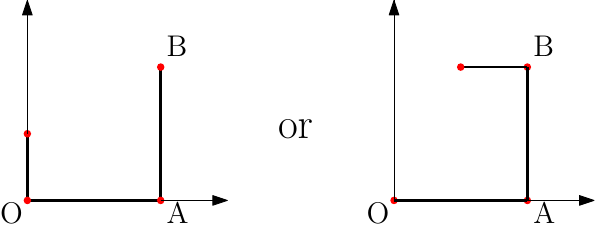}
    \caption{The shape of $E_{h}$.}\label{figure of essential skeleton}
  \end{figure}

  First $co_{([-h_{1}]_{+},0)}=c(0,0)=1$. Assume $co_{(u,0)}=c(0,[-h_{1}]_{+}-u)=\tilde{C}_{[-h_{1}]_{+}}^{u}$ when $u>r$. Then when $u=r$, we have $co_{(r,0)}=\tilde{C}_{[-h_{1}]_{+}}^{r}$ and
  \begin{equation*}
    \begin{array}{rl}
      c(0,[-h_{1}]_{+}-r)&=max\{0, \sum\limits_{k=1}^{[-h_{1}]_{+}-r}(-1)^{k-1}c(0,[-h_{1}]_{+}-r-k)\begin{pmatrix}-h_{1}+k-1 \\k\end{pmatrix}\} \\
       &=[\sum\limits_{k=1}^{[-h_{1}]_{+}-r}-\tilde{C}_{[-h_{1}]_{+}}^{[-h_{1}]_{+}-r-k}\begin{pmatrix}h_{1} \\k\end{pmatrix}]_{+}\\
       &=\tilde{C}_{[-h_{1}]_{+}}^{r}.
    \end{array}
  \end{equation*}
  So $co_{(u,0)}=c(0,[-h_{1}]_{+}-u)=\tilde{C}_{[-h_{1}]_{+}}^{u}$ for any point $(u,0)\in \overline{OA}$. Similarly, it can be proved that $$co_{([-h_{1}]_{+},v)}=c(v,0)=\tilde{C}_{[-h_{2}+c[-h_{1}]_{+}]_{+}}^{v}$$
  for any point $([-h_{1}]_{+},v)\in\overline{AB}$.

Similar discussion also works when point $(u,v)\in N_{h}$ not lying in $\overline{OA}$ or $\overline{AB}$. According to (\ref{equation: m(p)}), we can see that
  \begin{equation*}
    \begin{array}{ll}
      \sum\limits_{i=1}^{[-h_{1}]_{+}-u}m_{1}((u+i,v))\tilde{C}_{-h_{1}-bv}^{i} &
      =\sum\limits_{i=1}^{[-h_{1}]_{+}-u}(co_{(u+i,v)}-\sum\limits_{j=1}^{[-h_{1}]_{+}-u-i}m_{1}((u+i+j,v))\tilde{C}_{-h_{1}-bv}^{j})\tilde{C}_{-h_{1}-bv}^{i} \\
       & =\cdots \\
       & =\sum\limits_{i=1}^{[-h_{1}]_{+}-u}co_{(u+i,v)}\sum\limits_{r=1}^{i}\sum\limits_{i=i_{1}+\cdots+i_{r}}(-1)^{r-1}\prod\limits_{j=1}^{r}\tilde{C}_{-h_{1}-bv}^{i_{j}},
    \end{array}
  \end{equation*}
  where $i_{1},\cdots,i_{r}\in\Z_{>0}$. Take induction on $i$, assume \[\sum\limits_{r=1}^{k}\sum\limits_{k=k_{1}+\cdots+k_{r}}(-1)^{r-1}\prod\limits_{j=1}^{r}\tilde{C}_{-h_{1}-bv}^{k_{j}}=-\begin{pmatrix}
                                                                                                                           -[-h_{1}-bv]_{+} \\
                                                                                                                           k
                                                                                                                          \end{pmatrix}\]
  for $k<i$, then
  \begin{equation*}
    \begin{array}{ll}
      \sum\limits_{r=1}^{i}\sum\limits_{i=i_{1}+\cdots+i_{r}}(-1)^{r-1}\prod\limits_{j=1}^{r}\tilde{C}_{-h_{1}-bv}^{i_{j}} &
      =\tilde{C}_{-h_{1}-bv}^{i}-\sum\limits_{j=1}^{i-1}\tilde{C}_{-h_{1}-bv}^{j}(\sum\limits_{r=2}^{i}\sum\limits_{i-j=i_{2}+\cdots+i_{r}}(-1)^{r-2}\tilde{C}_{-h_{1}-bv}^{i_{j}})\\
      & =\sum\limits_{j=1}^{i}\begin{pmatrix}
                              [-h_{1}-bv]_{+} \\
                              j
                            \end{pmatrix}\begin{pmatrix}
                                           -[-h_{1}-bv]_{+} \\
                                           i-j
                                         \end{pmatrix}\\
       & =\sum\limits_{j=0}^{i}\begin{pmatrix}
                              [-h_{1}-bv]_{+} \\
                              j
                            \end{pmatrix}\begin{pmatrix}
                                           -[-h_{1}-bv]_{+} \\
                                           i-j
                                         \end{pmatrix}-
           \begin{pmatrix}
             -[-h_{1}-bv]_{+} \\
             i
           \end{pmatrix}\\
           & =-\begin{pmatrix}
             -[-h_{1}-bv]_{+} \\
             i
           \end{pmatrix},
    \end{array}
  \end{equation*}
  where the last equality holds because
  $\sum\limits_{j=0}^{k}\begin{pmatrix}
                           a_{1} \\
                           j
                         \end{pmatrix}\begin{pmatrix}
                                        a_{2} \\
                                        k-j
                                      \end{pmatrix}=\begin{pmatrix}
                                                      a_{1}+a_{2} \\
                                                      k
                                                    \end{pmatrix} $
  for any $a_{1},a_{2}\in\Z$ and $k\in\N$. Hence by induction we get that
  \begin{equation}\label{equation: m_1 in greedy basis}
    \begin{array}{ll}
      \sum\limits_{i=1}^{[-h_{1}]_{+}-u}m_{1}((u+i,v))\tilde{C}_{-h_{1}-bv}^{i} & =-\sum\limits_{i=1}^{[-h_{1}]_{+}-u}co_{(u+i,v)}\begin{pmatrix}
                                                                                                                                    -[-h_{1}-bv]_{+} \\
                                                                                                                                    i
                                                                                                                                  \end{pmatrix}, \\
       & =\sum\limits_{i}^{[-h_{1}]_{+}-u}(-1)^{i-1}co_{(u+i,v)}\begin{pmatrix}
                                                                  [-h_{1}-bv]_{+}+i-1 \\
                                                                  i
                                                                \end{pmatrix}.
    \end{array}
  \end{equation}
  Dually, we have \begin{equation}\label{equation: m_2 in greedy basis}
            \sum\limits_{i=1}^{v}m_{2}((u,v-i))\tilde{C}_{-h_{2}+cu}^{i}=\sum\limits_{i=1}^{v}(-1)^{i-1}co_{(u,v-i)}\begin{pmatrix}
                                                                                                                      [-h_{2}+cu]_{+}+i-1 \\
                                                                                                                      i
                                                                                                                    \end{pmatrix}.
          \end{equation}
 Then an induction on $(u,v)$ completes the proof of our claim as the right-hand sides of (\ref{equation: m_1 in greedy basis}) and of (\ref{equation: m_2 in greedy basis}) equals elements in the bracket of the right-hand side of (\ref{equation: greedy basis}) respectively when our claim holds for points $(u\',v\')$ such that $(u\',v\')\prec(u,v)$, which leads to $co_{(u,v)}=c(v,[-h_{1}]_{+}-u)$.

  Therefore,
  \begin{equation*}
    \begin{array}{rl}
      \rho_{h}|_{y_{i}\rightarrow 1,\forall i\in[1,2]}&=\sum\limits_{u,v\in\N}co_{(u,v)}x_{1}^{h_{1}+bv}x_{2}^{h_{2}-cu}\\
           &= x_{1}^{h_{1}}x_{2}^{h_{2}-c[-h_{1}]_{+}}\sum\limits_{u,v\in\N}c(v,[-h_{1}]_{+}-u)x_{1}^{bv}x_{2}^{c([-h_{1}]_{+}-u])}\\
           &= x_{1}^{-d_{1}}x_{2}^{-d_{2}}\sum\limits_{u,v\in\N}c(u,v)x_{1}^{bu}x_{2}^{cv}\\
           &=x[d].
    \end{array}
  \end{equation*}
\end{Proof}

Following Proposition \ref{polytope basis equals greedy basis}, since in this case $B$ is either invertible or 0, it is not hard to see that when $\A$ is a cluster algebra with pricipal coefficients, $\{\rho_{h}\mid h\in\Z^{2}\}$ is a $\Z Trop(Y)$-basis of $\A$. We would like to call it the {\bf polytope basis} as its construction is based on polytopes. Later, we will similarly construct the polytope basis for a TSSS cluster algebra of general rank.

According to Proposition \ref{polytope basis equals greedy basis}, Proposition 1.6 in \cite{LLZ} can be restated by $m_{1}$ and $m_{2}$ as follows.

\begin{Proposition}\cite{LLZ}\label{prop from llz}
  In the above settings, for any $h\in\Z^{2}$ satisfying $h_{1},h_{2}+ch_{1}<0$ and any point $(u,v)\in\N_{h}$, we have

    (i)\; $m_{1}((u,v))=0$ when $ch_{1}(-h_{1}-u)\leqslant b(h_{2}+ch_{1})v$;

    (ii)\; $m_{2}((u,v))=0$ when $ch_{1}(-h_{1}-u)\geqslant b(h_{2}+ch_{1})v$.
\end{Proposition}

For any $k\in[1,2]$, substituting (\ref{equation: m(p)}) into (\ref{equation: rho_h in rank 2}), we obtain
\begin{equation*}
  \begin{array}{rl}
    \rho_{h}= &\sum\limits_{(s_1,s_2)\in\N^2}m_k((s_1,s_2))\hat{Y}^{(s_1,s_2)}(\frac{M_{k}}{y_k^{2-k}})^{[-h_k-b_{k,3-k}s_{3-k}]_+}X^h \\
     =&\sum\limits_{(s_1,s_2)\in\N^2}(\frac{M_{k}}{x_{k}})^{[-h_k-b_{k,3-k}s_{3-k}]_+}m_k((s_1,s_2))Y^{(s_1,s_2)-(2-k)[-h_k-b_{k,3-k}s_{3-k}]_+e_k}\cdot\\
     &x_{k}^{[h_k+b_{k,3-k}s_{3-k}]_+}x_{3-k}^{h_{3-k}+b_{3-k,k}s_{k}}.
  \end{array}
\end{equation*}
So by adding up terms with the same $x_{k}-$degree, we can rewrite the above equation as
$$\rho_{h}=\sum\limits_{s\in\N}(\frac{M_{k}}{x_{k}})^{s}P_{s},$$
where $P_{s}\in\Z\P[x_{3-k}^{\pm1}]$ for $s>0$ and $P_{0}\in\Z\P[x_{3-k}^{\pm1},x_{k}]$. Hence according to the definition of mutations, $\rho_{h}\in\mathcal{U}_{\geqslant0}(\Sigma_{t_{0}})$, and so is $x_{1}\rho_{h}$.

The following result plays an important role in the construction and study of polytope functions.

\begin{Lemma}
  In the above settings, for any $h\in\Z^2$,

  (a)\;\begin{equation}\label{equation: decomposition of x1Ph}
         x_{1}\rho_{h}=\sum\limits_{(u,v)\in J^{(1)}_{h}}c_{u,v}y_{1}^{u}y_{2}^{v}\rho_{(h_{1}+1+bv,h_{2}-cu)},
       \end{equation}
       where $c_{u,v}\in\N$ and $J^{(1)}_{h}=\{(0,0)\}$ when $h_{1}\geqslant 0$ while
       $$J^{(1)}_{h}=\{(u,v)\in N_{h}\; |\; m_{1}((u+[-h_{1}-bv]_{+}, v))>0,\; m_{2}((u+[-h_{1}-bv]_{+}-1, v))>0\}\cup\{(1, 0)\}$$
       when $h_{1}<0$.

  (b)\;\begin{equation}\label{equation: decomposition of x2Ph}
         x_{2}\rho_{h}=\sum\limits_{(u,v)\in J^{(2)}_{h}}c\'_{u,v}y_{1}^{u}y_{2}^{v}\rho_{(h_{1}+bv,h_{2}+1-cu)},
       \end{equation}
       where $c\'_{u,v}\in\N$ and
       $$J^{(2)}_{h}=\{(u,v)\in N_{h}\; |\; m_{1}((u+[-h_{1}-bv]_{+},v))>0,\; m_{2}((u+[-h_{1}-bv]_{+}, v-1))>0\}\cup\{(0, 0)\}.$$
\end{Lemma}
\begin{Proof}
  Here we will only prove (\ref{equation: decomposition of x1Ph}). Then   (\ref{equation: decomposition of x2Ph}) can be proved similarly,  according to the duality of $x_{1}$ and $x_{2}$, $y_{1}$ and $y_{2}$ respectively.

  When $h_{1}\geqslant 0$, $x_{1}\rho_{h}=\rho_{(h_{1}+1,h_{2})}$, hence the equation (\ref{equation: decomposition of x1Ph}) holds in this case. When $h_{2}-c[-h_{1}]_{+}\geqslant0$, $\rho_{h}=(1+\hat{y}_{1})^{[-h_{1}]_{+}}X^{h}$, so $x_{1}\rho_{h}=(1+\hat{y}_{1})^{[-h_{1}]_{+}}X^{h+e_{1}}=\rho_{(h_{1}+1,h_{2})}+y_{1}\rho_{(h_{1}+1,h_{2}-c)}$, hence the equation (\ref{equation: decomposition of x1Ph}) holds.

  So in the following discussion, we may assume $h_1,h_2+ch_1<0$.

  Because of the correspondence between homogeneous Laurent polynomials and polytopes lying in the nonnegative cone when a degree vector is given, we will first show a decomposition of the weighted polytope $N_h$ as the sum of its weighted sub-polytopes, which then can induce the corresponding decomposition of $x_{1}\rho_h$ by equalizing the degree of both sides.

  In this proof, the partial order ``$\prec$" is defined as (\ref{partial order in proof}). We claim that we can obtain a chain of sub-polytopes $N(0)=N_{h},N(1),\cdots$, which eventually becomes the empty set $\emptyset$ from $N_{h}$ by iteratively subtracting sub-polytopes as follows and moreover, there is a unique minimal point $(u_{i},v_{i})\in N(i)$ with respect to ``$\prec$'', $(u_{i}-[-h_{1}-1-bv_{i}]_{+},v_{i})\in J^{(1)}_{h}$ and $co_{(u_{i},v_{i})}(N(i))>0$.

  $N(0)=N_{h}$ has the unique minimal point $(u_0,v_0)=(-h_{1},0)$ with respect to ``$\prec$''. Moreover, $(u_0-[-h_{1}-1-bv_0]_{+},v_0)=(1,0)\in J^{(1)}_{h}$ and $co_{(-h_{1},0)}(N(0))=1>0$. So, the claim holds when $i=0$. We set $c_{1,0}=1$ and $N(1)=N(0)-c_{1,0}N_{(h_{1}+1,h_{2}-c)}$.

  Assume the claim holds for $i\geqslant0$, that is we get a sub-polytope $N(i)$ and there is a unique minimal point $(u_i,v_i)\in N(i)$ with respect to ``$\prec$'', $(u_i-[-h_{1}-1-bv_i]_{+},v_i)\in J^{(1)}_{h}$ and $co_{(u_i,v_i)}(N(i))>0$. Then let $c_{u_i-[-h_{1}-1-bv_i]_{+},v_i}=co_{(u_i,v_i)}(N(i))$ and
  $$N(i+1)=N(i)-c_{u_i-[-h_{1}-1-bv_i]_{+},v_i}N_{(h_{1}+1+bv_i,h_{2}-c(u_i-[-h_{1}-1-bv_i]_{+}))}[(u_i-[-h_{1}-1-bv_i]_{+},v_i)]. $$

  By induction, we only need to show the above result holds for $N(i+1)$.

  Denote by $(u_{i+1},v_{i+1})$ a minimal point in $N(i+1)\cap\mathcal N$ with respect to ``$\prec$'' where
  $$ \mathcal N= \{(r+[-h_{1}-bs-1]_{+},s)|(r,s)\in J^{(1)}_{h}\}.$$
  According to Proposition \ref{prop from llz} and the definition of $J^{(1)}_{h}$, if a point $(r,s)\in J^{(1)}_{h}$, then  $$ch_{1}(-h_{1}-r-[-h_{1}-bs]_{+}+1)\leqslant b(h_{2}+ch_{1})s,\;\;\; ch_{1}(-h_{1}-r-[-h_{1}-bs]_{+})\geqslant b(h_{2}+ch_{1})s.$$ So for any two points $(r,s),(r\',s\')\in J^{(1)}_{h}$, $s>s\'$ induces $r+[-h_{1}-bs]_{+}\leqslant r\'+[-h_{1}-bs\']_{+}$. According to the definition of $N_{h}$, $m_{2}((u,v))=0$ when $h_{1}+bv\geqslant0$, hence $(r\',s\')\in J^{(1)}_{h}$ leads to $h_{1}+bs>0$. Thus $(r\'+[-h_{1}-bs\'-1]_{+},s\')\prec(r+[-h_{1}-bs]_{+},s)$. So when restricted in the set $\mathcal N$, the partial order ``$\prec$'' becomes a total order, which induces the uniqueness of $(u_{i+1},v_{i+1})$.

  Next we will show that $(u_{i+1},v_{i+1})$ is the unique minimal point in $N(i+1)$ with respect to ``$\prec$'' and $co_{(u_{i+1},v_{i+1})}(N(i+1))>0$. Recall that due to the inductive assumption,  $(u_{i},v_{i})$ is the unique minimal point in $N(i)$ with respect to ``$\prec$''. So it is enough to show that $co_{(r,s)}(N(i+1))=0$ for any $(r,s)\in N(i)$ not larger than or equal to $(u_{i+1},v_{i+1})$ with respect to ``$\prec$''.

  It is proved in \cite{CGMMRSW} that the greedy basis coincides with theta basis in rank 2 case, so together with Proposition \ref{polytope basis equals greedy basis} it induces the equivalence of polytope basis and theta basis in rank 2 case. Moreover, it was showed in \cite{GHKK} that $x_1\rho_h=\sum\limits_\alpha a_\alpha Y^{w_\alpha}\rho_{h_\alpha}$, where $a_\alpha$ is non-negative as it can be realized combinatorially as a counting of certain broken lines, $w_{\alpha}\in\N^2$ and $h_{\alpha}\in\Z^2$. Therefore, if $(u,v)$ is a minimal point in $N(i+1)$ with respect to ``$\prec$'', then $co_{(u,v)}(N(i+1))=a_\alpha>0$ for some $\alpha$.

  The summand of $x_{1}\rho_{h}$ corresponding to the 1-section of $N_{h}$ at $(u,v)$ equals
  \begin{equation*}
    \begin{array}{ll}
       & \sum\limits_{r\leqslant-h_{1}}m_{1}((r,v))(1+\hat{y}_{1}^{-1})^{-h_{1}-bv}\hat{Y}^{(r,v)}X^{h+e_{1}} \\
      = & \sum\limits_{r\leqslant-h_{1}}(m_{1}((r,v))+m_{1}((r+1,v)))(1+\hat{y}_{1}^{-1})^{-h_{1}-bv-1}Y^{(r,v)}x_1^{h_{1}+bv+1}x_2^{h_2-cr}.
    \end{array}
  \end{equation*}
  According to the choice of $(u,v)$, $co_{(r,v)}(N(i+1))=0$ when $r>u$, hence due to the definition of $N(i+1)$, the summand of $x_{1}\rho_{h}$ corresponding to the 1-section of $N(i+1)$ at $(u,v)$ must equal $\sum\limits_{r\leqslant u}a_r(1+\hat{y}_{1}^{-1})^{-h_{1}-bv-1}\hat{Y}^{(r,v)}X^{h+e_{1}}$ for some $a_r\leqslant m_{1}((r,v))+m_{1}((r+1,v))$ and thus $co_{(u,v)}(N(i+1))=a_{u}>0$. Therefore, $m_{1}((u,v))+m_{1}((u+1,v))>0$.

  Similarly, the summand of $x_{1}\rho_{h}$ corresponding to the 2-section of $N_{h}$ at $(u,v)$ equals
  \[\sum\limits_{s>0}m_{2}((u,s))(1+\hat{y}_{2})^{-h_{2}+cu}Y^{(u,s)}x_1^{h_{1}+bs+1}x_2^{h_2-cu}.\]
  According to the choice of $(u,v)$, $co_{(u,s)}(N(i+1))=0$ when $s<v$, hence the summand of $x_{1}\rho_{h}$ corresponding to the 2-section of $N(i+1)$ at $(u,v)$ must equal $\sum\limits_{s\geqslant v}a\'_s(1+\hat{y}_{2})^{-h_{2}+cu}\hat{Y}^{(u,s)}X^{h+e_{1}}$ for some $a\'_s<m_{2}((u,s))$ and thus $co_{(u,v)}(N(i+1))=a\'_{v}>0$. Therefore, $m_{2}((u,v))>0$.

  Because of Lemma \ref{fact} (a) and the fact $(u_i,v_i)\prec(u,v)$, $m_{1}((u,v))m_{2}((u,v))=0$. So $m_{1}((u+1,v))>0$ and $m_{2}((u,v))>0$, which means $(u,v)\in\mathcal{N}$ and thus $(u,v)=(u_{i+1},v_{i+1})$. In conclusion, $(u_{i+1},v_{i+1})$ is the unique minimal point in $N(i+1)$ and therefore $co_{(u_{i+1},v_{i+1})}(N(i+1))>0$.

  Note that $N(k)=\emptyset$ for large enough $k$ since $J^{(1)}_{h}$ is a finite set. Thus we obtain a chain of sub-polytopes $N_h=N(0),N(1),\cdots,N(k-1),N(k)=\emptyset$ of $N_h$ and then get a decomposition of $N_h$:
  \begin{equation*}
         \begin{array}{ll}
           N_h & =\sum\limits_{i=0}^k c_{u_i-[-h_{1}-1-bv_i]_{+},v_i}N_{(h_{1}+1+bv_i,h_{2}-c(u_i-[-h_{1}-1-bv_i]_{+}))}[(u_i-[-h_{1}-1-bv_i]_{+},v_i)] \\
            & =\sum\limits_{(u,v)\in J^{(1)}_{h}}c_{u,v}N_{(h_{1}+1+bv,h_{2}-cu)}[(u,v)], \;
         \end{array}
       \end{equation*}
  where $c_{u,v}=0$ when $(u-[-h_{1}-1-bv]_{+},v)$ is not a minimal point of some $N(i)$ with respect to ``$\prec$''.

  Due to the relation between homogeneous Laurent polynomials and polytopes lying in the nonnegative cone, in the perspective of Laurent polynomials, we have
  \[X^{a}\rho_{h}=\sum\limits_{(u,v)\in J^{(1)}_{h}}c_{u,v}X^{a_{u,v}}y_{1}^{u}y_{2}^{v}\rho_{(h_{1}+1+bv,h_{2}-cu)},\]
  where $a,a_{u,v}\in\Z^{2}$ equalize the degree of two sides. So by comparing the degree of two sides, we can see that $a_{u,v}=a+1$ for any $(u,v)\in J^{(1)}_{h}$, which leads to (\ref{equation: decomposition of x1Ph}).
\end{Proof}
\begin{Remark}
  In order to construct the decomposition (\ref{equation: decomposition of x1Ph}), besides the way given in the above proof, we can also calculate in the procedure as follows:

  Let $N_{0}=N_{h}$. For each $N_{i}$, we choose a minimal point $(u,v)\in N_{i}$ with weight $co_{(u,v)}(N_{i})$, then denote $N_{i+1}=N_{i}-co_{(u,v)}(N_{i})N_{(h_{1}+1+bv,h_{2}-cu)}[(u,v)]$. Then for large enough $i<\infty$ we get a decomposition as $N_{i}=0$, which is exactly (\ref{equation: decomposition of x1Ph}).

  The decomposition (\ref{equation: decomposition of x2Ph}) can also be constructed in the similar procedure.
\end{Remark}

In the language of polytopes, (\ref{equation: decomposition of x1Ph}) and (\ref{equation: decomposition of x2Ph}) claims that $N_h$ can be decomposed as a summation of certain sub-polytopes. This allows us to deal with many problems about them using induction on the partial order of polytopes.

\begin{Example}\label{eg in rank 2}
  let $\A$ be a cluster algebra with principal coefficients associated to the initial exchange matrix
  $B=\begin{pmatrix}
        0 & 2 \\
        -3 & 0
      \end{pmatrix} $
  and $h=(-5,7)$. Then $N_{h}$ is expressed by the following table of numbers:
   \begin{equation*}
    \begin{array}{cccccc}
       &  &  &  &  & 1 \\
       &  &  &  &  & 8 \\
       &  &  &  &  & 28 \\
       &  &  &  & 5 & 56 \\
       &  &  &  & 25 & 70 \\
       &  &  & 7 & 50 & 56 \\
       &  & 2 & 24 & 50 & 28 \\
       & 1 & 11 & 27 & 25 & 8 \\
      1 & 5 & 10 & 10 & 5 & 1
    \end{array}
  \end{equation*}
  where the number lying in the intersection of the $u$-th column from left to right and the $v$-th row from bottom to top equals the weight $co_{(u,v)}(N_{h})$ of a point $(u,v)$, or equivalently, the number $co_{(u,v)}(N_{h})$ is put at the $(u,v)$-position.

  Let $N(0)=N_{h}$, where $(0,0)$ is the unique minimal point in $N(0)$ with weight 1. So, $N(1)=N(0)-N_{(-4,7)}$, where $(1,0)$ is the unique minimal point with weight 1, and then $N(2)=N(1)-N_{(-4,4)}[(1,0)]$, where $(1,1)$ is the unique minimal point in $N(2)$ with weight 1. So, $N(3)=N(2)-N_{(-2,4)}[(1,1)]$, where $(2,2)$ is the unique minimal point in $N(3)$ with weight 2. Finally, $N(4)=N(3)-2N_{(0,1)}[(2,2)]$.
  \begin{equation*}
    \begin{array}{ccccc}
    \\\\\\
       &  &  &  & 1 \\
       &  &  &  & 5 \\
       &  &  & 2 & 10 \\
       &  &  & 8 & 10 \\
       &  & 5 & 10 & 5 \\
      1 & 4 & 6 & 4 & 1 \\\\
       &  & N_{(-4,7)} &  &
    \end{array}\qquad
    \begin{array}{ccccc}
       &  &  &  & 1 \\
       &  &  &  & 8 \\
       &  &  &  & 28 \\
       &  &  & 4 & 56 \\
       &  &  & 20 & 70 \\
       &  & 4 & 40 & 56 \\
       &  & 14 & 40 & 28 \\
       & 4 & 16 & 20 & 8 \\
      1 & 4 & 6 & 4 & 1 \\\\
       &  & N_{(-4,4)} &  &
    \end{array}\qquad
    \begin{array}{ccc}
       \\ \\ \\\\\\\\
       &  & 1 \\
       &  & 2 \\
      1 & 2 & 1 \\\\
       & N_{(-2,4)} &
    \end{array}
  \end{equation*}
  Since $N(4)=\emptyset$, we get
  \[N_{(-5,7)}=N_{(-4,7)}+N_{(-4,4)}[(1,0)]+N_{(-2,4)}[(1,1)]+2N_{(0,1)}[(2,2)],\]
  whose Laurent polynomial version equals
  \[x_{1}\rho_{(-5,7)}=\rho_{(-4,7)}+y_{1}\rho_{(-4,4)}+y_{1}y_{2}\rho_{(-2,4)}+2y_{1}^{2}y_{2}^{2}\rho_{(0,1)}.\]
  Similarly, we also have
  \[N_{(-5,7)}=N_{(-5,8)}+N_{(-3,5)}[(1,1)]+N_{(-3,2)}[(2,1)]+2N_{{-1,2}}[(3,2)]\]
  and
  \[x_{2}\rho_{(-5,7)}=\rho_{(-5,8)}+y_{1}y_{2}\rho_{(-3,5)}+y_{1}^{2}y_{2}\rho_{(-3,2)}+2y_{1}^{3}y_{2}^{2}\rho_{(-1,2)}.\]
\end{Example}

\begin{Lemma}\label{fact}\quad

(a)\; $\{p\in N_{h}|m_{1}(p)m_{2}(p)\neq 0\}=\{([-h_{1}]_{+},0)\}$, $\{([-h_{1}]_{+},v)\in N_{h}\}\subseteq\{p\in N_{h}|m_{1}(p)>0\}$ and $\{(u,0)\in N_{h}\}\subseteq\{p\in N_{h}|m_{2}(p)>0\}$.

(b)\; For any point $(u,v)\in N_{h}$,

(i)\; $m_{1}((u,v))>0$ if and only if there is $(u\',v\')\in J^{(2)}_{h}$ such that $m_{1}(u-u\',v-v\')>0$ in $N_{(h_{1}+bv\',h_{2}+1-cu\')}$.

(ii)\; $m_{2}((u,v))>0$ if and only if there is $(u^{\prime\prime},v^{\prime\prime})\in J^{(1)}_{h}$ such that $m_{2}(u-u^{\prime\prime},v-v^{\prime\prime})>0$ in $N_{(h_{1}+1+bv^{\prime\prime},h_{2}-cu^{\prime\prime})}$.

(c)\;For any points $(u,v)\in N_{h}$,

 (i)\; $m_{1}((u,v))>0$ or $m_{1}((u+1,v))>0$ if and only if there is $(u\',v\')\in J^{(1)}_{h}$ such that $m_{1}(u-u\',v-v\')>0$ in $N_{(h_{1}+1+bv\',h_{2}-cu\')}$.

 (ii)\; $m_{2}((u,v))>0$ or $m_{2}((u-1,v))>0$ if and only if there is $(u^{\prime\prime},v^{\prime\prime})\in J^{(2)}_{h}$ such that $m_{2}(u-u^{\prime\prime},v-v^{\prime\prime})>0$ in $N_{(h_{1}+bv^{\prime\prime},h_{2}+1-cu^{\prime\prime})}$.
\end{Lemma}
\begin{Proof}
  (a)\;If $p\in E_h$, then according to (\ref{equation: m_j for E_h}), $m_{1}(p)m_{2}(p)\neq 0$ only when $p=([-h_{1}]_{+},0)$; if $p\in N_h\setminus E_h$, due to (\ref{equation: weight of rank2}) and (\ref{equation: m(p)}), at least one of $m_{1}(p)$ and $m_{2}(p)$ equals 0, thus $m_{1}(p)m_{2}(p)=0$. In conclusion, $\{p\in N_{h}|m_{1}(p)m_{2}(p)\neq 0\}=\{([-h_{1}]_{+},0)\}$.

  It can be seen from (\ref{equation: m_j for E_h}) that $m_{1}(([-h_{1}]_{+},v))=co_{([-h_{1}]_{+},v)}>0$ and $m_{2}((u,0))=co_{(u,0)}>0$.

  (b)\;It is enough to show (i) holds, since (i) and (ii) are symmetric on indexes. According to (\ref{equation: decomposition of x2Ph}),
  \[m_{1}((u,v))=\sum\limits_{(u\',v\')\in J^{(2)}_{h}}c\'_{u\',v\'}m_{1}((u-u\',v-v\')),\]
  where in the right-hand side, $m_{1}((u-u\',v-v\'))$ is calculated in $N_{(h_{1}+bv\',h_{2}+1-cu\')}$. Moreover, $m_{1}((u-u\',v-v\'))\geqslant0$. Hence $m_{1}((u,v))>0$ is equivalent to $m_{1}(u-u\',v-v\')>0$ in $N_{(h_{1}+bv\',h_{2}+1-cu\')}$ for some $(u\',v\')\in J^{(2)}_{h}$.

  (c)\;It is sufficient to show the first claim holds. According to (\ref{equation: decomposition of x1Ph}),
  \[m_{1}((u,v))+[-h_1-bv]_{+}m_{1}((u+1,v))=\sum\limits_{(u\',v\')\in J^{(1)}_{h}}c_{u\',v\'}m_{1}((u-u\',v-v\')),\]
  where in the right-hand side, $m_{1}((u-u\',v-v\'))$ is calculated in $N_{(h_{1}+1+bv\',h_{2}-cu\')}$. Moreover, $m_{1}((u-u\',v-v\'))\geqslant0$. Hence $m_{1}((u,v))>0$ or $m_{1}((u+1,v))>0$ is equivalent to $m_{1}(u-u\',v-v\')>0$ in $N_{(h_{1}+1+bv\',h_{2}-cu\')}$ for some $(u\',v\')\in J^{(1)}_{h}$.
\end{Proof}

\begin{Lemma}\label{m_j in rank 2}
  Let $\A$ be a cluster algebra with principal coefficients of rank $2$ with initial exchange matrix
  $B=\begin{pmatrix}
      0 & b \\
      -c & 0
    \end{pmatrix}.$
  Assume $h\in\Z^{2}$, $j\neq\bar{j}\in\{1,2\}$ and $p=(u,v)\in N_{h}$. Then,

  (i)\; Assume $m_{j}(p)>0$ and $p\'=p+(-1)^{\bar{j}}ae_{\bar{j}}\in N_{h}$ for $a\in\N$, then $m_{j}(p\')>0$ if $co_{p\'}(N_{h})>0$.

  (ii) Assume $m_{j}(p)>0$ and $p^{\prime\prime}=p+(-1)^{\bar{j}}e_{j}\in N_{h}$, then $m_{j}(p^{\prime\prime})>0$.
  Moreover, $N_{h}$ is in the rectangular with vertices
  $$(0,0), \;\; ([-h_{1}]_{+},0), \;\; ([-h_{1}]_{+},[-h_{2}+c[-h_{1}]_{+}]_{+})\quad\text{and}\quad(0,[-h_{2}+c[-h_{1}]_{+}]_{+}),$$
  which would degenerate to a segment or a point when some vertices are coincident.

  (iii)\; $N_{h}\setminus\{([-h_{1}]_{+},0)\}$ can be divided into three connected areas as the form in Figure \ref{figure: m_j}, where
  \begin{itemize}
    \item any point $p$ in Area $\uppercase\expandafter{\romannumeral1}$ (in the right part) satisfies $m_{1}(p)\neq 0$ and $m_{2}(p)=0$;
    \item any point $p$ in Area $\uppercase\expandafter{\romannumeral2}$ (Area $\uppercase\expandafter{\romannumeral2}$ may not contain any point) satisfies $m_{1}(p)=m_{2}(p)=0$;
    \item any point $p$ in Area $\uppercase\expandafter{\romannumeral3}$ (in the lower part) satisfies $m_{1}(p)=0$ and $m_{2}(p)\neq0$;
  \end{itemize}
  and $m_{1}([-h_{1}]_{+},0)=m_{2}([-h_{1}]_{+},0)=1$.
\end{Lemma}
\begin{figure}[H]
  \centering
  \includegraphics[width=25mm]{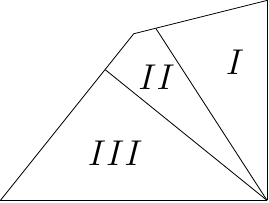}
  \caption{The polytope $N_{h}$ can be divided into three areas.}\label{figure: m_j}
\end{figure}

\begin{Proof}
When $h_{1}\geqslant 0$ or $h_{2}+ch_{1}\geqslant 0$, $N_{h}$ degenerates to a segment or a point, so it is easy to see the lemma holds. Hence now we only need to consider the case where $h_{1}<0$ and $h_{2}+ch_{1}<0$.

  It is shown in Theorem 1.7 of \cite{LLZ} that the elements of the greedy basis are universally indecomposable and that the greedy basis is independent of the choice of the initial seed. Therefore by Proposition \ref{polytope basis equals greedy basis}, the definition of $L^{t_{k}}$ and the fact that $B$ is of full rank in rank 2 case, there is $\alpha,h^{t_{k}}\in\Z^{2}$ such that $\hat{Y}_{t_{k}}^{\alpha}\rho^{t_{k}}_{h^{t_{k}}}=L^{t_k}(\rho_{h})$, which means the Newton polytope of the Laurent expression of $L^{t_k}(\rho_{h})$ in $X_{t_{k}}$ can be obtained from $N^{t_{k}}_{h^{t_{k}}}$ by a translation along $\alpha$. Thus we get an isomorphism $\sigma_{k}$ between two polytopes via the translation along $\alpha$.

  (i) \;Let $j=1$, then $m_1(p)>0$. For $p\'=p+(-1)^{2}ae_{2}=(u,v+a)$, on the contrary suppose $co_{(u,v+a)}(N_{h})>0$ but $m_{1}((u,v+a))=0$ for some $a\in\N$. If $deg_{x_{1}}(\hat{y}_1^{u}\hat{y}_2^{v+a}X^h)\geqslant0$, then $m_{1}((u,v+a))=co_{(u,v+a)}(N_{h})$, which is impossible. So $deg_{x_{1}}(\hat{y}_1^{u}\hat{y}_2^{v+a}X^h)<0$. Without loss of generality, assume $m_{1}(u,v\')=0$ if $v\'\in[v+1,v+a]$. Then according to the isomorphism $\sigma_{1}$ and the discussion in Remark \ref{mutation of P} about how the Laurent expression is changed under mutations, there is a point $(r,s)$ in $N^{t_{1}}_{h^{t_{1}}}$ correlated to $(u,v)$ under mutation $\mu_{1}$ satisfying $co_{(r,s)}(N^{t_{1}}_{h^{t_{1}}})=m_1((r,s))>0$, $co_{(r,s\')}(N^{t_{1}}_{h^{t_{1}}})=0$ for any $s\'\in[s+1,s+a]$. Then due to the construction of $N^{t_{1}}_{h^{t_{1}}}$, $m_{2}((r,s))>0$. So by Lemma \ref{fact} (a), $(u,v)$ must be $(-h_{1},0)$, then $m_{1}((u,v+a))=co_{(u,v+a)}(N_{h})>0$ according to the definition of $N_h$, which contradicts our assumption.

  Similarly for $j=2$. So $m_{j}(p\')>0$ if $m_{j}(p)>0$ and $co_{p\'}(N_{h})>0$ for $j=1,2$.

  (ii)\;We  prove (ii) by induction on $(-h_{1},-h_{2}-ch_{1})$ with respect to the canonical partial order.

  When $(-h_{1},-h_{2}-ch_{1})=(1,1)$, $N_{h}$ is as follows
  \begin{equation*}
    \begin{array}{cc}
       & 1 \\
      1 & 1
    \end{array}
  \end{equation*}
  so trivially the second half of (ii) holds. And $m_{1}((0,0))=m_{2}((1,1))=0$ while $m_{1}((1,0))=m_{1}((1,1))=m_{2}((0,0))=m_{2}((0,1))=1$. Then we can see easily that there is no point $p\in N_{h}$ satisfying the assumption condition of (ii) indeed, thus the first half of (ii) is true. So (ii) holds.

  Assume (ii) holds for all $(-h_{1},-h_{2}-ch_{1})<(l_{1},l_{2})\in\Z^{2}_{>0}$. Then we need to prove (ii) when $(-h_{1},-h_{2}-ch_{1})=(l_{1},l_{2})$.
  First we discuss the first part of (ii). On the contrary, suppose
  \begin{equation}\label{contrary}
    m_{j}(p)>0\;\;\;\text{ but}\;\;\; m_{j}(p^{\prime\prime})=0
      \end{equation}
  for some $p=(u,v),p^{\prime\prime}=p+(-1)^{\bar{j}}e_{j}\in N_h$ and $j\in\{1,2\}$. Here we choose $p$ to be a point closest to $(-h_{1},0)$ (which implies the minimality of $(-h_{1}-u)^{2}+v^2$) satisfying the conditions in (\ref{contrary}). Also we can assume $deg_{x_{j}}(\hat{y}_1^u\hat{y}_2^vX^h)\leqslant0$, otherwise we consider $N^{t_j}_{h^{t_j}}$ instead.

  It is enough to deal with the case where $j=1$ because of the duality of indices 1 and 2.

  According to (\ref{equation: decomposition of x2Ph}), Lemma \ref{fact} (b) and the inductive assumption, there is $(u\'-[-h_{1}-bv\']_+,v\')\in J^{(2)}_{h}$ such that $u\'\geqslant u\geqslant u\'-[-h_{1}-bv\']_+$, $v\'\leqslant v$ and $m_{1}((u-u\'+[-h_{1}-bv\']_+,v-v\'))>0$ in $N_{(h_{1}+bv\',h_{2}+1-cu-ch_{1}-bcv\')}$. If $u\'>u$, then $u\'\geqslant u+1$ and so $(u+1-u\'+[-h_{1}-bv\']_+,v-v\')$ lies in $N_{(h_{1}+bv\',h_{2}+1-cu-ch_{1}-bcv\')}$ because of the inductive assumption for this polytope, which means that in the right-hand side of (\ref{equation: decomposition of x2Ph}), there is a summand whose corresponding polytope $N'=N_{(h_{1}+bv\',h_{2}+1-cu-ch_{1}-bcv\')}[((u\'-[-h_{1}-bv\']_+,v\'))]$ contains both $(u,v)$ and $(u+1,v)$, then by inductive assumption $m_1(p'')=m_1((u+1,v))>0$ in $N'$, therefore $m_1((u+1,v))>0$ in $N_{h}$ because $N'$ is a summand in the decomposition of $N_h$ in (\ref{equation: decomposition of x2Ph}), which contradicts the assumption. So $u\'=u$.

  Denote by $V$ the set consisting of integers $v\'$ such that $(u-[-h_{1}-bv\']_+,v\')\in J^{(2)}_{h}$ satisfying $v\'\leqslant v$ and $m_{1}((u-u\'+[-h_{1}-bv\']_+,v-v\'))>0$ in $N_{(h_{1}+bv\',h_{2}+1-cu-ch_{1}-bcv\')}$. $(u-[-h_{1}-bv\']_+,v\')\in J^{(2)}_{h}$ induces $m_{1}((u,v\'))>0$ and $m_{2}((u,v\'-1))>0$ in $N_h$.

  Next we case by case show that the assumption will lead to a contradiction:

  (a)\; The case where there is $v\'\in V$ such that $v\'<v$ and $co_{(u+1,v)}(N_h)>0$.

  In this case, $m_{1}((u+1,v\'))>0$, since otherwise $(u,v\')$ becomes a closer point satisfying (\ref{contrary}), which contradicts our choice of $(u,v)$. However, $m_{1}((u+1,v))=0$ by assumption and $co_{(u+1,v)}(N_h)>0$, which contradicts to (i).

  (b)\; The case where there is $v\'\in V$ such that $v\'<v$ and $co_{(u+1,v)}(N_h)=0$.

  In this case, $m_{1}((u+1,v\'))>0$ as showed in (a). However $m_{1}((u+1,v))=co_{(u+1,v)}(N_h)=0$ and $deg_{x_{1}}(\hat{y}_1^{u+1}\hat{y}_2^vX^h)<0$. So similar to the proof of (i), we can find some $(r,s)\in N^{t_1}_{h^t_1}$ satisfying $m_{1}(r,s)>0$ in $N^{t_1}_{h^t_1}$ but $co_{(r,s+1)}(N^{t_1}_{h^{t_1}})=0$, which leads to $m_{2}(r,s)>0$ in $N^{t_1}_{h^t_1}$ and thus contradicts Lemma \ref{fact} (a).

  (c)\; The case where $V=\{v\}$ and $co_{(u+1,v)}(N_h)>0$.

  In this case, $m_2((u,v-1))>0$ in $N_h$, and $m_{1}((u,v))=co_{(u,v)}(N_h)=m_2((u,v-1))\tilde{C}_{-deg_{x_{2}}(\hat{y}_1^u\hat{y}_2^vX^h)}^{1}$. Comparing with (\ref{equation: weight of rank2}), we get $deg_{x_{2}}(\hat{y}_1^u\hat{y}_2^vX^h)=-1$. So $deg_{x_{2}}(\hat{y}_1^{u-i}\hat{y}_2^vX^h)\geqslant$ for any $i\in[1,u]$, which induces $m_{1}(u-i,v)=0$.

  Let $(r,s)$ be the unique minimal point in $N^{t_{1}}_{h^{t_{1}}}$ correlated to $(u,v)$ under $\mu_{1}$. Then $co_{(r-1,s)}(N^{t_{1}}_{h^{t_{1}}})=0$ as $m_{1}((u+1,v))=0$. Applying (\ref{equation: decomposition of x1Ph}) to $\rho^{t_{1}}_{h^{t_{1}}}$, we get an decomposition of $N^{t_{1}}_{h^{t_{1}}}$ as a sum of some sub-polytopes. And in particular there is a summand $N^{t_{1}}_{h^{t_{1}}+e_{1}}$. As we discussed before, $\hat{Y}_{t_{1}}^{\alpha}\rho^{t_{1}}_{h^{t_{1}}}=L^{t_1}(\rho_{h})$, so $\hat{Y}_{t_{1}}^{\alpha\'}\rho^{t_{1}}_{h^{t_{1}}}x_{1;t_{1}}=L^{t_1}(\rho_{h}\frac{M_{1}}{x_{1}})$ for some $\alpha\'\in\Z^{n}$. Then as a summand we get $\hat{Y}_{t_{1}}^{\alpha^{\prime\prime}}\rho^{t_{1}}_{h^{t_{1}}+e_{1}}=L^{t_1}(\rho_{h+(-1,b)})$ for some $\alpha^{\prime\prime}\in\Z^{n}$, which means the Newton polytope of the Laurent expression of $L^{t_1}(\rho_{h+(-1,b)})$ in $X_{t_{1}}$ can be obtained from $N^{t_{1}}_{h^{t_{1}}+e_{1}}$ by a translation.

  Moreover, according to (\ref{equation: decomposition of x1Ph}), $y_{1}\rho_{h}$ is a summand of $x_{1}\rho_{h+(-1,b)}$. Then we get
  \begin{equation}\label{equation: c}
    co_{u+1+deg_{x_{1}}(\hat{y}_1^u\hat{y}_2^vX^h)}(N_{h+(-1,b)})>0
  \end{equation}
  as $m_1((u,v))>0$ in $N_{h}$. Therefore according to the discussion in section 2 about the change of a Laurent expression under mutations, (\ref{equation: c}) induces $co_{(r,s)}(N^{t_{1}}_{h^{t_{1}}+e_{1}})>0$.

  By iteratively repeating the above discussion we have $co_{(r,s)}(N^{t_{1}}_{h^{t_{1}}+ie_{1}})>0$ for $i\in\N$, but it fails definitely when $i$ is large enough due to the construction of these polytopes, which induces a contradiction.

  (d)\;In the case where $V=\{v\}$ and $co_{(u+1,v)}(N_h)=0$.

  In this case, $m_{1}((u,v-i))=0$ for any $i\in[1,v]$, since otherwise there is $j\in[1,v]$ such that $(u,v-j)$ is a closer point to satisfy (\ref{contrary}), contradicting our choice of $(u,v)$. If $m_{1}((u+1,v-i))>0$ for some $i\in[1,v]$, then it leads to a contradiction by a discussion similar to that in case (b). Therefore $m_{1}((u+1,v-i))=0$ for any $i\in[1,v]$. Because $co_{(u+1,v)}(N_h)=0$, so in the right-hand side of (\ref{equation: decomposition of x1Ph}), we can fined $(u^{\prime\prime},v^{\prime\prime})\in J^{(1)}_{h}$ such that $m_{1}((u,v))>0$ in the polytope $N_{(h_{1}+1+bv^{\prime\prime},h_{2}-cu^{\prime\prime})}[(u^{\prime\prime},v^{\prime\prime})]$. Due to the definition of $J^{(1)}_{h}$, $m_{1}((u^{\prime\prime},v^{\prime\prime}))>0$ in $N_{h}$. So $u^{\prime\prime}\geqslant u+2$. Hence by the inductive assumption for polytope $N_{(h_{1}+1+bv^{\prime\prime},h_{2}-cu^{\prime\prime})}$, $(u+1,v)\in N_{(h_{1}+1+bv^{\prime\prime},h_{2}-cu^{\prime\prime})}[(u^{\prime\prime},v^{\prime\prime})]$. Then due to the inductive assumption for polytope $N_{(h_{1}+1+bv^{\prime\prime},h_{2}-cu^{\prime\prime})}$, $m_{1}((u+1,v))>0$ in $N_{(h_{1}+1+bv^{\prime\prime},h_{2}-cu^{\prime\prime})}[(u^{\prime\prime},v^{\prime\prime})]$, which leads to $co_{(u+1,v)}(N_{h})\geqslant co_{(u+1,v)}(N_{(h_{1}+1+bv^{\prime\prime},h_{2}-cu^{\prime\prime})}[(u^{\prime\prime},v^{\prime\prime})])>0$, contradicting $co_{(u+1,v)}(N_h)=0$.

  In summary, $m_{1}(p^{\prime\prime})>0$ when $m_{1}(p)>0$. Similarly, it can be proved for $j=2$ case.

  Now we discuss the second part. Note that according to the definition of $N_{h}$, if $(u,0),(-h_{1},v)\in N_{h}$, then $u\in[0,-h_{1}]$ and $v\in[0,-h_{2}-ch_{1}]$. Suppose there is $(u,v)\in N _{h}$ with $v>-h_{2}-ch_{1}$. Without loss of generality, we assume $u=max\{w|(w,v)\in N_{h}\}$. Then $u<-h_{1}$, and $m_{1}((u,v))>0$. So according to the first part of (ii), we can see that $m_{1}((u+1,v))>0$, which leads to $(u+1,v)\in N_{h}$, contradicting the choice of $u$. So for any $(u,v)\in N _{h}$, $v\leqslant -h_{2}-ch_{1}$. Dually, it can also be proved similarly for any $(u,v)\in N _{h}$, $u\geqslant 0$. Therefore $N_{h}$ is in the unique rectangular with vertices
   $$(0,0),\;\; ([-h_{1}]_{+},0),\;\; (0,[-h_{2}+c[-h_{1}]_{+}]_{+}),\;\; ([-h_{1}]_{+},[-h_{2}+c[-h_{1}]_{+}]_{+}).$$

  In conclusion, by induction we have proved (ii).

  (iii) \;The shape of $N_{h}$ is of the form in Figure \ref{figure: m_j} according to (ii) and the definition of $N_{h}$. While
  $$m_{1}([-h_{1}]_{+},0)=m_{2}([-h_{1}]_{+},0)=1$$
  can be calculated directly from the definition of $N_{h}$. Moreover, because of (i), we can see that if there is $(u,v)\in N_h$ satisfying $m_{1}((u,v))>0$, then $m_{1}((u\',v))>0$ when $u\'>u$ and $(u\',v)\in N_h$. So there is a left bound of Area $\uppercase\expandafter{\romannumeral1}$ such that all points in $N_h$ lying in the right of this bound is in Area $\uppercase\expandafter{\romannumeral1}$, i.e., Area $\uppercase\expandafter{\romannumeral1}$ lies in the right part of $N_h$ as showed in Figure \ref{figure: m_j}. Similarly, Area $\uppercase\expandafter{\romannumeral3}$ lies in the lower part by (i) and the rest is Area $\uppercase\expandafter{\romannumeral2}$.
\end{Proof}

\begin{Lemma}\label{face for rank 2}
  Let $\A$ be a cluster algebra with principal coefficients of rank $2$ in the above setting. Assume $h\in\Z^{2}$ and $t_{k}\in\T_{n}$ is a vertex connected to $t_{0}$ by the edge labeled $k\in[1,2]$. Suppose the dimension of $N_{h}$ is not zero, i.e., at least one of $h_{1}$ and $h_{2}$ is negative, and a segment $l$ connecting $p_{1},p_{2}$ is an edge of $N_{h}$. Then,

  (i)\;$\rho^{t}_{h}\in\N[Y_{t}][X_{t}^{\pm 1}]$ is universally indecomposable and $\{L^{t_{0}}(\rho^{t}_{h})\mid h\in\Z^{2}\}$ (see the definition in (\ref{L(f)})) is independent of the choice of $t$. Moreover, $L^{t_{k}}(\rho_{h})=\rho^{t_{k}}_{h^{t_{k}}}$, where $$h^{t_{k}}=h-2h_{k}e_{k}+h_{k}[b^{t}_{k}]_{+}+[-h_{k}]_{+}b^{t}_{k}$$
  for any $k\in[1,2]$.

  (ii)\;Up to interchanging $p_{1}$ and $p_{2}$, $p_{1}>p_{2}$ and there is $s\in\Z_{>0}$ such that $l$ is isomorphic to $N_{(0,-s)}$.
\end{Lemma}
\begin{Proof}
  (i)\;It is claimed in theorem 1.7 of \cite{LLZ} that the elements of the greedy basis are universally indecomposable and that the greedy basis is independent of the choice of the initial seed. Therefore as said in the proof of last lemma, by Proposition \ref{polytope basis equals greedy basis}, the definition of $L^{t_{k}}$ and the fact that $B$ is of full rank in rank 2 case, $\rho^{t}_{h}$ is universally indecomposable and there is $\alpha,h\'\in\Z^{2}$ such that $L^{t_{k}}(\rho_{h})=Y_{t_{k}}^{\alpha}\rho^{t_{k}}_{h\'}$, where $\rho^{t_{k}}_{h\'}$ is defined associated to $h\'$ with initial seed being $\Sigma_{t_{k}}$.

  On the other hand, by lemma \ref{m_j in rank 2} (ii), the definition of $L^{t_{k}}$ and the mutation formula of $Y$-variables, it can be checked that $X_{t_{k}}^{h^{t_{k}}}$ is a summand of $L^{t_{k}}(M_{k;t_{0}}^{[-h_{k}]_{+}}X^{h})$, while the latter is a summand of $L^{t_{k}}(\rho_{h})$. Therefore $\alpha=0$ and $h\'=h^{t_{k}}$. And thus $\{L^{t_{0}}(\rho^{t}_{h})\mid h\in\Z^{2}\}$ is independent of the choice of $t$.

  (ii)\;Denote by $h^{t}\in\Z^{2}$ for any $t\in\T_{2}$ such that $L^{t}(\rho_{h})=\rho^{t}_{h^{t}}$ and by $(u_{t},v_{t})$ the unique maximal point in $N^{t}_{h^{t}}$. Because $u_{t},v_{t}\geqslant 0$, there is a vertex $t\in\T_{2}$ such that $(u_{t},v_{t})\leqslant(u_{t\'},v_{t\'})$ for any $t\'$ connected to $t$. It can be check that such $(u_{t},v_{t})$ is unique.

  We may assume $t_{0}$ satisfies the above condition. In the sequel of this proof, we always denote $u=u_{t_{0}},v=v_{t_{0}}$.

  If $(u,v)=(0,0)$, then $N_{h}$ is the origin. According to the mutation formula (\ref{equation: mutation of x}) and (\ref{equation: mutation of y}), two Laurent monomials corresponding to two points in a proper face (i.e., a face with dimension in $[0,dim(N_{h^{t}}^{t})-1]$) of $N_{h^{t}}^{t}$ respectively must correlate to two Laurnt monomials corresponding to two points in some proper face of $N_{h^{t\'}}^{t\'}$ for any two connected vertices $t,t\'\in\T_{2}$. So for any $l$ which is a face of $N_{h^{t}}^{t}$ with dimension $1$ for $t\in\T_{2}$, $l$ is either in $E_{h^{t}}^{t}$ or is isomorphic to an edge in $E_{h^{t\'}}^{t\'}$ for some $t\'\in\T_{2}$. And according to the definition of polytope $N_{h}$ for $h\in\Z^2$, the (ii) holds for edges in the essential skeleton. Therefore (ii) holds when there is $t\in\T_{2}$ such that $N_{h^{t}}^{t}$ is a point.

  If $(u,v)\neq(0,0)$, then according to the mutation formula (\ref{equation: mutation of x}) and (\ref{equation: mutation of y}), we have $u_{t\'}=u$ and $v_{t\'}=cu-v$, where $t\'$ is connected to $t_{0}$ by an edge labeled 2. So by our assumption that $v_{t\'}\geqslant v$, we have $2v\leqslant cu$. Similarly, $u_{t^{\prime\prime}}=bv-u$ and $v_{t^{\prime\prime}}=v$, where $t^{\prime\prime}$ is connected to $t_{0}$ by an edge labeled 1, hence we have $2u\leqslant bv$.

  First assume $2v= cu$, and the Newton polytope is of the following form as Figure \ref{figure: face for rank 2},
  \begin{figure}[H]
    \centering
    \includegraphics[width=30mm]{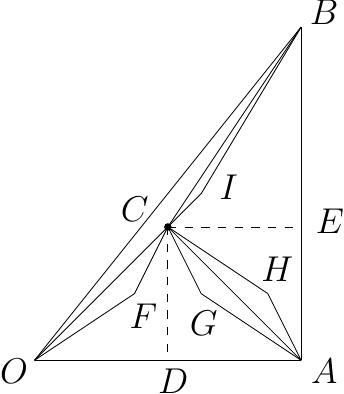}
    \caption{The Newton polytope $N_{h}$ when $2v=cu$.}\label{figure: face for rank 2}
  \end{figure}
  \noindent where Laurent monomials corresponding to points lying at the left (resp. right) side of the dashed line $CD$ have positive (resp. negative) exponents of $x_{2}$ while those corresponding to points lying above (below) the dashed line $CE$ have positive (resp. negative) exponents of $x_{1}$. Let $AGCFO$ (resp. $ABICH$) denotes the Area $III$ (resp. Area $I$) in Lemma \ref{m_j in rank 2}. Hence for any line $z_{2}=r$ below $CE$, it intersects with $OF$ (or $FC$) and $AH$ (or $HC$) at two points (not necessarily lattice points) and the length of the segment connecting these points is $u-br$.

  On the other hand, $z_{2}=r$ also intersects with $AC$ and $OC$ at two points (not necessarily lattice points) and the length of the segment connecting these points is $u-br$. Moreover, since $2v=u$, $N_{h^{t_{2}}}^{t_{2}}=N_{h}$ up to reflection and translation, so $OFC$ is symmetric with $AGC$ over $CD$. Therefore, $OFC$ lies below $OC$ and $AHC$ lies above $AC$.

  And in general for $2v\leqslant cu$, according to (\ref{equation: decomposition of x2Ph}), Area $I$ for $2v\leqslant cu$ case lies in $ABICH$, which leads to that in this case the new bound $OF\'C\'$ lies below $OC\'$, where the letters with prime mean those in $2v\leqslant cu$ case.

  Dually, we also have that $BIC$ lies below $BC$.

  Therefore, we get that $N_{h}=\Delta_{OAB}$, where $O$ and $B$ are the only two points in $OB$ with nonzero weights, which equal to 1. Again according to mutation formula (\ref{equation: mutation of x}) and (\ref{equation: mutation of y}), we see that for any $l$ which is a face of $N_{h^{t}}^{t}$ with dimension $1$ for $t\in\T_{2}$, $l$ is either in $E_{h^{t}}^{t}$ or is isomorphic to $OB$, which is isomorphic to $N_{(0,-1)}$.
\end{Proof}

\begin{Remark}\label{shape for rank 2}
  According to the discussion in above proof, we know that for any $h\in\Z^{2}$, $N_{h}$ is one of the following forms in Figure \ref{figure: shape for rank 2} (up to reflections over the line $z_{1}=z_{2}$):

  Figure \ref{figure: shape for rank 2}.(1): in the case where $h_{1},h_{2}\geqslant 0$.

  Figure \ref{figure: shape for rank 2}.(2): in the case where $h_{2}\geqslant 0$, $h_{1}<0$ and $-h_{2}-ch_{1}\leqslant 0$.

  Figure \ref{figure: shape for rank 2}.(3)-(4): in the case where $h_{1},h_{2}<0$.

  Figure \ref{figure: shape for rank 2}.(5)-(7): in the case where $h_{2}\geqslant 0$, $h_{1}<0$ and $-h_{2}-ch_{1}>0$.
  \begin{figure}[H]
    \centering
    \includegraphics[width=130mm]{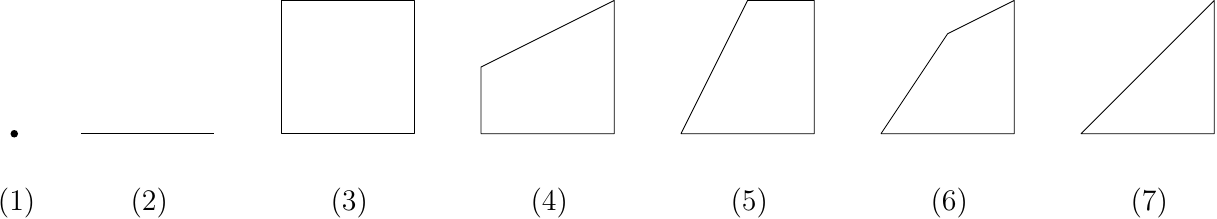}
    \caption{The shape of $N_{h}$ for $h\in\Z^{2}$ up to reflections.}\label{figure: shape for rank 2}
  \end{figure}
\end{Remark}

\begin{Remark}
  In the perspective of greedy basis, Lee, Li Li and Zelevinsky present a precise description of $support(N_h)$ for any $h\in\Z^{2}$ in \cite{LLZ,LLZ2}, which also leads to Lemma \ref{face for rank 2} (ii) and Remark \ref{shape for rank 2}.
\end{Remark}

\begin{Lemma}\label{unique indecomposable in rank 2}
  Let $\A$ be a cluster algebra with principal coefficients of rank $2$ in the above setting, $h\in\Z^{2}$. Then $\rho_{h}$ is the unique indecomposable Laurent polynomial in $X$ in $\mathcal{U}^{+}_{\geqslant0}(\Sigma_{t_{0}})$ which has $X^{h}$ as a summand and whose support is contained in $supp(N_{h})$.
\end{Lemma}
\begin{Proof}
  According to Lemma \ref{face for rank 2} (i) and the definition of $\rho_{h}$ for any $h\in\Z^{2}$, $\rho_{h}$ is an indecomposable Laurent polynomial in $\mathcal{U}^{+}_{\geqslant0}(\Sigma_{t_{0}})$ which has $X^{h}$ as a summand and whose support is contained in $supp(N_{h})$. Then we only need to show the uniqueness.

  Let ``$\prec$'' be the partial order defined in the proof of Proposition \ref{polytope basis equals greedy basis}.

  Assume $f\neq\rho_{h}$ is another one satisfying the conditions in this lemma and $N(f)$ is its corresponding polytope. We choose $p$ to be a minimal point with respective to ``$\prec$'' in $N_{h}$ satisfying $co_{p}(N_{h})\neq co_{p}(N(f))$. Since $f\neq \rho_{h}$, such point exists.

  If $p\neq ([-h_{1}]_{+},0)$, then $co_{q}(N(f))=co_{q}(N_h)$ for any $q\prec p$. Because $f\in\mathcal{U}^{+}_{\geqslant0}(\Sigma_{t_{0}})$, $f$ can be expressed as a Laurent polynomial in $X_{t_{1}}$ and $X_{t_{2}}$ with non-negative constant coefficients respectively, where $t_{i}$ is the vertex connected to $t_0$ by an edge labeled $i$ for $i=1,2$, so $co_{p}(N(f))>co_{p}(N_h)$ according to (\ref{equation: weight of rank2}) and (\ref{equation: m(p)}) and thus $p-[-deg_{x_{1}}(\hat{Y}^{p}X^{h})]_{+}e_{1},p+[-deg_{x_{2}}(\hat{Y}^{p}X^{h})]_{+}e_{2}\in supp(N(f))$ following a discussion similar to that in Subsection 3.1. On the other hand, due to Lemma \ref{fact} (a), either $m_{1}(p)=0$ or $m_{2}(p)=0$ in $N_{h}$, which together with Lemma \ref{m_j in rank 2} induces one of $p-[-deg_{x_{1}}(\hat{Y}^{p}X^{h})]_{+}e_{1}$ and $p+[-deg_{x_{2}}(\hat{Y}^{p}X^{h})]_{+}e_{2}$ do not lying in $supp(N_{h})$. Hence $supp(N(f))\nsubseteq supp(N_{h})$, contradicting the assumption $supp(N(f))\subseteq supp(N_{h})$.

  If $p=([-h_{1}]_{+},0)$, then compare $N(f)$ and $co_{p}(N(f))N_{h}$ instead. Then either we get $N(f)=co_{p}(N(f))N_{h}$, that is, $f=co_{p}(N(f))\rho_{h}$ for $co_{p}(N(f))>1$, which contradicts the indecomposability of $f$, or similar to the above discussion, there is a point in $supp(N(f))$ but not in $supp(N_{h})$, which leads to $supp(N(f))\nsubseteq supp(N_{h})$, contradicting the assumption $supp(N(f))\subseteq supp(N_{h})$.

  Therefore in conclusion, $\rho_h$ is the unique one satisfying the conditions in the lemma.
\end{Proof}

\subsection{Construction program of $N_h$ and $\rho_h$}\quad

In general case, the process of this construction is pretty complex. So our strategy is using sections as well as formulas similar  to (\ref{equation: decomposition of x1Ph}) and (\ref{equation: decomposition of x2Ph}) in general rank case to provide an inductive construction. It may not be able to calculate concrete coefficients conveniently. But, this inductive algorithm already contains much information about cluster algebras, especially about cluster monomials.

In the definition of $\rho_{h}$, the combinatorial objects of a polytope, such as sections and faces, help us to understand the structure of $\rho_{h}$ better. It is why we introduce polytopes corresponding to homogeneous Laurent polynomials in our study.

\begin{Definition}\label{essential}\textbf{(Essential skeleton)}
  Let $\A$ be a cluster algebra of rank $n$ with initial exchange matrix $B=(b_{ij})_{i,j\in [1,n]}$ and initial cluster $X$.

  (a)\; For any Laurent monomial $p$ in $X$, subset $I\subseteq[1,n]$, set of subsets $\mathcal{K}=\{K^{(1)}_l\subseteq I|l\in I\}$, vector $\lambda_0=(\lambda_{0i})_{i\in I}\in\{-1,1\}^{|I|}$ and matrix $\Lambda=(\lambda_{uv})\in Mat_{|I_1|\times|I|}(\{-1,1\})$, define a set $E(p,I,\mathcal{K},\lambda_0,\Lambda)$ to be the maximal set which is the union of segments $\overline{q_{r-1}q_{r}}$ for any sequence of Laurent monomials $q_{0},\cdots,q_{r}=p\'$ in $X$ and any sequence of indices $i_{1},\cdots,i_{r}$ with $i_{j}\in I_j$ for any $j\in[1,r]$ satisfying:

  (i) $q_{0}=p$, $i_0=\emptyset$ and $\epsilon_{k}^{(1)}=\lambda_{0k}$ for any $k\in I$. $K_{j}=\{k\in I|\text{There is a sequence }l_1,\cdots,l_s=k\text{ of }I\text{ such that }deg_{x_{l_1}}(q_{j-1})<0\text{ and }l_t\in K^{(j)}_{l_{t-1}}\text{ for any }t\in[2,s]\}$ and $I_{j}=\{s\in I|deg_{x_{s}}(q_{j-1})<0,s\notin K_{j}\}$.

  (ii) $i_{j}\in I_j\setminus\{i_{j-1}\}$, $\epsilon_{j}=\epsilon_{i_j}^{(j)}$ and $q_{j}=q_{j-1}\hat{y}_{i_{j}}^{-\epsilon_{j}deg_{x_{i_{j}}}(q_{j-1})}$ is a vertex of $conv(E(p,I,\mathcal{K},\lambda_0,\Lambda))$ for any $j\in[1,r]$.

  (iii) \begin{equation}\label{equation: essential skeleton}
          \epsilon_{k}^{(j+1)}=\left\{\begin{array}{ll}
                          -\epsilon_{k}^{(j)}, & \text{if $k=i_{j}$;} \\
                          \lambda_{i_1k}, & \text{if $j=1$ and $k\neq i_j$;} \\
                          \epsilon_{k}^{(j)}, & \text{if $j>1$ and either $deg_{x_k}(q_{j-1})<0$ or $k\in K_{j}$;} \\
                          \overline{sign}(b_{ki_{j-1}})\overline{sign}({b_{ki_j}})\epsilon_{k}^{(j)},  & \text{otherwise.}
                        \end{array}\right.
        \end{equation}
  for any $k\in I$, where $\overline{sign}(b_{ks})=sign(b_{ks}+0.5)$.

  (iv) $K^{(j+1)}_{i_j}=\{k\in I| b_{ki_j}\epsilon_{i_j}^{(j+1)}<0,b_{i_j k}\epsilon_{k}^{(j+1)}<0\}$, $K^{(j+1)}_{l}=\{k\in I\setminus\{i_j\}| b_{kl}\epsilon_{l}^{(j+1)}<0,b_{lk}\epsilon_{k}^{(j+1)}<0, \text{ either }\epsilon_l^{(j+1)}=\epsilon_l^{(j)}\text{ and }k\in K^{(j)}_l\text{ or }\epsilon_l^{(j+1)}\neq\epsilon_l^{(j)}\}$ for any $l\in I\setminus\{i_j\}$.

  $\overline{q_{j-1}q_{j}}$ in the above sequence is parallel to $e_{i_{j}}$ for any $j\in[1,r]$, so we label it by $i_{j}$.

  (b)\; We define the \textbf{essential skeleton} $E_{h}$ to be $E(X^{h},[1,n],\{\emptyset,\cdots,\emptyset\},\sum\limits_{i=1}^ne_i,\Lambda)$ with $\lambda_{ij}=1$ for any $i,j\in[1,n]$. Denote by $V_{h}$ the set consisting of vertices of the segments in $E_{h}$.
\end{Definition}

\begin{Remark}
(i)\;The definition of essential skeleton essentially follows the idea explained in Subsection 3.1. We start with $X^h$ and iteratively adding edges into $E_{h}$ corresponding to complements of vertices already in $E_{h}$. $\epsilon_{j}$ is chosen so that the polytope $N_{h}$ constructed below based on $E_{h}$ can satisfy the mutation formula in Theorem \ref{properties in case tsss} (iii) in the sequel.

(ii)\; For any vertex $q_j$ of $E(p,I,\mathcal{K},\lambda_0,\Lambda)$ with a sequence $q_0,\cdots,q_r$, $E(p,I,\mathcal{K},\lambda_0,\Lambda)$ equals\\ $E(q_j,I,\{K_l^{(j+1)}|l\in I\},(\epsilon_l^{(j+1)})_{l\in I},(\lambda_{uv})_{u\in I_{j+1},v\in I})$, where for each $u$, we may choose the sequence such that either $\overline{q_{j-1}q_j}$ or $\overline{q_jq_{j+1}}$ parallels to $e_u$, and let
\[\lambda_{uv}=\left\{\begin{array}{ll}
                        \epsilon_{v}^{(j+1)}, & \text{if $\overline{q_jq_{j+1}}$ parallels to $e_u$}; \\
                        -\epsilon_{v}^{(j)}, & \text{if $\overline{q_{j-1}q_j}$ parallels to $e_u$, }deg_{x_v}(q_{j-1})\geqslant0, v\notin K_j\text{ and }b_{i_{j-1}v}b_{i_j v}<0; \\
                        \epsilon_{v}^{(j)}, & \text{otherwise}.
                      \end{array}\right.\]
\end{Remark}

Note that there is not necessarily a sequence which satisfies above conditions containing all vertices in $V_{h}$.

It can be verified that for any subset $I\'\subseteq [1,n]$, if we delete all segments in $E_{h}$ with labels not in $I$, then there are unique maximal and minimal points in every connected component respectively because of the sign-skew-symmetry of the exchange matrix by induction on $|[1,n]\setminus I\'|$.

\begin{Example}\label{example of essential skeleton}
  (i)\; Let $\A$ be a cluster algebra having principal coefficients with the initial seed $(X,Y,B)$, where $X=(x_{1},x_{2},x_{3})$, $Y=(y_{1},y_{2},y_{3})$ and
  \[B=\left(\begin{array}{ccc}
  0 & 2 & -1 \\
  -1 & 0 & 2 \\
  1 & -4 & 0
  \end{array}\right).\]
  Following the notions and notations of Definition \ref{essential}, the essential skeleton $E_{h}$ for $h=(-5,2,4)$ can be calculated as follows:\\
  (a)\; $q_{0}=x_{1}^{-5}x_{2}^{2}x_{3}^{4}$, $K_1=\emptyset$, $i_{1}=1$ and $\epsilon_{1}=1$. So $q_{1}=q_{0}\hat{y}_{1}^{5}=y_{1}^{5}x_{1}^{-5}x_{2}^{-3}x_{3}^{9}$.\\
  (b)\; $K_2=\{3\}$, so $i_{2}=2$ and $\epsilon_{2}=1$. Then $q_{2}=q_{1}\hat{y}_{2}^{3}=y_{1}^{5}y_{2}^{3}x_{1}x_{2}^{-3}x_{3}^{-3}$.\\
  (c)\; $K_3=\emptyset$, $i_{3}=3$ and $\epsilon_{3}=1$. So $q_{3}=q_{2}\hat{y}_{3}^{3}=y_{1}^{5}y_{2}^{3}y_{3}^{3}x_{1}^{-2}x_{2}^{3}x_{3}^{-3}$.\\
  (d)\; $K_4=\{2\}$, $i_{4}=1$ and $\epsilon_{4}=1$. So $q_{4}=q_{3}\hat{y}_{1}^{2}=y_{1}^{7}y_{2}^{3}y_{3}^{3}x_{1}^{-2}x_{2}x_{3}^{-1}$.\\
  (e)\; $K_5=\emptyset$, $i_{5}=3$ and $\epsilon_{5}=-1$. So $q_{5}=q_{4}\hat{y}_{3}^{-1}=y_{1}^{7}y_{2}^{3}y_{3}^{2}x_{1}^{-1}x_{2}^{-1}x_{3}^{-1}$.\\
  (f)\; $K_6=\{1\}$, $i_{6}=2$ and $\epsilon_{6}=-1$. So $q_{6}=q_{5}\hat{y}_{2}^{-1}=y_{1}^{7}y_{2}^{2}y_{3}^{2}x_{1}^{-2}x_{2}^{-1}x_{3}^{3}$.\\
  (g)\; $K_7=\emptyset$, $i_{7}=1$ and $\epsilon_{7}=-1$. However in this case $q_6\hat{y}_1^2=y_{1}^{5}y_{2}^{2}y_{3}^{2}x_{1}^{-2}x_{2}x_{3}$ is not a vertex of their convex hull. Hence it is not included in $E_h$ and thus the procedure ends here. Therefore, $E_{h}$ is the set consisting of red edges as showed in the left-hand side of Figure \ref{example2} (In order to save space, we put the figure in the next subsection, where there is also a need for this figure and its right-hand side).

  (ii)\footnote{This example was provided by Jiarui Fei, which helps us to revise the notion of essential skeleton in Definition \ref{essential}i.}\;
  Let $\A$ be a cluster algebra having principal coefficients of type A with the initial seed $(X,Y,B)$, where $X=(x_{1},x_{2},x_{3})$, $Y=(y_{1},y_{2},y_{3})$ and
  \[B=\left(\begin{array}{ccc}
  0 & -1 & 1 \\
  1 & 0 & -1 \\
  -1 & 1 & 0
  \end{array}\right).\]
  Following the notions and notations of Definition \ref{essential}, the essential skeleton $E_{h}$ for $h=(0,-1,-2)$ can be calculated as follows:\\
  (a)\; $q_{0}=x_{2}^{-1}x_{3}^{-2}$, $K_1=\emptyset$ and $i_{1}$ can be either 2 or 3. When $i_{1}=2$, $\epsilon_1=1$, $q_{1}=q_{0}\hat{y}_{2}=y_{2}x_{1}^{-1}x_{2}^{-1}x_{3}^{-1}$.\\
  (b)\; Since $K_{2}=\{3\}$, $i_{2}$ can only be 1 and $\epsilon_{2}=1$. So $q_{2}=q_{1}\hat{y}_{1}=y_{1}y_{2}x_{1}^{-1}x_{3}^{-2}$.\\
  (c)\; $K_3=\emptyset$, hence $i_{3}=3$ and $\epsilon_{3}=1$. So $q_{3}=q_{2}\hat{y}_{3}^{2}=y_{1}y_{2}y_{3}^{2}x_{1}x_{2}^{-2}x_{3}^{-2}$.\\
  (d)\; $K_4=\{1\}$, hence $i_{4}=2$ and $\epsilon_{4}=1$. So $q_{4}=q_{3}\hat{y}_{2}^{2}=y_{1}y_{2}^{3}y_{3}^{2}x_{1}^{-1}x_{2}^{-2}$.\\
  (e)\; $K_5=\emptyset$, hence $i_{5}=1$ and $\epsilon_{5}=-1$. So $q_{5}=q_{4}\hat{y}_{1}^{-1}=y_{2}^{3}y_{3}^{2}x_{1}^{-1}x_{2}^{-3}x_{3}$.\\
  (f)\; $K_6=\emptyset$, hence $i_6=2$ and $\epsilon_6=-1$. So $q_{6}=q_{5}\hat{y}_{2}^{-3}=y_{3}^{2}x_{1}^{2}x_{2}^{-3}x_{3}^{-2}$.\\
  (g)\; $K_7=\{1\}$, hence $i_7=3$ and $\epsilon_{7}=-1$. Then $q_{7}=q_{6}\hat{y}_{3}^{-2}=x_{2}^{-1}x_{3}^{-2}=q_0$.\\
  If choose $i_1$ to be 3, then we can similarly obtain the sequence of points inversely from $q_7$ to $q_0$. Thus we complete the calculation of essential skeleton.
\end{Example}

As a generalization of the definitions of $N_{h}$ and $\rho_h$ in rank 2 case, we present the construction of the polytope $N_{h}$ and the polytope function $\rho_{h}$ for a cluster algebra with principal coefficients of arbitrary rank as follows. In the sequel, $\rho_{h}^{t_{0}}$ and $N_{h}^{t_{0}}$ will be abbreviated as $\rho_{h}$ and $N_{h}$ respectively.
\begin{Construction}\label{construction}\textbf{(Construction program of $N_h$ and $\rho_h$)}
  Given a TSSS cluster algebra $\A$ of rank $n$ with principal coefficients and a vector $h\in\Z^{n}$, we can construct $N_{h}$ as well as $\rho_{h}$ as follows:

  When $h\in\N^{n}$, let $N_{h}$ be the origin and $\rho_{h}=X^{h}$.

  When $h\in \Z^n\backslash \N^n$, $N_{h}$ as well as $\rho_{h}$ can be constructed inductively step by step:

  (i)\; When $n=1$, Let $N_{h}$ equal $\overline{pq}$ and $\rho_{h}=X^{h}(\hat{y}_{1}+1)^{[-h]_{+}}$, where $p,q$ are the origin and $[-h]_{+}$ respectively.

  (ii)\; When $n>1$ and $h\in \Z^n\backslash\N^{n}$, we construct $N_{h}$ iteratively. Assume we have already construct $N_{h\'}$ less than $N_{h}$ with respective to the partial order induced by sub-polytopes, then we can construct $N_{h}$ in following two cases:

  (a)\; If there is $k\in[1,n]$ such that $deg_{x_{k}}(\hat{Y}^{p}X^{h})\geqslant 0$ for any point $p\in\gamma_{k;0}(N_{\pi_{k}(h)}|_{\A_{k}})$, then  $N_{h}=\gamma_{k;0}(N_{\pi_{k}(h)}|_{\A_{k}})$ with dimension not larger than $n-1$, where $\A_{k}$ is the cluster algebra of rank $n-1$ with principal coefficients whose initial exchange matrix is obtained from $B$ by deleting its $k$-th row and column while $N_{\pi_{k}(h)}|_{\A_{k}}$ is the polytope associated to $\pi_{k}(k)$ for $\A_{k}$;

  (b)\; If for any $k\in[1,n]$, $deg_{x_{k}}(\hat{Y}^{p}X^{h})< 0$ holds for some point $p\in\gamma_{k;0}(N_{\pi_{k}(h)}|_{\A_{k}})$, then the dimension of $N_h$ is $n$ (which means $N_{h}$ is the colimit of a family of polytpes with dimension $n$ when it contains infinitely many points). Choose any $k$ satisfying $N_{\pi_{k}(h)}|_{\A_{k}}$ has dimension $n-1$ and $i\neq k\in[1,n]$. Then we can construct $N_h$ inductively by the following steps (1) to (4):

  (1)\;First we can calculate the essential skeleton $E_{h}$ by Definition \ref{essential}. Then replace each edge $\overline{pq}$ (assume $p<q$) in $E_{h}$ parallel to $e_{i}$ by two segments $\overline{pq\'}$ and $\overline{p\'q}$, where $p\'=p+e_{i}$ while $q\'=q-e_{i}$. Here the segments $\overline{pq\'}$ and $\overline{p\'q}$ may degenerate to points, and we regard them separately.

  After replacement $E_{h}$ can be divided into several components satisfying that two segments $l$ and $l\'$ are in the same component if they are connected by a sequence of segments $l=l_{0},\cdots,l_{r}=l\'$ such that $l_{i}$ and $l_{j}$ have a same endpoint. There is a unique way to complete each component $\theta$ to be
  $$E(h+e_{i}+w_{\theta}B^{\top},[1,n])[w_{\theta}]$$
  such that
  $$conv(E(h+e_{i}+w_{\theta}B^{\top},[1,n]))\subseteq conv(E_{h})$$
  for some $w_{\theta}\in\N^{n}$. Then let $U_{E_{h}}$ be the set consisting of above polytopes $N_{h+e_{i}+w_{\theta}B^{\top}}[w_{\theta}]$ and
  \[\mathcal{R}=\left\{\gamma_{k;q_{k}}(N_{\pi_{k}(h+qB^{\top})}[\pi_{k}(q)]|_{\A_{k}})|q\text{ is a minimal point in some face $F$ of }conv(E_{h})\right.\]\[\left.\text{with }F=conv(E_{h})\cap\text{ the hyperplane } z_{k}=q_{k}\right\},\]
  where $q$ is an arbitrary point in the segment $\overline{q\'q^{\prime\prime}}\in E_{h}$ with $q\'$ and $q^{\prime\prime}$ two vertices, and $q_{k}$ denotes the $k$-th entry of $q$.

  (2)\;Denote $\mathcal{R}_{s,+}=\{\gamma_{k;q_{k}}(N_{\pi_{k}(h+qB^{\top})}[\pi_{k}(q)]|_{\A_{k}})\in\mathcal{R}|q_{k}\leqslant s\}$ and $\mathcal{R}_{s,-}=\mathcal{R}\setminus\mathcal{R}_{s,+}$. Choose $a\in\N$ satisfying that $a\geqslant deg_{x_{k}}(\hat{Y}^{p}X^{h})$ for any point $p\in \mathcal{R}$.

  Let $N(s,\epsilon)$ be the polytope consisting of points $p$ satisfying $p_{k}=s$ and there is a point $q\in\mathcal{R}_{s,\epsilon}$ such that $\overline{pq}$ parallels to $e_{k}$, $l(\overline{pq})\leqslant a-deg_{x_{k}}(\hat{Y}^{q}X^{h})$ with weight $co_{p}(N(s,\epsilon))=\binom{a-deg_{x_{k}}}{l(\overline{pq})}$. Define a sequence of polytopes as $N(s,+,i)=\gamma_{k;s}(N_{\pi_{k}(h+pB^{\top})}[\pi_{k}(p)]|_{\A_{k}})$, where $p$ is a minimal point in $N(s,+)-\sum\limits_{j=0}^{i-1}N(s,+,j)$ with positive weight and $N(s,+,0)=\emptyset$.

  Dually define a sequence of polytopes as $N(s,-,i)=\gamma_{k;s}(N_{\pi_{k}(\alpha_{p})}[w_p]|_{\A_{k}})$, where $p$ is a maximal point in $N(s,-)-\sum\limits_{j=0}^{i-1}N(s,-,j)$ with positive weight, $N(s,-,0)=\emptyset$ while $\pi_{k}(\alpha_{p})\in\Z^{n-1}$ and $w_{p}\in\N^{n-1}$ satisfying that $\gamma_{k;s}(N_{\pi_{k}(\alpha_{p})}[w_p]|_{\A_{k}})$ is the unique polytope having $p$ as the maximal point and $\alpha_{p}=h+\gamma_{k;s}(w_{p})B^{\top}$. Let $\tilde{S}_{s}$ be the set consisting of $N(s,+,i)$ and $N(s,-,i)$ for $i\in\N$.

  (3)\;Define $U_{h}^{0}=\bigcup\limits_{j}\{(N_{v_j}[\gamma_{k;0}(w_{j})])^{c_{w_{j},\alpha_{j}}}\; |\;v_j=\gamma_{k;h_{k}+\gamma_{k;0}(w_{j})(b^{\top})_{k}}(\alpha_{j})\}$, where $(b^{\top})_{k}$ represents the $k$-th column of $B^{\top}$, while $c_{w_{j},\alpha_{j}}$, $\alpha_{j}$ and $w_{j}$ are those appearing in the right-hand side of the equation $\rho_{\pi_{k}(h)}x_{i}=\sum\limits_{j}c_{w_{j},\alpha_{j}}Y^{w_{j}}\rho_{\alpha_{j}}$ for the cluster algebra $\A_{k}$. Here all sets denoted by $U_{h}^{l}$ or $U_{i}$ are multisets, and the superscript of a polytope denotes its multiplicity which equals the number of this polytope in the multiset.

  When $U_{h}^{l}$ has been determined for any $0\leqslant l<s$, $U_{h}^{s}$ is defined as follows.

  According to the iterative construction of sub-polytopes $N_{h\'}$, the intersection of $z_{k}=s$ and $\sum\limits_{N_{\alpha_{j}}[w_{j}]\in \bigcup\limits_{0\leqslant l< s}U_{h}^{l}}N_{\alpha_{j}}[w_{j}]+\sum\limits_{\substack{N_{\alpha}[w]\in U_{E_{h}}\\\text{and }w(k)=s}}N_{\alpha}[w]$ (or $\sum\limits_{N_{\alpha_{j}}[w_{j}]\in \bigcup\limits_{0\leqslant l< s}U_{h}^{l}}N_{\alpha_{j}}[w_{j}]$ respectively) equals the sum of certain polytopes associated to vectors in $\Z^{n}$ for $\A_{k}$ up to a translation. So we use a countable set $\tilde{J}$ (or $J$ respectively) to parameterize these polytopes and then the sum is $\sum\limits_{j\in \tilde{J}} \gamma_{k;s}(N_{\alpha\'_{j}}[w\'_{j}]|_{\A_{k}})$ (or $\sum\limits_{j\in J} \gamma_{k;s}(N_{\alpha\'_{j}}[w\'_{j}]|_{\A_{k}})$ respectively), where $w(k)$ is the $k$-th entry of $w$.

  There is $j_0\in J$ such that $\gamma_{k;s}(N_{\alpha\'_{j_{0}}}[w\'_{j_{0}}]|_{\A_{k}})$ is the summand of $\sum\limits_{j\in J} \gamma_{k;s}(N_{\alpha\'_{j}}[w\'_{j}]|_{\A_{k}})$ containing a point $p_{0}$ with minimal $deg_{x_{k}}(\hat{Y}^{p_0}X^{h})$ and weight $c>0$ in $\sum\limits_{j\in \tilde{J}} \gamma_{k;s}(N_{\alpha\'_{j}}[w\'_{j}]|_{\A_{k}})$. Then associate to $\gamma_{k;s}(N_{\alpha\'_{j_{0}}}[w\'_{j_{0}}]|_{\A_{k}})$ the polytope $\gamma_{k;s}(N_{\alpha\'}[w\']|_{\A_{k}})\in\tilde{S}_{s}$ satisfying that $\hat{Y}^{w\'_{j_{0}}}\rho_{\alpha\'_{j_{0}}}$ appears in the following decomposition
  \begin{equation}\label{equation in the construction}
    cY^{w\'}\rho_{\alpha\'}x_{i}=\sum\limits_{j\in J^{\prime\prime}}Y^{w\'_{j}}\rho_{\alpha\'_{j}},
  \end{equation}
  for $\A_{k}$ induced by the construction of $N_{\alpha\'}$ and $deg_{x_{k}}(\hat{Y}^{p}X^{h})\geqslant deg_{x_{k}}(\hat{Y}^{p_{0}}X^{h})$ for any point $p$ in $\gamma_{k;s}(N_{\alpha\'}[w\']|_{\A_{k}})$. Let $N_{1}=\sum\limits_{j\in \tilde{J}\'} \gamma_{k;s}(N_{\alpha\'_{j}}[w\'_{j}])$, where $\tilde{J}\'$ is the maximal subset of $\tilde{J}$ making $N_{1}$ a summand of $c\gamma_{k;s}(N_{\alpha\'}[w\'])$ whose corresponding Laurent polynomial appears in the right-hand side of (\ref{equation in the construction}).

  Let $J\'=\tilde{J}\'\cap J$ and $U_{1}$ be the minimal set consisting of polytopes $N_{\nu_{j}}[v_{j}]$ such that the intersection of $z_{k}=s$ and $\sum\limits_{N_{\nu_{j}}[v_{j}]\in U_{1}}N_{\nu_{j}}[v_{j}]$ equals $\sum\limits_{j\in J^{\prime\prime}\setminus J'}\gamma_{k;s}(N_{\alpha\'_{j}}[w\'_{j}])$ and the homogeneous Laurent polynomial with degree $h$ corresponding to the intersection of $z_{k}=r$ and $\sum\limits_{N_{\nu_{j}}[v_{j}]\in U_{1}}N_{\nu_{j}}[v_{j}]$ appears in the decomposition of $\sum\limits_{\gamma_{k;s}(N_{\pi_{k}(\alpha_{p})}[w_p]|_{\A_{k}})\in\tilde{S}_{r}}x_{i}\hat{Y}^{\gamma_{k;s}(w_p )}\rho_{\pi_{k}(\alpha_{p})}|_{\A_{k}}$ according to the iterative construction.

  Replacing $\sum\limits_{j\in \tilde{J}} \gamma_{k;s}(N_{\alpha\'_{j}}[w_{j\'}])$ by $\sum\limits_{j\in \tilde{J}} \gamma_{k;s}(N_{\alpha\'_{j}}[w_{j\'}])-N_{1}$ and repeating the above process to get $N_{2}$ and $U_{2}$. This goes on until $\sum\limits_{j\in \tilde{J}} \gamma_{k;s}(N_{\alpha\'_{j}}[w_{j\'}])-\sum\limits_{i}N_{i}=0$. Then finally, let $U^{s}_{h}=\bigcup\limits_{i}U_{i}$, and we call it the $s$-th stratum of $N_h$ for $x_{i}$ along direction $k$.

  (4)\;At last, let $N_{h}=\sum\limits_{N_{\alpha_{j}}[w_{j}]\in \bigcup\limits_{l}U_{h}^{l}}N_{\alpha_{j}}[w_{j}]$ and $\rho_{h}=\sum\limits_{p\in N_{h}}co_{p}(N_{h})\hat{Y}^{p}X^{h}$, where $U_{h}^{l}$ runs over all strata of $N_h$ for $x_{i}$ along direction $k$.
\end{Construction}

As in the case of rank 2, we call $\rho_h$ the \textbf{polytope function} associated to $h$ for any $n$. It is easy to check that $\rho_{h}$ is homogeneous with degree $h$ under the canonical $\Z^{n}$-grading since
\[deg(\hat{y_{i}})=deg(y_{i}X^{b_{i}})=-b_{i}+\sum\limits_{j=1}^{n}b_{ji}e_{j}=0\]
for any $i\in[1,n]$, so $deg(\rho_{h})=deg(F_{h}|_{\F}(\hat{Y}))+deg(X^{h})=h$.

In general, we are not sure whether $\rho_{h}$ constructed above has finitely many terms, but it must be a formal Laurent polynomial. So we abuse the definition of polytope and also call the convex hull of a vector set of infinitely many elements a \emph{polytope}. However, because such ``polytope'' corresponds to a formal Laurent polynomial, it can be regarded as the colimit of a family of polytopes. So in the sequel, we will always treat them as polytopes. And in the sequel proofs, we usually only concern about polytopes corresponding to Laurent polynomials. While for those ``polytopes" corresponding to formal Laurent polynomials, we need to further more take colimit. However, for convenience, we will omit the process of colimit and regard the results of polytopes to hold also for such ``polytopes" once they hold for polytopes corresponding to Laurent polynomials.

\begin{Remark}\label{rmk of construction}
  (i)\; In Construction \ref{construction}(ii)(a), $N_{h}$ does not depend on the choice of $k$.

  In fact, in (a), if there are two different indices $j,k\in[1,n]$ with $j<k$ such that
  \[deg_{x_{k}}(\hat{Y}^{p}X^{h})\geqslant 0\text{ for any point }p\in\gamma_{k;0}(N_{\pi_{k}(h)}|_{\A_{k}});\]
  \[deg_{x_{j}}(\hat{Y}^{p}X^{h})\geqslant 0\text{ for any point }p\in\gamma_{j;0}(N_{\pi_{j}(h)}|_{\A_{j}}),\]
  then $deg_{x_{j}}(\hat{Y}^{p}X^{h})\geqslant 0$ for any point $p\in\gamma_{k;0}(\gamma_{j;0}(N_{\pi_{j}\circ\pi_{k}(h)}|_{\A_{k,j}}))$, where $\A_{k,j}$ is the cluster algebra with principal coefficients whose initial exchange matrix is obtained from $B$ by deleting its $k$-th and $j$-th rows and columns. Therefore we have $$\gamma_{k;0}(N_{\pi_{k}(h)}|_{\A_{k}})=\gamma_{k;0}(\gamma_{j;0}(N_{\pi_{j}\circ\pi_{k}(h)}|_{\A_{k,j}})) =\gamma_{j;0}(\gamma_{k-1;0}(N_{\pi_{k-1}\circ\pi_{j}(h)}|_{\A_{k,j}}))=\gamma_{j;0}(N_{\pi_{j}(h)}|_{\A_{j}}).$$
  So, Construction \ref{construction}(ii)(a) is well-defined.

  (ii)\; In (b) (2), $N_{\pi_{k}(\alpha_{p})}[w_p]|_{\A_{k}}$ in fact equals $conv(\{\pi_{k}(p)-q|q\in N_{\pi_{k}(h+pB^{\top})}|_{\A_{k}\'}\})$ equipped with weights $co_{q}(N_{\pi_{k}(\alpha_{p})}[w_p])=co_{p-q}(N_{\pi_{k}(h+pB^{\top})}|_{\A_{k}\'})$, where $\A_{k}\'$ is the cluster algebra associated to the matrix obtained from $-B$ by deleting the $k$-th column and row.

  (iii)\; It can be verified that $N_{h}$ is independent of the choice of $k$ and $i$ in Construction \ref{construction} (b). In fact, $N_{h}$ can be written as a sum of its sub-polytopes, and so can these sub-polytopes. So for $i\neq j\in[1,n]\setminus\{k\}$, by induction on the order induced by sub-polytopes, we can see that
  \[N_{h}=\sum\limits_{N_{\alpha}[w]\in \bigcup\limits_{l}U_{h}^{l}(i)}\sum\limits_{N_{\beta}[u]\in \bigcup\limits_{l}U_{\alpha}^{l}(j)}N_{\beta}[u+w]=\sum\limits_{N_{\alpha\'}[w\']\in \bigcup\limits_{l}U_{h}^{l}(j)}\sum\limits_{N_{\beta\'}[u\']\in \bigcup\limits_{l}U_{\alpha\'}^{l}(i)}N_{\beta\'}[u\'+w\'],\]
  where $U_{h}^{l}(i)$ denotes the $l$-th stratum of $N_h$ for $x_{i}$ along direction $k$. Moreover, if we iteratively decompose polytopes as a sum of sub-polytopes for enough large times, we finally obtain a unique decomposition of $N_{h}$ as a sum of polytopes in dimension 1 as we claimed in Subsection 3.1.

  Also, it can be verified that the union of strata of $N_{h}$ for $x_{i}$ along direction $k$ is independent of the choice of $k$ by comparing the construction of $N_{h}$ with respect to different choices of $k$. Therefore the decomposition (\ref{equation: statement for decomposition}) can be uniquely calculated analogous to (\ref{equation: decomposition of x1Ph}) and (\ref{equation: decomposition of x2Ph}) in general case.

  Hence, although we need to fix $k$ and $i$ during the construction of $N_{h}$ and $\rho_{h}$, neither $N_{h}$ nor $\rho_{h}$ depends on the choice of $k$ and $i$. In other words, $N_{h}$ and $\rho_{h}$ are uniquely determined by the initial seed $\Sigma_{t_{0}}$ and the vector $h$.

  (iv)\; Following (iii), in practical, Construction \ref{construction} (b) can be simplified by choosing $k$ and $i$ satisfying $N_{\pi_{k}(h)}|_{\A_{k}}$ has dimension $n-1$ and $b_{ik}\geqslant0$. Such pair can always be found because if $b_{jk}<0$ for any $j\in[1,n]\setminus\{k\}$, then $h_{k}<0$ since $deg_{x_{k}}(\hat{Y}^{p}X^{h})< 0$ holds for some point $p\in\gamma_{k;0}(N_{\pi_{k}(h)}|_{\A_{k}})$ and $deg_{x_{k}}(\hat{y}_{j})>0$. Therefore, we can find $k\'\neq k$ satisfying $N_{\pi_{k\'}(h)}|_{\A_{k\'}}$ has dimension $n-1$ and according to our assumption $b_{kk\'}>0$.

  In this case, we can omit $U_{E_{h}}$ and replace $\mathcal{R}$ by $\{\gamma_{k;0}(N_{\pi_{k}(h)}|_{\A_{k}})\}$.
\end{Remark}

\begin{Definition}
  Fix $s_{i}\in\{\pm 1\}$ for any $i\in [1,n]$, define an order with respect to $s_{i}$ in a polytope $N$ as $p\leqslant_{s_{i}} p\'$ for any point $p,p\'\in N$ if $p$ is in the $i$-section at $p\'$ and $s_{i}z_{i}(p)\leqslant s_{i}z_{i}(p\')$, where $z_{i}(p)$ is the $i$-th coordinate of $p$. Then define a map $$m_{i}:\; N\cap\Z^{n}\;\longrightarrow\;\Z$$ such that $m_{i}(p)=0$ if $p\notin N$ and inductively,
  \begin{equation*}
  m_{i}(p)=co_{p}(N)-\sum\limits_{v<_{s_{i}}p}m_{i}(v)\tilde{C}_{-deg_{x_{i}}(p)}^{l(\overline{pv})}.
  \end{equation*}
\end{Definition}

\begin{Remark}
  (i) The definition of $N_{h}$ in rank 2 case is compatible to Construction \ref{construction} according to (\ref{equation: decomposition of x1Ph}) and (\ref{equation: decomposition of x2Ph}).

  (ii) The main reason why we do not construct $N_{h}$ using the form similar to (\ref{equation: weight of rank2}) is that in general there may be a point $p$ such that $m_{i}(p)>0$ for any $i\in[1,n]$ no matter how we choose $s_{k}$ for $k\in[1,n]$, so it is difficult to determine the weight $co_{p}$ directly this way.

  Instead, we construct it as a sum of certain sub-polytopes as a generalization of (\ref{equation: decomposition of x1Ph}) and (\ref{equation: decomposition of x2Ph}). Then there comes the problem how to determine such decomposition. So we choose an integer $a$ as in Construction \ref{construction}, then $\rho_{h}$ must be a summand of the decomposition of $x_{k}^{a}\rho_{h-ae_{k}}$. Because $deg_{x_{k}}(\hat{Y}^{p}X^{h-ae_{k}})\leqslant0$ for any point $p\in\mathcal{R}$, it helps us to complete $\sum\limits_{j\in J} \gamma_{k;s}(N_{\alpha\'_{j}}[w\'_{j}])$ as a Laurent polynomial which maintains a Laurent polynomial under mutations in direction $[1,n]\setminus\{k\}$.
\end{Remark}

\subsection{Main theorem on $N_h$ and $\rho_h$ in any rank case}\quad

In this subsection, we show that most properties in Subsection 3.2 still hold in arbitrary rank case, which will lead to some important applications in the following sections.

\begin{Theorem}\label{properties in case tsss}
  Let $\A$ be a TSSS cluster algebra having principal coefficients and $h\in\Z^{n}$. Then,

  (i)\;For any $i\in[1,n]$, there is a decomposition
  \begin{equation}\label{equation: statement for decomposition}
    x_{i}\rho_{h}=\sum\limits_{w,\alpha} c_{w,\alpha}Y^{w}\rho_{\alpha},
  \end{equation}
  where $w\in\N^{n},\alpha\in\Z^{n}$ and $c_{w,\alpha}\in\N$.

  (ii)\;The polytope function $\rho_{h}$ is the unique indecomposable formal Laurent polynomial in $X$ in $\widehat{\mathcal{U}}_{\geqslant0}(\Sigma_{t_{0}})$ which has $X^{h}$ as a summand and whose support is contained in $supp(N_{h})$.

  (iii)\; For any $h\in\Z^{n}$ and any $k\in[1,n]$, there is
  \begin{equation}\label{mutation formula for h}
    h^{t_{k}}=h-2h_{k}e_{k}+h_{k}[(b_{k})^{\top}]_{+}+[-h_{k}]_{+}(b_{k})^{\top}
  \end{equation}
  such that $L^{t_{k}}(\rho_{h})=\rho^{t_{k}}_{h^{t_{k}}}$, where $t_{k}\in\T_{n}$ is the vertex connected to $t_{0}$ by an edge labeled $k$ and $h_{k}$ is the $k$-th entry of $h$.

  (iv)\;For any $p,p\'\in N_{h}$, if the segment $l$ connecting $p$ and $p\'$ is parallel to the $k$-th coordinate axis for some $k\in[1,n]$ and $m_{k}(p),m_{k}(p\')>0$, then $m_{k}(p^{\prime\prime})>0$ for any point $p^{\prime\prime}\in l$.

  (v)\; Let $S$ be an $r$-dimensional face of $N_{h}$ for $h\in\Z^{n}$ such that $\rho_{h}\in \mathcal{U}_{\geqslant0}(\Sigma)$. Then there is a vector $h\'\in\Z^{ldim(S)}$ and a cluster algebra $\A\'$ with principal coefficients of rank $ldim(S)$ and a non-negative polytope projection $\tau: N_{h\'}|_{\A\'}\rightarrow S$. In particular, $\tau$ is an isomorphism when $ldim(S)=r$.
\end{Theorem}
\begin{Proof}
  We prove this theorem using an induction on the partial order induced by sub-polytopes.

  When $h\in\Z^{2}$, (i) is equivalent to the equations (\ref{equation: decomposition of x1Ph}) and (\ref{equation: decomposition of x2Ph}) for $i=1,2$ respectively; (ii) follows from Lemma \ref{unique indecomposable in rank 2}; (iii) is ensured by Lemma \ref{face for rank 2} (i); (iv) follows from Lemma \ref{m_j in rank 2} and (v) is exactly Lemma \ref{face for rank 2} (ii).

  Assume they hold for all $h\'\in\Z^{n}$ satisfying $N_{h\'}$ is a sub-polytope of $N_{h}$. Then let us focus on $h$.

  If there is $k\in[1,n]$ such that $deg_{x_{k}}(\hat{Y}^{p}X^{h})\geqslant 0$ for any point $p\in\gamma_{k;0}(N_{\pi_{k}(h)}|_{\A_{k}})$, which according to Construction \ref{construction} means $N_{h}=\gamma_{k;0}(N_{\pi_{k}(h)}|_{\A_{k}})$, then the theorem for $N_{h}$ is induced by that for $N_{\pi_{k}(h)}|_{\A_{k}}$ and that for rank 2 case (we can consider the mutation in direction $k$ by looking at $\{i,k\}$-sections of $N_{h}$ for any $i\in[1,n]\setminus\{k\}$). So in the following we always assume $dim(N_{h})=n$ and it is constructed by Construction \ref{construction} (b).

  (i)\;If there is $k\neq i$ such that $dim(N_{\pi_{k}(h)}|_{\A_{k}})=n-1$, where $\A_{k}$ is the cluster algebra of rank $n-1$ with principal coefficients whose initial exchange matrix is obtained from $B$ by deleting its $k$-th row and column, then according to the construction of $N_{h}$, we have $N_{h}=\sum\limits_{N_{\alpha_{j}}[w_{j}]\in \bigcup\limits_{l}U_{h}^{l}}N_{\alpha_{j}}[w_{j}]$. So $X^{\beta}\rho_{h}=\sum\limits_{N_{\alpha_{j}}[w_{j}]\in \bigcup\limits_{l}U_{h}^{l}}Y^{w_{j}}\rho_{\alpha_{j}}$ for some $\beta\in\Z^{n}$. Moreover, $deg(Y^{w_{j}}\rho_{\alpha_{j}})=h+e_{i}=deg(\rho_{h})+e_{i}$ due to the definition of each stratum for any $N_{\alpha_{j}}[w_{j}]\in \bigcup\limits_{l}U_{h}^{l}$, hence $\beta=e_{i}$, which leads to
  \begin{equation}\label{equation: decomposition of rho_h for x_i}
    x_{i}\rho_{h}=\sum\limits_{N_{\alpha}[w]\in\bigcup\limits_{l} U_{h}^{l}}Y^{w}\rho_{\alpha}
  \end{equation}
  for any $i\in[1,n]$, which is the formal Laurent polynomial version of $N_{h}=\sum\limits_{N_{\alpha_{j}}[w_{j}]\in\bigcup\limits_{l}U_{h}^{l}}N_{\alpha_{j}}[w_{j}]$.

  If $dim(N_{\pi_{k}(h)}|_{\A_{k}})<n-1$ for any $k\neq i$, define a sum of sub-polytopes of $N_{h}$ recursively: let $w_{1}=0$, $\alpha_{1}=h+e_{k}$ and $c_{w_{1},\alpha_{1}}=1$. For any $j>1$, let $w_{j}$ be a minimal point in $N_{h}-\sum\limits_{l=1}^{j-1}c_{w_{l},\alpha_{l}}N_{\alpha_l}[w_l]$, $\alpha_{j}=h+e_{k}+w_{j}B^{\top}$ and $c_{w_{j},\alpha_{j}}=co_{w_{j}}(N_{h}-\sum\limits_{l=1}^{j-1}c_{w_{l},\alpha_{l}}N_{\alpha_l}[w_l])$. Then similar to the proof of (\ref{equation: decomposition of x1Ph}) and (\ref{equation: decomposition of x2Ph}), we can show that $c_{w_{j},\alpha_{j}}>0$ for any $j$ with the help of inductive assumption of (i) for sub-polytopes and Construction \ref{construction}. Therefore we get $N_{h}=\sum\limits_{l}c_{w_{l},\alpha_{l}}N_{\alpha_l}[w_l]$ with $c_{w_{l},\alpha_{l}}>0$. Then similar to the above case, (i) holds via comparing the degree of both sides.

  \textbf{First part of (ii).}\;We first show that $\rho_{h}$ is an indecomposable formal Laurent polynomial in $X$ in $\widehat{\mathcal{U}}_{\geqslant0}(\Sigma_{t_{0}})$ which has $X^{h}$ as a summand and whose support is contained in $supp(N_{h})$. Then after the proof of (iv), we prove the uniqueness.

  Denote $\widehat{\mathcal{U}}_{\geqslant0}^{k}(\Sigma_{t})=\N\P[[X_{t}]]\cap\bigcap\limits_{t\'}\N\P[[X_{t\'}]]$ for any $k\in[1,n]$ and $t\in\T_{n}$, where $t\'$ runs over all vertices connected to $t$ by an edge not labeled $k$. According to the construction and inductive assumption of (ii) for sub-polytopes, we have $f_{s}\in\bigcap\limits_{l\in[1,n]\setminus \{k\}}\widehat{\mathcal{U}}_{\geqslant0}^{k}(\Sigma_{t_{l}})$, where $f_{s}$ is the formal Laurent polynomial corresponding to the intersection of $z_{k}=s$ and $N_{h}$ while $t_{l}$ represents the vertex connected to $t_{0}$ by an edge labeled $l$. Hence so does $\rho_{h}=\sum\limits_{s\in\N}f_{s}$. Moreover, this together with inductive assumption of (ii) for sub-polytopes induces $Y^{w_{j}}\rho_{\alpha_{j}}\in\bigcap\limits_{l\in[1,n]\setminus \{i\}}\widehat{\mathcal{U}}_{\geqslant0}^{i}(\Sigma_{t_{l}})\text{ for any }N_{\alpha_{j}}[w_{j}]\in \bigcup\limits_{l}U_{h}^{l}$. Hence so does $\rho_{h}$ following (i). Thus we get $\rho_{h}\in\widehat{\mathcal{U}}_{\geqslant0}(\Sigma)$ as $i\neq k$,.

  We claim that $\rho_h$ itself is the only non-zero summand of $\rho_{h}$ in $\widehat{\mathcal{U}}_{\geqslant0}(\Sigma)$. Assume there is a summand $f\neq\rho_{h}$ of $\rho_{h}$ in $\widehat{\mathcal{U}}_{\geqslant0}(\Sigma)$. Let $p\in supp(\rho_h)$ be a minimal point such that $co_{p}(f)<co_{p}(\rho_h)$. Then $f$ is a summand of $\rho_{h}-\hat{Y}^{p}X^{h}(1+\hat{y}_{i})^{[-deg_{x_{i}}(\hat{Y}^{p}X^{h})]_{+}}$ for any $i$ since $f\in\widehat{\mathcal{U}}_{\geqslant0}(\Sigma)$. In other words, because during the inductive construction of $\rho_h$, every term $p$ of $\rho_{h}$ is added as a complement of some monomial summand $q$ in certain direction, so $f$ being a summand of $\rho_{h}-p$ leads to $f$ being a summand of $\rho_{h}-p-q$ so as to make sure $f\in\widehat{\mathcal{U}}_{\geqslant0}(\Sigma)$. Therefore, we finally get $f$ being a summand of $\rho_{h}-\rho_{h}=0$ by continuously applying the above result, which shows our claim.

  So $\rho_{h}$ is indecomposable in $\widehat{\mathcal{U}}_{\geqslant0}(\Sigma)$, it has $X^{h}$ as a summand and its support is contained in $supp(N_{h})$.

  (iii)\;We check how the corresponding polytope is changed under the change of initial seed. Denote by $L^{t_{k}}(N_{h})$ the corresponding polytope after the initial seed changing to $\Sigma_{t_{k}}$, i.e., the polytope corresponding to $L^{t_{k}}(\rho_{h})$. We will compare the weights of two polytopes $L^{t_{k}}(N_{h})$ and $N_{h^{t_{k}}}^{t_{k}}$ to show they are the same.

  According to (\ref{equation: decomposition of rho_h for x_i}) we have
  \begin{equation*}
    L^{t_{k}}(\rho_{h})=\sum\limits_{N_{\alpha}[w]\in\bigcup\limits_{s} U_{h}^{s}}L^{t_{k}}(Y^{w}\rho_{\alpha}x_{i}^{-1}),
  \end{equation*}
  where $U_h^s$ is the $s$-th stratum of $N_h$ for $x_{i}$ along direction $r$. Since $b_{kr}b_{kr}^{t_{k}}\leqslant0$, according to Remark \ref{rmk of construction} (iv), we can choose $i\neq k$ such that either $b_{ir}\geqslant0$ or $b_{ir}^{t_{k}}\geqslant0$. Without loss of generality assume $b_{ir}\geqslant0$ (otherwise we can compare $N_{h}$ and $L^{t_{0}}(N_{h^{t_{k}}}^{t_{k}})$ instead).

  First we deal with the case where $r\neq k$ and the dimension of $N_{\pi_{r}(h^{t_{k}})}^{t_{k}}|_{\A_{r}}$ equal $dim(N_{h^{t_{k}}}^{t_{k}})-1$.
  According to inductive assumption of (iii) for $N_{\pi_{r}(h^{t_{k}})}^{t_{k}}|_{\A_{r}}$, we get that $L^{t_{k}}(N_{\pi_{r}(h)}|_{\A_{r}})=N_{\pi_{r}(h^{t_{k}})}^{t_{k}}|_{\A_{r}}$, which means that the intersection of the hyperplane $z_{r}=0$ and $L^{t_{k}}(N_{h})$ equals that of the hyperplane $z_{r}=0$ and $N_{h^{t_{k}}}^{t_{k}}$. Moreover, denote by $U_{h^{t_{k}}}^{s;t_{k}}$ the $s$-th stratum of $N_{h^{t_{k}}}^{t_{k}}$ for $x_{i}$ along direction $r$, then due to the above discussion and induction assumption (iii), $U_{h^{t_{k}}}^{0;t_{k}}=\{L^{t_{k}}(N_{\alpha}[w])|N_{\alpha}[w]\in U_{h}^{0}\}$. Inductively, assume $\{L^{t_{k}}(N_{\alpha}[w])|N_{\alpha}[w]\in U_{h}^{l}\}$ is a subset of $U_{h^{t_{k}}}^{l;t_{k}}$ for any $l\in[0,s-1]$, then the intersection of $z_{r}=s$ and $\sum\limits_{l=0}^{s-1}\sum\limits_{N_{\alpha}[w]\in U_{h}^{l}}L^{t_{k}}(N_{\alpha}[w])$ is a sub-polytope of the intersection of $z_{r}=s$ and $\sum\limits_{l=0}^{s-1}\sum\limits_{N_{\alpha}^{t_{k}}[w]\in U_{h^{t_{k}}}^{l;t_{k}}}N_{\alpha}^{t_{k}}[w]$, which leads to $\{L^{t_{k}}(N_{\alpha}[w])|N_{\alpha}[w]\in U_{h}^{s}\}$ is a subset of $U_{h^{t_{k}}}^{s;t_{k}}$ according to Construction \ref{construction}. Hence in conclusion $L^{t_{k}}(N_{h})$ is a sub-polytope of $N_{h^{t_{k}}}^{t_{k}}$, that is $L^{t_{k}}(\rho_{h})$ is a summand of $\rho_{h^{t_{k}}}^{t_{k}}$.

  By (i), the first part of (ii) and the inductive assumption of (iii) for sub-polytopes, we have $L^{t_{k}}(\rho_{h})\in \mathcal{U}_{\geqslant 0}(\Sigma)$.

  On the other hand as we mentioned in the proof of the first part of (ii), according to the construction, the only non-zero summand of $\rho_{h^{t_{k}}}^{t_{k}}$ in $\mathcal{U}^{+}_{\geqslant 0}(\Sigma_{t_{k}})$ is itself. So $L^{t_{k}}(\rho_{h})=\rho^{t_{k}}_{h^{t_{k}}}$.

  (iv)\;Note that this result does not depend on the choice of $\epsilon_{k}$. In the construction of $N_{h}$, we can choose $i\neq k$ and get $x_{i}\rho_{h}=\sum\limits_{N_{\alpha}[w]\in\bigcup\limits_{s} U_{h}^{s}}Y^{w}\rho_{\alpha}$.
  By inductive assumption, (iv) holds for each $N_{\alpha}$, i.e., for any $p\in N_{\alpha}$, all points $q$ in the $k$-section of $N_{\alpha}$ at $p$ with $m_{k}(q)>0$ compose an interval. Therefore for any $p\in N_{h}$, all points $q$ in the $k$-section of $N_{h}$ at $p$ with $m_{k}(q)>0$ compose a union of several intervals. Moreover, due to the definition of $\rho_{h}$, particularly the definition of strata and (\ref{equation: decomposition of rho_h for x_i}) as well as Lemma \ref{m_j in rank 2} (look at $\{j,k\}$-sections for $j\neq k$), we can see that the union of these intervals is still an interval. Hence (iv) holds.

  {\bf Second part of (ii).} \;Now we prove $\rho_{h}$ is the unique one to satisfy the conditions in (ii) by showing that $\rho_{h}=f$ if $f$ is an indecomposable formal Laurent polynomial in $\widehat{\mathcal{U}}_{\geqslant0}(\Sigma)$ which has $X^{h}$ as a summand and whose support is contained in $supp(N_{h})$.

  Because of (iii) and (iv), for large enough $a\in\N$, we have $a\rho_{h}-f\in\widehat{\mathcal{U}}_{\geqslant0}(\Sigma)$. Since  $f\in\widehat{\mathcal{U}}_{\geqslant0}(\Sigma)$, $a\rho_{h}-f$ is a summand of $a\rho_{h}$. However, according to the construction of $\rho_{h}$, the set of summands of $a\rho_{h}$ in $\widehat{\mathcal{U}}_{\geqslant0}(\Sigma)$ equals $\{a\'\rho_{h}|a\'\in\N\text{ and }a\'\leqslant a\}$ (the proof is analogous to that of the indecomposability of $\rho_h$ in $\widehat{\mathcal{U}}_{\geqslant0}(\Sigma)$ above). Therefore, due to the indecomposibility of $f$, $f=\rho_{h}$. In summary, we get that the formal Laurent polynomial satisfying conditions in (ii) uniquely exists.

  (v)  We concentrate on $N_{h}$ which is finite. Denote $E^{0}=E_{h}\cap E(N_{h})$ and $E^{i}=\{l\in E(N_{h})\mid l\notin \bigcup\limits_{j=0}^{i-1}E^{j}$, and it is an edge of a 2-dimensional face of $N_{h}$ which contains an edge in $\bigcup\limits_{j=0}^{i-1}E^{j}$ parallel to $v$ for each vector $v$ in its lattice generating set based on its minimal vertex $\}$. Following (i), it can be seen that $E(N_{h})=\bigcup\limits_{j\in\N}E^{j}$ by induction.

  Next we prove (v) by induction on the dimension $r$. Our proof contains two parts: the first part is to prove (v) for 2-dimensional faces; the second part is to prove (v) for $r$-dimensional faces under the inductive assumption of faces with dimension less than $r$.

  For any 2-dimensional face $S$ of $N_{h}$ which contains an edge in $E^{0}$ parallel to $v$ for each vector $v$ in its lattice generating set based on its minimal vertex, $ldim(S)=2$ and $S$ has two non-parallel edges in $E^{0}$. Then there is $i\neq j\in[1,n]$ satisfying that $S$ is parallel to the 2-dimensional plane determined by $e_i$ and $e_j$ (in fact, $i$ and $j$ are the labels of the above two edges in $S$ assigned in the definition of $E_{h}$). Then due to the construction of $\rho_{h}$ and inductive assumption of (v) for sub-polytopes, there is $h\'\in\Z^{2},w\in\N^{n}$ such that $S=\gamma(N_{h\'}|_{\A\'})[w]$, where $\gamma$ is the canonical embedding of $\R e_i\oplus\R e_j$ into $\R^n$ and $\A\'$ is a cluster algebra with principal coefficients whose initial exchange matrix is obtained from $B$ by deleting all but the $i$-th and $j$-th rows and columns. Hence (v) holds for such faces.

  Recall that $E^{1}$ consists of edges $l$ of faces $S$ in the above case which are not included in $E^{0}$. In the proof of Lemma \ref{face for rank 2} (ii), we divide the above faces $S$ into two kinds, those with $(u,v)=(0,0)$ and with $(u,v)>(0,0)$ respectively (here we use the same notation with that in the proof of Lemma \ref{face for rank 2}). When $(u,v)=(0,0)$, there is a mutation sequence $\mu$ such that $N^{\mu(t_{0})}_{h^{\prime \mu(t_{0})}}|_{\A\'}$ is of dimension 1. Therefore, the face $S$ correlates to an edge in $E^{0}$ of $N^{\mu(t_{0})}_{h^{\mu(t_{0})}}$ under mutation sequence $\mu$, which means that $l$ correlates to an edge of $N^{\mu(t_{0})}_{h^{\mu(t_{0})}}$ parallel to some coordinate axis, where $\mu(t_{0})$ represents the vertex connected to $t_{0}$ by the path induced from $\mu$. When $(u,v)>(0,0)$, there is no point in $l$ other than two vertices.

  For any face $S$ of $N_{h}$, let $s$ be the minimal integer such that $S$ contains an edge in $\bigcup\limits_{j=0}^{s-1}E^{j}\}$ parallel to $v$ for each vector $v$ in its lattice generating set based on its minimal vertex. We take induction on $s$ to prove an edge $l$ of $S$ in $\bigcup\limits_{j=0}^{s-1}E^{j}\}$ satisfies one of the following conditions:

  (a) We can find a mutation sequence $\mu(l)=\mu_{i_{p}}\circ\cdots\circ\mu_{i_{1}}$ such that there is an edge $l_{p}$ in $N^{\mu(t_{0})}_{h^{\mu(t_{0})}}$ parallel to some coordinate axis and correlated to $l$ under $\mu(l)$.

  (b) $l$ does not correlate to an edge parallel to some coordinate axis under any mutation sequence and there are only two points on $l$.

  \noindent hence $S$ satisfies one of the following conditions:

  (A) Every edge of $S$ satisfies condition (a), then there is a sequence of faces $s_0=S,\cdots,S_p=S_l$ satisfying $S_s$ is a face correlated to $S_{s-1}$ under $\mu_{i_s}$, a non-negative projection $\tau_S:N_{f}|_{\A\'}\rightarrow S$ and a non-negative projection $\tau_l:\mu(l)\'(N_{f}|_{\A\'})\rightarrow S_l$ for each edge $l$ of $S$ such that $i=j$ when $\tilde{\tau}_S(e_i)=e_j$ or $\tilde{\tau}_l(e_i)=e_j$,
  \[b_{ik}(l)=\sum\limits_{s=1}^n b_{is}^tv_{sk;l},\]
  where $\A\'$ is a cluster algebra of rank $\max\{ldim(S_s)|s\in[0,p]\}$ with principal coefficients associated to the initial exchange matrix $B\'$, $\mu(l)\'$ is induced from $\mu(l)$ by deleting mutations $\mu_{i_s}$ if there does not exist a segment parallel to $e_{i_s}$ in $S_{s-1}$ and $S_s$, $b_{ik}(l)$ represents the corresponding entry of $\mu(l)\'(B\')$ and $\tilde{\tau}_l(e_k)=(v_{1k;l},\cdots,v_{nk;l})$.

  (B) There is an edge $l$ of $S$ satisfies condition (b), so there is a non-negative projection $\tau_S:N_{f}|_{\A\'}\rightarrow S$ such that the support of $N_f|_{\A\'}$ is one of the following forms in Figure \ref{figure: support of N_f} with horizontal edges representing polytopes lying in the hyperplane $z_i=0$ up to translations, where $\tilde{\tau}_S(e_i)$ parallel to $l$ and $\A\'$ is a cluster algebra of rank $ldim(S)$ with principal coefficients associated to the initial exchange matrix $B\'$ such that the $i$-th row of $B\'$ does not affect $N_f|_{\A\'}$. The other rows of $B\'$ can be determined by restricting $\tau_S$ to polytopes represented by horizontal edges in Figure \ref{figure: support of N_f}.
  \begin{figure}[H]
    \centering
    \includegraphics[width=70mm]{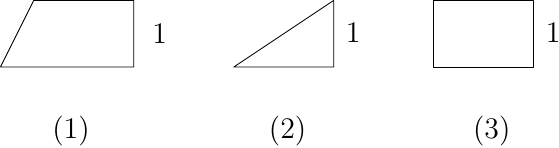}
    \caption{Three forms of supports of $N_{f}|_{\A^{\prime}}$}\label{figure: support of N_f}
  \end{figure}

  The $s=0$ case has been showed above. Assume the claim is true for less than $s$ case, and now we deal with the $s$ case.

  For any $f\in\Z^{n}$, due to the mutation formula (\ref{equation: mutation of x}), (\ref{equation: mutation of y}) and $\hat{y}_{j}=y_{j}\prod\limits_{i=1}^{n}x_{i}^{b_{ij}}$, (iii) indicates that any face $R$ of $N_{f}$ correlates either to an face of $N_{f^{t_{k}}}^{t_{k}}$ isomorphic to $R$ if there does not exist a segment parallel to the $e_k$ in $R$ or a face isomorphic to $\mu_{k}(R)$ under the mutation in direction $k$ otherwise.

  According to the inductive assumption, an edge $l$ of $S$ in $E^{s}$ is in a 2-dimensional face $S\'$ of $N_{h}$ satisfying (A) or (B) which contains an edge in $\bigcup\limits_{i=0}^{s-1}E^{i}$ parallel to $v$ for each vector $v$ in its lattice generating set based on its minimal vertex, hence there is a
  non-negative projection $\tau_{S\'}: N_{f\'}|_{\A^{\prime\prime}}\rightarrow S\'$ and a 2-dimensional face $S^{\prime\prime}$ such that $\tau_{S\'}$ maps an edge $l\'$ of $S^{\prime\prime}$ to $l$. If either $S\'$ satisfies (A) and $l\'$ satisfies (b) or $S\'$ satisfies (B), $l$ has no point other than its vertices and according to the construction of $N_{h}$, inductive assumption and lemma \ref{face for rank 2} (ii), $l$ satisfies (b) and there is a projection from a polytope whose support is one of the forms in Figure \ref{figure: support of N_f} to $S$ since there is no interior point in $S$. During the construction of this polytope, the $i$th-row of exchange matrix is not used, where $i$ is the index such that an edge parallel to $e_i$ is mapped to $l$. So $S$ satisfies (B) in this case.

  Otherwise if $S\'$ satisfies (A) and $l\'$ satisfies (a) for any edge $l$ of $S$, there is a mutation sequence $\mu(l)\'$ such that $l'$ correlates to an edge parallel to some $e_i$ under $\mu(l)\'$, which leads to a mutation sequence $\mu(l)$ such that $l$ correlates to an edge parallel to $e_i$ under $\mu(l)$. Hence $l$ satisfies (a). Denote by $t=\mu(l)(t_0)$ the vertex connected to $t_0$ by the path induced from $\mu(l)$ and by $S^t$ the face correlated to $S$ containing an edge $l^t$ parallel to $e_i$ and correlated to $l$ under $\mu(l)$. Due to the construction of $N_{h^t}^t$, for each point $p\in S^t$, each Laurent monomial $\hat{Y}_t^pX_t^{h^t}$ should be contained in some Laurent polynomial of the form $x_{i;t}^{deg_{x_{i;t}}(\hat{Y}_t^pX_t^{h^t})}M_{i;t}^{[-deg_{x_{i;t}}(\hat{Y}_t^pX_t^{h^t})]_+}X_t^{\alpha}$. Also note that $deg_{x_{i;t}}(\hat{Y}_t^qX_t^{h^t})-deg_{x_{i;t}}(\hat{Y}_t^pX_t^{h^t})=(q-p)(B_{i\cdot}^t)^\top$ for any two points $p,q\in S^t$, where $B_{i\cdot}^t$ denotes the $i$-th row of $B_t$. Hence if every edge of $S$ satisfies condition (a), there is a projection $\tau_S:N_{f}|_{\A\'}\rightarrow \check{S}$, where $\A\'$ is a cluster algebra of rank $ldim(S)$ with principal coefficients associated to $B\'$ satisfying
  \[b_{ik}(l)=\sum\limits_{s=1}^n b_{is}^tv_{sk;l},\]
  and $\check{S}$ is a subpolytope of $S$. Moreover, for each $l$, according to (i) and (iii), we may take an induction on the partial order of polytopes for those in the decomposition of $\rho_{h^t}^tx_{i;t}^a$ for $a\in\N$ to confirm that $S^t$ is correlated to $\check{S}$ under $\mu(l)\'$, hence $\check{S}=S$ and $S$ satisfies (B).

\end{Proof}

According to Theorem \ref{properties in case tsss} (iii), we can extend Theorem \ref{properties in case tsss} (ii) to the following corollary, which shows that by constructing $\rho_{h}$ satisfying certain properties locally for vertices connected to $t_{0}$, we succeed in obtaining some element satisfying these properties globally for any vertices in $\T_{n}$ as said in subsection 3.1.

\begin{Corollary}\label{independent of initial cluster}
  Let $\A$ be a cluster algebra having principal coefficients, $h\in\Z^{n}$ and $t\in\T_{n}$. Then

  (i)\;$L^{t}(\rho_{h})$ is indecomposable in $\mathcal{U}^{+}_{\geqslant0}(\Sigma_{t})$. Therefore, $\rho_{h}$ is universally indecomposable.

  (ii)\;Both $\widehat{\mathcal{P}}=\{L^{t_{0}}(\rho_{h}^{t})|h\in\Z^{n}\}$ and $\mathcal{P}=\widehat{\mathcal{P}}\cap \N\P[X^{\pm 1}]$ are independent of the choice of $t$.

  (iii)\;The set $\{X_{t}^{\alpha}\mid \alpha\in\N^{n},t\in\T_{n}\}$ consisting of coefficient free cluster monomials is a subset of $\mathcal P$.
\end{Corollary}

\begin{Remark}
  Analogous conclusions of Theorem \ref{properties in case tsss} (iii) and Corollary \ref{independent of initial cluster} were proved in \cite{LLZ} and \cite{LLS} for Newton polytopes of cluster variables in rank $2$ and rank $3$ cases respectively. This is not a coincidence. According to Theorem \ref{expression of a cluster variable}, up to multiplying a Laurent monomial $X_{t_{0}}$, $F_{l;t}$ determines the Laurent expression of $x_{l;t}$ in $X_{t_{0}}$. So the Newton Polytope of $F_{l;t}$ determines that of $x_{l;t}$ up to a translation along the exponential vector of the above Laurent monomial.

  Theorem \ref{properties of N_h for skew-symmetrizable} in the sequel can be regarded as an enhanced version of Theorem \ref{properties in case tsss} (v) for skew-symmetrizable case. Both of them are inspired by Theorem 6.8 of \cite{F} for Newton polytopes associated to modules, which presents a specific relation between such a polytope and its facets.
\end{Remark}

Due to Theorem \ref{properties in case tsss} (iii), we have the following definition:

\begin{Definition}
  The polytope $N_{h^{t_{k}}}^{t_{k}}$ of $\rho_{h^{t_{k}}}^{t_{k}}$ is called the \textbf{mutation of the polytope $N_{h}$ of $\rho_h$} in direction $k$, and it is denoted as $\mu_{k}(N_{h})$.
\end{Definition}

Please refer to \cite{P} for more details about polytope mutations.

The relation among $N_h$, $\mu_{k}(N_{h})$ and their corresponding Laurent polynomials is shown in Figure \ref{figure: polytope mutation}:
  \begin{figure}[H]
  \centering
  \includegraphics[width=30mm]{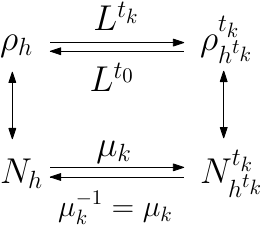}
  \caption{The mutation of $N_{h}$ in direction $k$.}\label{figure: polytope mutation}
  \end{figure}

\begin{Example}
  (i)\;Construction \ref{construction} is compatible with the definition of $\rho_{h}$ in Subsection 3.2 for any $h\in\Z^{2}$ as the former in fact follows the equations (\ref{equation: decomposition of x1Ph}) and (\ref{equation: decomposition of x2Ph}). Let the cluster algebra and $h=(-5,7)$ be those in Example \ref{eg in rank 2}. Following Construction \ref{construction}, choose $k=2$ and $i=1$, so $b_{ik}>0$, then

  (1)\;We have $\rho_{-5}=x_{1}^{-5}(1+y_{1})^{5}$ and $x_{1}\rho_{-5}=x_{1}^{-4}(1+y_{1})^{4}+y_{1}x_{1}^{-4}(1+y_{1})^{4}=\rho_{-4}+y_{1}\rho_{-4}$ for $\A_{2}$. Hence we get $U^{0}_{(-5,7)}=\{N_{(-4,7)},N_{(-4,4)}[(1,0)]\}$.

  (2)\;Choose $a=7$. The intersection of hyperplane $z_{2}=1$ and $N_{(-4,7)}+N_{(-4,4)}[(1,0)]$ equals $9\gamma_{2;1}(N_{-2}[2])+8\gamma_{2;1}(N_{-2}[3])$. Since $\tilde{S}_{1}=\{\gamma_{2;1}(N_{-3}[1]),\gamma_{2;1}(N_{-3}[2])\}$ and $(5,1)$ is the unique point with minimal $x_{2}$-degree and weight 8 in the intersection, we get $8\gamma_{2;1}(N_{\alpha\'}[w\']|_{\A_2})=8\gamma_{2;1}(N_{-3}[2]|_{\A_{2}})=N_{1}$. So $U_{1}=\emptyset$. While $N_{2}=\gamma_{2;1}(N_{-2}[2])|_{\A_{2}}$, thus we obtain $U_{2}=\{N_{(-2,4)}[(1,1)]\}$ and $N_{3}=\emptyset$. Therefore $U^{1}_{(-5,7)}=U_{2}=\{N_{(-2,4)}[(1,1)]\}$. Inductively we can also get $U^{2}_{(-5,7)}=\{(N_{(0,1)}[(2,2)])^{2}\}$ and $U^{i}_{(-5,7)}=\emptyset$ for $i>2$.

  (3)\;Then we get that $N_{(-5,7)}=N_{(-4,7)}+N_{(-4,4)}[(1,0)]+N_{(-2,4)}[(1,1)]+2N_{(0,1)}[(2,2)]$,
  which coincides with the polytope associated to $h=(-5,7)$ presented in Example \ref{eg in rank 2}.

  (ii)\;In Example \ref{example of essential skeleton}, we can calculate $N_{h}$ as well as $\rho_{h}$ according to Construction \ref{construction}. Choose $k=3$ and $i=1$, then

  (1)\;The essential skeleton is the set consisting of red edges in Figure \ref{example2} as calculated in Example \ref{example of essential skeleton}. Moreover we get
  \[U_{E_{h}}=\{N_{-4,2,4},N_{-4,1,5}[(1,0,0)],N_{-3,1,3}[(3,1,1)]\}\]
  and
  \[\mathcal{R}=\{\gamma_{3;0}(N_{-5,2}|_{\A_{3}}),\gamma_{3;3}(N_{-2,3}[(5,3)]|_{\A_{3}})\}.\]

  We have $x_{1}\rho_{(-5,2)}=\rho_{(-4,2)}+y_{1}\rho_{(-4,1)}+y_{1}y_{2}\rho_{(-2,1)}$ for $\A_{3}$. Hence \[U^{0}_{(-5,2,4)}=\{N_{(-4,2,4)},N_{(-4,1,5)}[(1,0,0)],N_{(-2,1,1)}[(1,1,0)]\}.\]

  (2)\;The intersection of hyperplane $z_{3}=1$ and $N_{(-4,2,4)}+N_{(-4,1,5)}[(1,0,0)]+N_{(-2,1,1)}[(1,1,0)]+N_{-3,1,3}[(3,1,1)]$ (or $N_{(-4,2,4)}+N_{(-4,1,5)}[(1,0,0)]+N_{(-2,1,1)}[(1,1,0)]$ respectively) is $2\gamma_{3;1}(N_{(-3,1)}[(3,1)]|_{\A_3})+2\gamma_{3;1}(N_{(-1,1)}[(3,2)]|_{\A_3})+3\gamma_{3;1}(N_{(-1,0)}[(4,2)]|_{\A_3})$ (or $\gamma_{3;1}(N_{(-3,1)}[(3,1)]|_{\A_3})+2\gamma_{3;1}(N_{(-1,1)}[(3,2)]|_{\A_3})+3\gamma_{3;1}(N_{(-1,0)}[(4,2)]|_{\A_3})$ respectively).

  Since $\mathcal{R}_{1,+}=\{\gamma_{3;0}(N_{(-5,2)}|_{\A_{3}})\}$ while $\mathcal{R}_{1,-}=\{\gamma_{3;3}(N_{(-2,3)}[(5,3)]|_{\A_{3}})\}$, we can choose $a=9$. Then it can be calculated to see that $$\tilde{S}_{1}=\{\gamma_{(3;1)}(N_{-6,4}|_{\A_{3}}),\gamma_{3;1}(N_{(-6,3)}[(1,0)]|_{\A_{3}}),\gamma_{3;1}(N_{(-4,3)}[(1,1)]|_{\A_{3}}), \gamma_{3;1}(N_{(-4,2)}[(2,1)]|_{\A_{3}}),\gamma_{3;1}(N_{(-2,1)}[(3,2)]|_{\A_{3}})\}.$$
  Therefore, since $(6,3,1)$ is the unique point with minimal $x_{3}$-degree and weight 2 in the intersection, so we get $2\gamma_{3;1}(N_{\alpha\'}[w\']|_{\A_3})=2\gamma_{3;1}(N_{(-4,2)}[(2,1)]|_{\A_3})=2\gamma_{3;1}(N_{(-3,1)}[(3,1)]|_{\A_3})+2\gamma_{3;1}(N_{(-3,2)}[(2,1)]|_{\A_3})$ and $N_{1}=2\gamma_{3;1}(N_{(-3,1)}[(3,1)]|_{\A_3})$. Hence $U_{1}=\{N_{(-3,1,3)}[(3,1,1)],(N_{(-3,2,2)}[(2,1,1)])^{2}\}$. Then for $2\gamma_{3;1}(N_{(-3,1)}[(3,1)])+2\gamma_{3;1}(N_{(-1,1)}[(3,2)])+3\gamma_{3,1}(N_{(-1,0)}[(4,2)])-N_{1}=2\gamma_{3;1}(N_{(-1,1)}[(3,2)])+3\gamma_{3,1}(N_{(-1,0)}[(4,2)])$, we get $3\gamma_{3;1}(N_{(-2,1)}[(3,2)]|_{\A_3})=3\gamma_{3;1}(N_{(-1,1)}[(3,2)]|_{\A_3})+3\gamma_{3;1}(N_{(-1,0)}[(4,2)]|_{\A_3})$, $N_{2}=2\gamma_{3;1}(N_{(-1,1)}[(3,2)])+3\gamma_{3,1}(N_{(-1,0)}[(4,2)])$ and $U_{2}=\gamma_{3;1}(N_{(-1,1)}[(3,2)]|_{\A_3})$. As $2\gamma_{3;1}(N_{(-3,1)}[(3,1)])+2\gamma_{3;1}(N_{(-1,1)}[(3,2)])+3\gamma_{3,1}(N_{(-1,0)}[(4,2)])-N_{1}-N_{2}=\empty$, so
  \[U^{1}_{(-5,2,4)}=U_{1}\cup U_{2}=\{N_{(-3,1,3)}[(3,1,1)],(N_{(-3,2,2)}[(2,1,1)])^{2},N_{(-1,1,-1)}[(3,2,1)]\}.\]

  Analogously, we can continue to obtain that $$\tilde{S}_{2}=\{\gamma_{3;2}(N_{(-7,6)}|_{\A_{3}}),\gamma_{3;2}(N_{(-5,5)}[(1,1)]|_{\A_{3}}),\gamma_{3;2}(N_{(-3,3)}[(3,2)]|_{\A_{3}}), \gamma_{3;2}(N_{(-3,2)}[(4,2)]|_{\A_{3}}),\gamma_{(3;2)}(N_{(-1,1)}[(5,3)]|_{\A_{3}})\},$$
  while the intersection of $z_{3}=2$ and $N_{(-4,2,4)}+N_{(-4,1,5)}[(1,0,0)]+N_{(-2,1,1)}[(1,1,0)]$ equals $\gamma_{3;2}(N_{(-2,3)}[(3,2)]|_{\A_3})+\gamma_{3;2}(N_{(-2,1)}[(5,2)]|_{\A_3})+3\gamma_{3;2}(N_{(0,1)}[(5,3)]|_{\A_3},4\gamma_{3;2}(N_{(0,0)}[(5,4)]|_{\A_3})$. Then we can calculate iteratively to get
  \[U^{2}_{(-5,2,4)}=\{(N_{(-2,2,0)}[(4,2,2)])^{2}\}.\]
  Moreover, $U^{i}_{(-5,2,4)}=\emptyset$ when $i>2$.

  (3)\; Then we get that $N_{(-5,2,4)}=N_{(-4,2,4)}+N_{(-4,1,5)}[(1,0,0)]+N_{(-2,1,1)}[(1,1,0)]+N_{(-3,1,3)}[(3,1,1)]+2N_{(-3,2,2)}[(2,1,1)]+N_{(-1,1,-1)}[(3,2,1)]+2N_{(-2,2,0)}[(4,2,2)]$ as showed in Figure \ref{example2} (left) below.

  \begin{figure}[H]
    \centering
    \includegraphics[width=100mm]{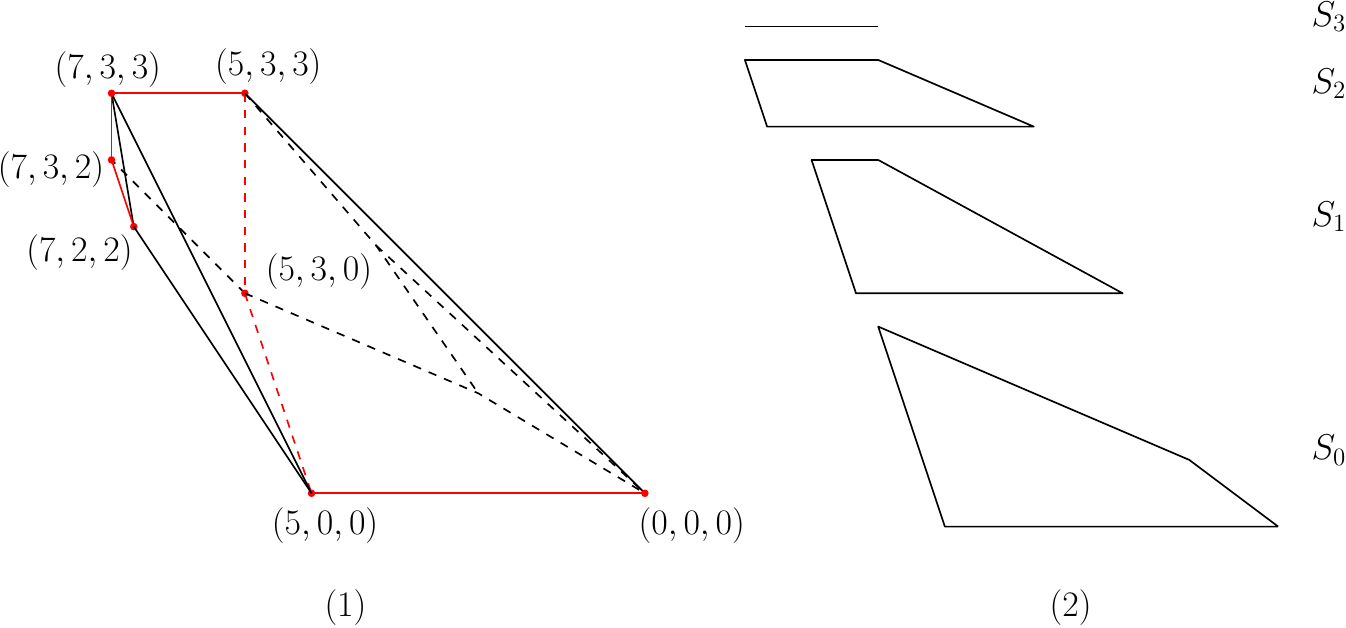}
    \caption{The polytope $N_{h}$ and its intersections with hyperplanes $z_{3}=i$ for $i\in[0,3]$.}\label{example2}
  \end{figure}
  The intersections $S_{i}$ of $N_{h}$ and hyperplanes $z_{3}=i$ are
  \[S_{0}=\{\gamma_{3;0}(N_{(-5,2)})\},\quad S_{1}=\{(\gamma_{3;1}(N_{(-4,2)}[(2,1)]))^{2},(\gamma_{3;1}(N_{(-2,1)}[(3,2)]))^{3}\},\]
  \[S_{2}=\{\gamma_{3;2}(N_{(-3,3)}[(3,2)]), \gamma_{3;2}(N_{(-3,2)}[(4,2)]),(\gamma_{3;2}(N_{(-1,1)}[(5,3)]))^{3}\} \text{ and } S_{3}=\{\gamma_{3;3}(N_{(-2,3)}[(5,3)])\}\]
  respectively, see Figure \ref{example2} (right). And the corresponding Laurent polynomial is
  \begin{figure}[H]
    \centering
    \includegraphics[width=150mm]{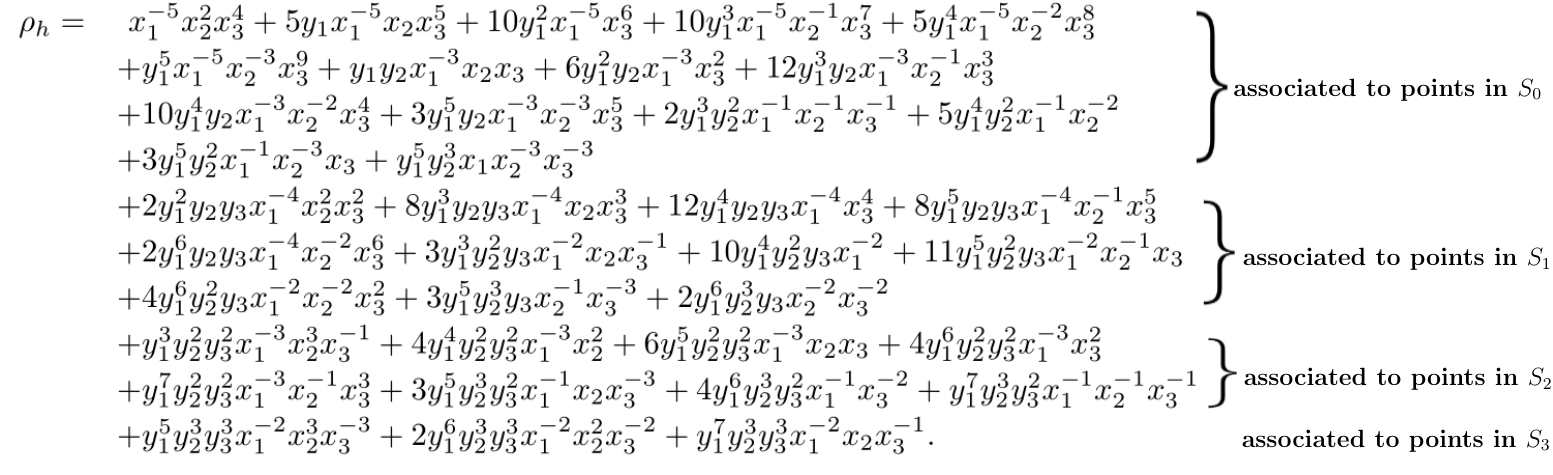}
  \end{figure}
\end{Example}

\begin{Remark}\label{remark after the theorem}
  As in the proof of Theorem \ref{properties in case tsss}, We define $F_{h}=\rho_{h}|_{x_{i}\rightarrow 1,\forall i\in[1,n]}$. When $h=g_{l;t}$, the above $F_h$ is exactly the $F$-polynomial of $x_{l;t}$. So, the polynomial $F_h$ defined here is a generalization of the $F$-polynomial associated to a cluster variable. Hence we call $F_h$ {\bf the $F$-polynomial associated to the vector $h$}.

  For any circle $\gamma$ in the exchange graph with endpoint $t$, $L^{t;\gamma}(f)=qf$ for any homogeneous formal Laurent polynomial $f\in\bigcap\limits_{t\'\in\gamma}\Z[Y_{t\'}][[X_{t\'}^{\pm 1}]]\in\Z Trop(Y_{t})[[X_{t}^{\pm1}]]$ by the definition of $L^{t;\gamma}$, where $q\in Trop(Y_{t})$; on the other hand $L^{t;\gamma}(\rho_{h})=\rho_{h}$ following Theorem \ref{properties in case tsss} (iii). So $q=1$ and thus $L^{t;\gamma}$ only depends on the endpoints of $\gamma$ for any path $\gamma$ in $\T_{n}$, which justifies the notation $L^{t}$.

  According to Theorem \ref{expression of a cluster variable}, it is natural to generalize the definition of $\rho_{h}$ for a cluster algebra over an arbitrary semifield $\P$ as
  \[\rho_{h}|_{\P}:=\frac{F_{h}|_{\F}(\hat{Y})}{F_{h}|_{\P}(Y)}X^{h}\in\N\P[[X^{\pm 1}]].\]
  This does not depend on the choice of the initial vertex $t_{0}$ due to Theorem \ref{properties in case tsss} (iii) and the fact that $L^{t_{k}}(\rho_{h})=\frac{F_{h}|_{\F}(\hat{Y})}{F_{h}|_{Trop(Y_{t_{k}})}(Y)}X^{h}$.

  Obviously, $\rho_{h}|_{\P}$ is related to the choice of the semifield $\P$. However we in general omit the subscript of semifield if there is no risk of confusion. For example when we talk about $\rho^{t}_{h}$ and $\rho^{t\'}_{h}$, the semifields are $Trop(Y_{t})$ and $Trop(Y_{t\'})$ respectively.

  For a cluster algebra $\A$ of rank $n$ over semifield $\P$, denote by $\widehat{\mathcal{P}}|_{\P}$ the set consisting of all such formal Laurent polynomials $\rho_h|_{\P}$, i.e.,
  \[\widehat{\mathcal{P}}|_{\P}=\{\rho_{h}|_{\P}\in\N\P[[X^{\pm 1}]]\mid h\in\Z^{n}\}.\]
  While $\mathcal{P}|_{\P}$ is a subset of $\widehat{\mathcal{P}}|_{\P}$ such that
  $$\mathcal{P}|_{\P}=\widehat{\mathcal{P}}|_{\P}\cap\N\P[X^{\pm 1}].$$
  It follows that both $\widehat{\mathcal{P}}|_{\P}$ and $\mathcal{P}|_{\P}$ do not depend on the choice of the initial vertex $t_{0}$.

  Define $H$ to be the index set of $\mathcal{P}|_{\P}$, i.e., $H:=\{h\in\Z^{n}\mid \rho_{h}|_{\P}\in\mathcal{P}\}$. This set is independent of the choice of semifield $\P$, and according to Theorem \ref{properties in case tsss} (iii), there is a canonical bijection between $H^{t}$ and $H^{t\'}$ for any $t,t\'\in\T_{n}$ sending $h^{t}$ to $h^{t\'}$.
\end{Remark}

\begin{Corollary}\label{unique maximal and minimal term}
  Let $\A$ be a cluster algebra having principal coefficients, $h\in H$. Then there is a unique maximal point $p$ in $N_{h}$. Hence the $F$-polynomial $F_{h}$ has a unique term $Y^{p}$ with maximal $Y$-degree as well as a constant term and $co_{Y^{p}}(F_{h})=co_{1}(F_{h})=1$.
\end{Corollary}
\begin{Proof}
  According to Theorem \ref{properties in case tsss} (v) and Lemma \ref{face for rank 2} (ii), for two vertices $p$ and $q$ in an edge of $N_{h}$, we always have $p>q$ or $p<q$. And according to Remark \ref{shape for rank 2}, the corollary is easy to be verified when $N_{h}$ has dimension 2.

  In this proof, when we say a path in $N_{h}$, we always mean a sequence of directed edges $\underline{l_{1}\cdots l_{s}}$ in $N_{h}$ satisfying that it does no contain any oriented circle and the target of $l_{i}$ equals the source of $l_{i+1}$ for any $i\in[1,s-1]$. For two paths $\zeta$ and $\zeta\'$, we say {\bf $\zeta\'$ lies locally above $\zeta$} if for any vertices $p$ in $\zeta\'$, there is a vertices $q$ in $\zeta$ satisfying $p\geqslant q$. Note that it is possible for us to choose the same $q$ for various vertices $p$ when $\zeta'$ lies locally above $\zeta$.  Moreover if $\zeta\'$ lies locally above $\zeta$ but $\zeta$ does not lies locally above $\zeta\'$, then we say $\zeta\'$ {\bf strictly lies locally above $\zeta$.}

  Next we prove by contradiction that there is a unique maximal point $p$ in $N_{h}$. So we assume there are two different maximal points $p$ and $q$ in $N_{h}$.

  Choose a path in $N_{h}$ from $p$ to $q$ such that there is no path strictly lying locally above it. The existence of such path is ensured by the finiteness of $V(N_{h})$. Denote it by $\underline{l_{1} l_{2} \cdots l_{r}}$. Then there is a minimal vertex $p_{2}$ in the path with respective to ``$<$''. Let $s\in[1,r-1]$ such that $p_{1}$ and $p_{2}$ are the vertices of $l_{s}$ while $p_{2}$ and $p_{3}$ are those of $l_{s+1}$. Thus $p_{1}>p_{2}$ and $p_{3}>p_{2}$ by Theorem \ref{properties in case tsss} (v) and Lemma \ref{face for rank 2} (ii). Then we can find a sequence of edges $f_{0},\cdots,f_{j}$ having $p_{2}$ as a common vertex and a sequence of 2-dimensional faces $S_{1},\cdots,S_{j}$ of $N_{h}$ satisfying that $f_{0}$ equals $l_{s}$, $f_{j}$ equals $l_{s+1}$ if forgetting direction, and $f_{i-1},f_{i}$ are two edges in $S_{i}$. Such sequences of edges and faces exist as we can look at the neighborhood of $p_{2}$ in $N_{h}$, which is a {\em cone} with several edges of finite length including $f_{0}$ and $f_{j}$, then we can clockwise or counterclockwise enumerate edges from $f_{0}$ to $f_{j}$ on the {\em cone}.

  Moreover, because of the convexity of $N_{h}$ and the fact that $p_{1}>p_{2}$ as well as $p_{3}>p_{2}$, we can further more find $f_{0},\cdots,f_{j}$ such that $p>p_{2}$, where $p$ is any vertices in these edges other than $p_{2}$. So according to Remark \ref{shape for rank 2} and Theorem \ref{properties in case tsss} (v) we get that $p_{2}$ is the unique minimal point of $S_{i}$ for any $i\in[0,j]$. Then in each $S_{i}$ there is a path from the vertex of $f_{i-1}$ other than $p_{2}$ to the vertex of $f_{i}$ other than $p_{2}$ not passing through $p_{2}$. Connecting these paths we get a path $\zeta$ from $p_{1}$ to $p_{3}$ satisfying that $\zeta$ strictly lies locally above path $\underline{l_{s} l_{s+1}}$. Hence we obtain a new path $\underline{l_{1} \cdots l_{s-1} \zeta l_{s+2}\cdots l_{r}}$ by replacing $\underline{l_{s}l_{s+1}}$ with $\zeta$. Then we obtain a path in $N_{h}$ strictly lying locally above $\underline{l_{1} l_{2}\cdots l_{r}}$, which contradicts the choice of path $\underline{l_{1} l_{2}\cdots l_{r}}$.

  Hence there is a unique maximal point $p$ in $N_{h}$, i.e., the $F$-polynomial $F_{h}$ has a monomial $Y^{p}$ as the unique maximal term.

  The definition of $\rho_{h}$ ensures that the origin is the unique minimal point in $N_{h}$. And obviously the origin and $p$ are both vertices of $N_{h}$. By Theorem \ref{properties in case tsss} (v), if $q$ is a vertex of $N_{h}$, then $co_{q}(N_{h})=1$. Thus $co_{0}(N_{h})=co_{p}(N_{h})=1$.
\end{Proof}

It can be seen that when restricted to $F$-polynomials associated to cluster variables, this corollary is a generalization of the corresponding result in \cite{GHKK} from skew-symmetrizable case to TSSS case, which is equivalent to the sign-coherence of $c$-vectors according to \cite{FZ4} (although the equivalence is proved for skew-symmetrizable cluster algebras in \cite{FZ4}, the proof holds in TSSS case).

\section{Special case: polytopes associated to cluster variables}

We have known in \cite{F} that the Newton polytope of an $F$-polynomial is defined associated to representations of a finite-dimensional basic algebra, and it admits some interesting combinatorial properties. But those cluster algebras, whose categorification have not been found so far, are not suitable for the theory in \cite{F}. In this sense, it is necessary for us to establish the theory of these Newton polytopes for general TSSS cluster algebras. This is one of the motivations for our construction in the last section.

In this section we will take a look at cluster variables, which turn out to be contained in $\mathcal{P}$. Hence the results in last section hold for cluster variables. In particular, we have a recurrence formula and universally positivity for cluster variables as special case polytope functions (Theorem \ref{from general to cluster variables}).

\subsection{Recurrence formula and positivity of a cluster variable}\quad

Denote by $N_{l;t}^{t\'}$ the Newton polytope of an $F$-polynomial $F_{l;t}^{t\'}$. The $F$-polynomial of a cluster variable can be generalized to any cluster monomial $X_{t}^{\alpha}$ as $F_{X_{t}^{\alpha}}=\prod\limits_{i}F_{i;t}^{\alpha_{i}}$. We denote the Newton polytope of $F_{X_{t}^{\alpha}}$ as $N(X_{t}^{\alpha})$.

\begin{Theorem}\label{from general to cluster variables}
  (\textbf{Recurrence formula})  Let $\A$ be a TSSS cluster algebra having principal coefficients, then $x_{l;t}=\rho_{g_{l;t}}$ and $N_{l;t}=N_{g_{l;t}}$. Following this view, we have that
   \begin{equation}
     co_{p}(N_{l;t})=co_{p}(N_{g_{l;t}})=\sum\limits_{N_{\alpha_{j}}[w_{j}]\in \bigcup\limits_{r}U_{g_{l;t}}^{r}}co_{p}(N_{\alpha_{j}}[w_{j}])
   \end{equation}
   and
   \begin{equation}
     x_{l;t}=X^{g_{l;t}}(\sum\limits_{p\in N_{g_{l;t}}}co_{p}(N_{l;t})\hat{Y}^{p}),
   \end{equation}
   where $U_{g_{l;t}}^{r}$, running over all $r$-th strata of the polytope $N_h$ for $x_{i}$ along direction $k$, and hence all $N_{\alpha_{j}}[w_{j}]$ are defined in the construction \ref{construction} (b) with $h=g_{l;t}$.
\end{Theorem}
\begin{Proof}
  According to Corollary \ref{independent of initial cluster}, any cluster monomial $X_{t}^{\alpha}=\rho^{t}_{\alpha}$ is in $\mathcal{P}$ (defined in Corollary \ref{independent of initial cluster}) for any $t\in\T_{n}$,  $\alpha\in\N^{n}$, and $h^{t_{k}}=h-2h_{k}e_{k}+h_{k}[b^{t}_{k}]_{+}+[-h_{k}]_{+}b^{t}_{k}$ for any $t_{k}$ connected to $t$ by an edge labeled $k$, which coincides with the mutation formula of $g$-vectors. Hence we can see that $X_{t}^{\alpha}=\rho_{g(X_{t}^{\alpha})}$ as well as $N(X_{t}^{\alpha})=N_{g(X_{t}^{\alpha})}$. In particular, this is true for any cluster variable. Thus $x_{l;t}$ and $N_{l;t}$ naturally inherit the properties about $\rho_{h}$ and $N_{h}$ claimed in Theorem \ref{properties in case tsss}. And the constructions of $\rho_{h}$ and $N_{h}$ provide a recurrence formulas for $N_{l;t}$ and $x_{l;t}$.
\end{Proof}

Repeating the Recurrence formula in the above theorem, we can expresses the weight of each point in $N_{l;t}$ as a sum of weights of points in some polytopes of the form $N_{h\'}[w]$ with $h\'\in\Z^{2}$ and $w\in\N^{n}$, thus by (\ref{equation: weight of rank2}) it is in fact a sum of certain binomial coefficients, which are naturally non-negative. So this provides a proof of the positivity of cluster variables in a totally sign-skew-symmetric cluster algebra. Then we obtain directly the follows:
\begin{Corollary}\label{TSSS positivity} $(${\bf Positivity for TSSS cluster algebras}$)$
  Let $\A$ be a TSSS cluster algebra with principal coefficients and $(X,Y,B)$ be its initial seed. Then every cluster variable in $\A$ is a Laurent polynomial over $\N[Y]$ in $X$. In particular, the positivity of TSSS cluster algebras holds.
\end{Corollary}
As a special case, the proof of this result provides a new method different from that in \cite{GHKK} to present the positivity of cluster variables in a skew-symmetrizable cluster algebra of rank greater than 2. We use the conclusion of \cite{GHKK} only in the case of rank 2 as the starting of inductive method. In general, the positivity for TSSS cluster algebras can not be proved via the method in \cite{GHKK}.

\begin{Example}\label{example}
  Let $\A$ be a cluster algebra having principal coefficients with the initial seed $(X,Y,B)$, where $X=(x_{1},x_{2},x_{3})$, $Y=(y_{1},y_{2},y_{3})$ and
  \[B=\left(\begin{array}{ccc}
  0 & -3 & 16 \\
  3 & 0 & -6 \\
  -16 & 6 & 0
  \end{array}\right).\]
  Then in the seed $(X_{t},Y_{t},\tilde{B}_{t})=\mu_{3}\circ\mu_{1}\circ\mu_{2}((X,\tilde{B}))$, the Laurent expression of $x_{3;t}$ in $X$ is
  \begin{equation*}
    \begin{array}{rl}
      x_{3;t} & =\frac{y_{3}[y_{2}^{3}y_{1}x_{2}^{3}x_{3}^{2}+(y_{2}x_{3}^{6}+x_{1}^{3})^{3}]^{2}+x_{1}^{2}x_{2}^{6}}{x_{1}^{2}x_{2}^{6}x_{3}} \\
       & =(y_{1}^{2}y_{2}^{6}y_{3}x_{2}^{6}x_{3}^{4}+2y_{1}y_{2}^{6}y_{3}x_{2}^{3}x_{3}^{20}+6y_{1}y_{2}^{5}y_{3}x_{1}^{3}x_{2}^{3}x_{3}^{14}+ 6y_{1}y_{2}^{4}y_{3}x_{1}^{6}x_{2}^{3}x_{3}^{8} \\
       & +2y_{1}y_{2}^{3}y_{3}x_{1}^{9}x_{2}^{3}x_{3}^{2}+y_{2}^{6}y_{3}x_{3}^{36}+ 6y_{2}^{5}y_{3}x_{1}^{3}x_{3}^{30}+15y_{2}^{4}y_{3}x_{1}^{6}x_{3}^{24}+20y_{2}^{3}y_{3}x_{1}^{9}x_{3}^{18}\\
       & +15y_{2}^{2}y_{3}x_{1}^{12}x_{3}^{12}+6y_{2}y_{3}x_{1}^{15}x_{3}^{6}+y_{3}x_{1}^{18}+x_{1}^{2}x_{2}^{6})/x_{1}^{2}x_{2}^{6}x_{3}.
    \end{array}
  \end{equation*}

 Hence the corresponding Newton polytope $N_{3;t}$ is as follows
  \begin{figure}[H]
    \centering
    \includegraphics[width=25mm]{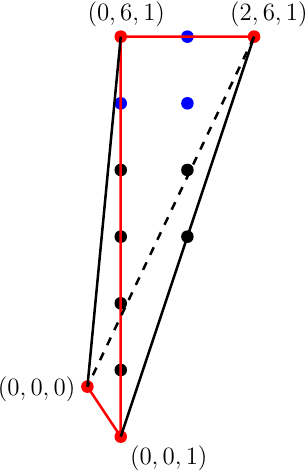}
    \caption{The Newton polytope $N_{3;t}$}
  \end{figure}
  \noindent where the set of points is the support of $F_{3;t}$ and we mark points in $V_{3;t}$ and edges in $E_{3;t}$ in red. It can be seen that the support of $F_{3;t}$ is saturated. Moreover, it can be checked that the constant coefficients of these Laurent monomials satisfy Theorem \ref{from general to cluster variables}. We calculate the coefficients of blue points in $\rho_{g_{3;t}}$ for example, where $g_{3;t}=(0,0,-1)$.
  \[co_{y_{2}^{5}y_{3}x_{1}x_{2}^{-6}x_{3}^{29}}(\rho_{g_{3;t}})=\tilde{C}_{6}^{5}=6=co_{y_{2}^{5}y_{3}x_{1}x_{2}^{-6}x_{3}^{29}}(P_{3;t}),\] \[co_{y_{1}y_{2}^{6}y_{3}x_{1}^{-2}x_{2}^{-3}x_{3}^{19}}(\rho_{g_{3;t}})=\tilde{C}_{2}^{1}=2=co_{y_{1}y_{2}^{6}y_{3}x_{1}^{-2}x_{2}^{-3}x_{3}^{19}}(P_{3;t}),\]
  and
  \begin{equation*}
    \begin{array}{rl}
      co_{y_{1}y_{2}^{5}y_{3}x_{1}x_{2}^{-3}x_{3}^{13}}(\rho_{g_{3;t}}) & =max\{co_{y_{2}^{5}y_{3}x_{1}x_{2}^{-6}x_{3}^{29}}(\rho_{g_{3;t}})\tilde{C}_{-deg_{x_{1}}(y_{2}^{5}y_{3}x_{1}x_{2}^{-6}x_{3}^{29})}^{1}, \\
      &\qquad\qquad co_{y_{1}y_{2}^{6}y_{3}x_{1}^{-2}x_{2}^{-3}x_{3}^{19}}(\rho_{g_{3;t}})\tilde{C}_{-deg_{x_{2}}(y_{1}y_{2}^{6}y_{3}x_{1}^{-2}x_{2}^{-3}x_{3}^{19})}^{1}\} \\
        & =max\{0,6\}\\
        & =6\\
        & =co_{y_{2}^{5}y_{3}x_{1}^{3}x_{2}^{3}x_{3}^{14}}(P_{3;t}).
    \end{array}
  \end{equation*}
  We can calculate to see that $N_{3;t}=N_{g_{3;t}}$ and $x_{3;t}=\rho_{g_{3;t}}$.
\end{Example}

\subsection{On a conjecture of Fei}\quad

Now we show some properties of Newton polytopes associated to cluster variables, which may usually not be true for Newton polytopes associated to general  vectors in $H$.

\begin{Theorem}\label{properties for cluster variable case}
  Let $\A$ be a TSSS cluster algebra having principal coefficients, $l\in[1,n]$ and $t\in\T_{n}$. Then, for $F$-polynomial $F_{l;t}$ associated to $x_{l;t}$ and its corresponding Newton polytope $N_{l;t}$, the following statements hold:

  (i)\; The support of $F$-polynomial $F_{l;t}$ is saturated.

  (ii)\; For any $p\in N_{l;t}$, $co_{p}(N_{l;t})=1$ if and only if $p\in V(N_{l;t})$.
\end{Theorem}
\begin{Proof}
  We prove this theorem by induction on the length of the path connecting $t$ and $t_{0}$. It is trivial when the length is $0$, i.e., $t_{0}=t$. Assume it is true for $t\'\in\T_{n}$. We claim that it is also true for $t_{0}\in\T_{n}$ connected to $t\'$ by an edge labeled $k\in[1,n]$.

  (i)\; Assume on the contrary there is a point $p\in N_{l;t}$ such that $co_{p}(N_{l;t})=0$. If there is no $p_{1}\in N_{l;t}$ such that $\frac{p}{p_{1}}=\hat{y}_{k}^{a}$ and $co_{p_{1}}(N_{l;t})\neq0$ for some $a\in\Z_{>0}$, then according to the mutation formulas, we know that there is $p\'$ correlated to $p$ (with respect to the mutation at direction $k$) satisfying $co_{p\'}(N^{t\'}_{l;t})=0$. Moreover, as $p\in N_{l;t}$, we have $p\'\in N_{l;t}^{t\'}$. Thus $N_{l;t}^{t\'}$ is not saturated, which contradicts our inductive assumption. Similarly, we obtain a contradiction if there is no $p_{1}\in N_{l;t}$ such that $\frac{p}{p_{1}}=\hat{y}_{k}^{a}$ and $co_{p_{1}}(N_{l;t})\neq0$ for some $a\in\Z_{<0}$.

  Therefore, there must be $p_{1},p_{2}\in N_{l;t}$ such that $\frac{p}{p_{1}}=\hat{y}_{k}^{a_{1}}$, $\frac{p}{p_{2}}=\hat{y}_{k}^{a_{2}}$ where $a_{1}\in \Z_{>0},a_{2}\in\Z_{<0}$ and $co_{p_{1}}(N_{l;t})co_{p_{2}}(N_{l;t})\neq 0$. If $deg_{x_{k}}(p)\geqslant 0$, then $deg_{x_{k}}(p_{1})\geqslant 0$ and $deg_{x_{k}}(p_{2})\geqslant 0$. Hence we always have $m_{k}(p)=co_{p}(N_{l;t})=0$ while $m_{k}(p_{1})=co_{p_{1}}(N_{l;t})\neq0$ and $m_{k}(p_{2})=co_{p_{1}}(N_{l;t})\neq0$, which contradicts Theorem \ref{properties in case tsss} (iv). Otherwise if $deg_{x_{k}}(p)< 0$, then due to the definition of $m_{k}$, we can find proper $p_{1}$ and $p_{2}$ satisfying $m_{k}(p_{1})m_{k}(p_{2})\neq 0$, which contradicts Theorem \ref{properties in case tsss} (iv). So in conclusion we get that $N_{l;t}$ is saturated.

  (ii)\; If there is a point $p\in N_{l;t}$ satisfying that $co_{p}(N_{l;t})=1$ but $p\notin V(N_{l;t})$, then according to the mutation formula (\ref{equation: mutation of x}) and (\ref{equation: mutation of y}), there is a point $p\'\in N_{l;t}^{t\'}$ correlated to $p$ with $co_{p\'}(N_{l;t}^{t\'})=1$. Hence $p\'\in V(N_{l;t}^{t\'})$, then due to the mutation formula (\ref{equation: mutation of x}), (\ref{equation: mutation of y}), Theorem \ref{properties in case tsss} (iv) and (i) of this theorem, either $p\in V(N_{l;t})$ or $co_{p}(N_{l;t})>1$, which contradicts our assumption.
\end{Proof}

In the context of categories of representations of a finite-dimensional algebra, the analogous results were given in \cite{F}. Moreover, Jiarui Fei conjectured the following statements hold there.

\begin{Conjecture} \label{F} \cite{F}\;
  Let $\A$ be a TSSS cluster algebra with principal coefficients. Then,

  (i)\;A point $p$ in the Newton polytope associated to any coefficient free cluster monomial is a vertex if and only if $co_{p}=1$.

  (ii)\;The support of the $F$-polynomial of any cluster monomial is saturated.
\end{Conjecture}

In \cite{F}, this conjecture was proved to be true when the initial exchange matrix $B_{t_{0}}$ is acyclic and skew-symmetric. Here, as a direct corollary of Theorem \ref{properties for cluster variable case}, we give a positive answer to the conjecture.

\begin{Corollary}\label{answer to fei}
  Let $\A$ be a TSSS cluster algebra with principal coefficients and $t\in \T_{n}$. Then,

  (i)\;A point $p$ in the Newton polytope $N(X_{t}^{\alpha})$ associated to any cluster monomial $X_{t}^{\alpha}$ is a vertex if and only if $co_{p}(N_{g(X_{t}^{\alpha})})=1$ for any $\alpha\in\N^{n}$.

  (ii)\;The support of the $F$-polynomial $F_{X_{t}^{\alpha}}$ of any cluster monomial $X_{t}^{\alpha}$ is saturated.
\end{Corollary}
\begin{Proof}
  (i)\; According to the definition of Minkowski sum, a point $q\in N\oplus N\'$ is a vertex if and only if there is unique vertices $p\in N$ and $p\'\in N\'$ such that $q=p+p\'$. Then by Theorem \ref{properties for cluster variable case} (ii), $p\in V(N(X_{t}^{\alpha}))$ is equivalent to $co_{p}=1$ since $N(X_{t}^{\alpha})=\bigoplus\limits_{i=1}^{n}(\bigoplus\limits_{j=1}^{\alpha_{i}}N_{i;t})$.

  (ii)\; For any point $q\in N\oplus N\'$, there are points $p\in N$ and $p\'\in N\'$ such that $q=p+p\'$. If the supports of $N$ and $N\'$ are both saturated, then $co_{p}(N),co_{p\'}(N\')\neq 0$. So (ii) is induced by Theorem \ref{properties for cluster variable case} (i) and the universally positivity of any cluster variable since $N(X_{t}^{\alpha})=\bigoplus\limits_{i=1}^{n}(\bigoplus\limits_{j=1}^{\alpha_{i}}N_{i;t})$.
\end{Proof}

However, Theorem \ref{properties for cluster variable case} does not hold for general $N_{h}$ with arbitrary $h\in\Z^{n}$. Next we present a counterexample when $h$ is not a $g$-vector of some cluster monomial.

\begin{Example}
  In Example \ref{eg in rank 2}, $\rho_{h}$ is not a cluster monomial since it can never be expressed as a polynomial in certain cluster. $N_{h}$ is a triangle given in Example \ref{eg in rank 2} with the lengths of two right-angle sides equaling 5 and 8 respectively. We can see that   $co_{(4,6)}(N_{h})=0$ and $co_{(1,1)}(N_{h})=1$. However, both $(4,6)$ and $(1,1)$ lie in the interior of $N_{h}$, which indicates neither (i) nor (ii) of Theorem \ref{properties for cluster variable case} holds for $N_{h}$.
\end{Example}

\begin{Proposition}
  Let $\A$ be a TSSS cluster algebra having principal coefficients, $l,r\in[1,n],t\in\T_{n}$ and $S$ be an $r$-lattice-dimensional face of $N_{l;t}$. Then there is a cluster algebra $\A\'$ with principal coefficients of rank $r$, a Newton polytope $S\'$ corresponding to some coefficient-free cluster monomial in $\A\'$ and a non-negative projection $\tau$ from $S\'$ to $S$.
\end{Proposition}

\begin{Proof}
  Because of Theorem \ref{properties in case tsss} (v), there is a non-negative projection from $N_{h\'}|_{\A\'}$ to $S$ for some $h\'\in\Z^{r}$. So we only need to show that $N_{h\'}|_{\A\'}$ equals the Newton polytope associated to some coefficient free cluster monomial in $\A\'$.

  We prove the proposition by induction on the length of the path connecting $t$ and $t_{0}$. It is trivial when the length is $0$, i.e., $t_{0}=t$. Assume it is true for $t\'\in\T_{n}$. We claim that it is true for $t_{0}\in\T_{n}$ connected to $t\'$ by an edge labeled $k\in[1,n]$.

  According to the mutation formula (\ref{equation: mutation of x}) and (\ref{equation: mutation of y}), any face $S$ of $N_{l;t}$ correlates to a face $S^{\prime\prime}$ of $N^{t\'}_{l;t}$, which is a projection of the Newton polytope $N$ of some coefficient free cluster monomial in $\A\'$ due to the inductive assumption. Denote this projection by $\tau\'$. Since $S$ correlates to $S^{\prime\prime}$ with respect to the mutation in direction $k$, it is either a projection of $N$ when $\tilde{\tau}\'(e_{i})\neq e_{k}$ for any $i$ or of $\mu_{i}(N)$ when there is $i$ such that $\tilde{\tau}\'(e_{i})=e_{k}$. In both case, $S$ is a projection of the Newton polytope associated to some coefficient free cluster monomial.
\end{Proof}

\vspace{4mm}

\section{Relations among $F$-polynomials, $d$-vectors, $g$-vectors and cluster variables}

In this section, we will use the results in the last two sections to investigate some concrete relations among $F$-polynomials, $d$-vectors, $g$-vectors and cluster variables in a totally sign-skew-symmetric cluster algebra.

\subsection{From $F$-polynomials to $g$-vectors, cluster variables and $d$-vectors with positivity}\quad

\begin{Lemma}\label{non-maximal and non-minimal}
  Let $\A$ be a cluster algebra having principal coefficients. Denote by $p$ a vertex of $N_{h}$, and $p\'$ another vertex of the $k$-section at $p$ for some $k\in[1,n]$. If $l(\overline{pp\'})>[-deg_{x_{k}}(p)]_{+}$, then $-d_{k}(\rho_{h})<deg_{x_{k}}(p)<deg_{x_{k}}(\rho_{h})$, where $d_{k}(\rho_{h})$ denotes the $k$-th element of the denominator vector of $\rho_{h}$.
\end{Lemma}
\begin{Proof}
  There is nothing to say when the dimension is 0 or 1 since the assumption never holds. So we assume the dimension is at least 2.

  Let $K$ be the maximal section of $N_{h}$ at $p$ satisfying that $deg_{x_{k}}(q)=deg_{x_{k}}(p)$ for any point $q\in K$. Obviously, the segment connecting $p$ and $p\'$ lies in $K$. If $K$ is a face of $N_{h}$, then by Theorem \ref{properties in case tsss} (v), $K$ equals a polytope $N_{h\'}$ for some $h\'\in\Z^{r}$ with $r\leqslant n$ up to a translation. So according to the definition of $N_{h\'}$, we always have $l(\overline{pp\'})=[-deg_{x_{k}}(p)]_{+}$, which contradicts our assumption. Hence $K$ can not be a face of $N_{h}$.

  On the other hand, by Theorem \ref{expression of a cluster variable}, the $x_{k}$-degree of a point in $N_{h}$ linearly depends on its coordinates, which leads to that the convex hull of the set of points with maximal or minimal $x_{k}$-degree is a face of $N_{h}$.

  Therefore, combining the above facts we get $-d_{k}(\rho_{h})<deg_{x_{k}}(p)<deg_{x_{k}}(\rho_{h})$.
\end{Proof}

According to the definitions of $N_{h}$ and $\rho_{h}$, for any $h\in H$ and $k\in [1,n]$, we have the following $x_{k}$-degree decomposition:
\[\rho_{h}=\sum\limits_{i=-d_{k}(\rho_{h})}^{deg_{x_{k}}(\rho_{h})}x_{k}^{i}\rho_{h}^{(i)} =\sum\limits_{i=-d_{k}(\rho_{h})}^{deg_{x_{k}}(\rho_{h})}x_{k}^{i}M_{k}^{[-i]_{+}}\rho_{h}^{(i;k)},\]
where $\rho_{h}^{(i)}=M_{k}^{[-i]_{+}}\rho_{h}^{(i;k)}$ and $\rho_{h}^{(i;k)}$ is a Laurent polynomial in $\N\P[x_{1}^{\pm},\cdots,x_{k-1}^{\pm},x_{k+1}^{\pm},\cdots,x_{n}^{\pm}]$.

\begin{Lemma}\label{P(0)}
  Let $\A$ be a cluster algebra with principal coefficients and $h\in H$. Then for any $k\in [1,n]$, we have

  (i)\;$M_{k}\nmid \rho_{h}^{(-d_{k}(\rho_{h});k)}$;\;\;\;
  (ii)\;$M_{k}\nmid \rho_{h}^{(deg_{x_{k}}(\rho_{h});k)}$. \;\;
\end{Lemma}
\begin{Proof}
  In fact this is a direct corollary of Lemma \ref{non-maximal and non-minimal}.

  Note that the $x_{k}$-degree of a summand Laurent monomial of $\rho_{h}$ linearly depends on the coordinates of its corresponding point in $N_{h}$. Hence there must be a vertex $p$ in $N_{h}$ such that $\hat{Y}^{p}X^h$ has maximal or minimal $x_{k}$-degree. Denote by $p\'$ the other vertex of the $k$-section at $p$. Then by Lemma \ref{non-maximal and non-minimal}, $l(\overline{pp\'})\leqslant[-deg_{x_{k}}(p)]_{+}$. Hence there is no summand $P$ of $\rho_{h}$ such that $p$ is a summand of $P$ and $M_{k}^{[-deg_{x_{k}}(p)]_{+}+1}\mid P$.
  In particular,
  $$M_{k}^{[d_{k}(\rho_{h})]_{+}+1}\nmid x_{k}^{-d_{k}(\rho_{h})}M_{k}^{[d_{k}(\rho_{h})]_{+}}\rho_{h}^{(-d_{k}(\rho_{h});k)}$$
  and
  $$M_{k}^{[-deg_{x_{k}}(\rho_{h})]_{+}+1}\nmid x_{k}^{deg_{x_{k}}(\rho_{h})}M_{k}^{[-deg_{x_{k}}(\rho_{h})]_{+}}\rho_{h}^{(-deg_{x_{k}}(\rho_{h});k)},$$ i.e., $M_{k}\nmid \rho_{h}^{(-d_{k}(\rho_{h});k)}$ and $M_{k}\nmid \rho_{h}^{(deg_{x_{k}}(\rho_{h});k)}$.
\end{Proof}

\begin{Example}
  (i)\;   When $-d_{k}(\rho_{h})<s<deg_{x_{k}}(\rho_{h})$, there can be $M_{k}\mid \rho_{h}^{(s;k)}$. Let the initial seed be as that in Example \ref{example}, then in the seed $\mu_{3}\circ\mu_{1}\circ\mu_{2}((X,\tilde{B}))$ there is a cluster variable
  \[\rho_{(0,0,-1)}=\frac{y_{3}(y_{2}^{3}y_{1}x_{2}^{3}x_{3}^{2}+(y_{2}x_{3}^{6}+x_{1}^{3})^{3})^{2}+x_{1}^{2}x_{2}^{6}}{x_{1}^{2}x_{2}^{6}x_{3}}.\]
  Hence $-d_{1}(\rho_{(0,0,-1)})=-2$ and $deg_{x_{1}}(\rho_{(0,0,-1)})=16$. Choose $s$ to be 1, the summand with all $x_{1}$-degree 1 terms is $x_{1}\cdot6(y_{1}x_{2}^{3}+x_{3}^{16})y_{2}^{5}y_{3}x_{2}^{-6}x_{3}^{14}$. So $\rho_{(0,0,-1)}^{(1;1)}=6(y_{1}x_{2}^{3}+x_{3}^{16})y_{2}^{5}y_{3}x_{2}^{-6}x_{3}^{14}$. Since $M_{1}=y_{1}x_{2}^{3}+x_{3}^{16}$, we have $M_{1}\mid\rho_{(0,0,-1)}^{(1;1)}$ in this case.

  (ii)\;  Lemma \ref{P(0)} does not hold for any semifield $\P$. One counterexample is given in \cite{FK}. Let $\P=\{1\}$ and let the initial seed be $(X,B)$, where
    \[B=\left(\begin{array}{ccc}
                0 & 2 & -1 \\
                -2 & 0 & 1 \\
                1 & -1 & 0
              \end{array}\right).\]
  Then in seed $\mu_{1}\circ\mu_{2}\circ\mu_{3}((X,B))$ there is a cluster variable
  $\rho_{(0,0,-1)}|_{\P}=\frac{x_{1}^{2}+2x_{1}x_{2}+x_{2}^{2}+x_{3}}{x_{1}x_{2}x_{3}}.$

  We can see that $-d_{3}(\rho_{(0,0,-1)})=-1$,$M_{3}=x_{1}+x_{2}$ and $\rho_{(0,0,-1)}^{(-1;3)}=M_{3}$. So $M_{3}\mid \rho_{(0,0,-1)}^{(-d_{3}(\rho_{(0,0,-1)});3)}$.
\end{Example}

For any $i\in[1,n]$ and $t\in\T_{n}$, define a map
\[\phi_{i}^{t}:\; \Z\P[x_{1;t},\cdots,x_{n;t}]\longrightarrow \Z\P[x_{1;t},\cdots,x_{i-1;t},x_{i+1;t},\cdots,x_{n;t}]\]
such that $\phi_{i}^{t}(P)=P\mid_{x_{i;t}\rightarrow 0}$.
\begin{Theorem}\label{d-vector and lengths}
  Let $\A$ be a TSSS cluster algebra having principal coefficients and $h\in H$. Then $[d_{k}(\rho_{h})]_{+}$ equals the maximal length of edges of $N_{h}$ parallel to the $k$-th coordinate axis for any $k\in [1,n]$.
\end{Theorem}
\begin{Proof}
  Denote by $l_k$ the maximal length of edges of $N_{h}$ parallel to the $k$-th coordinate axis. Because $x_{k}$-degree of a summand Laurent monomial of $\rho_{h}$ linearly depends on the coordinates of its corresponding point in $N_{h}$, so there must be a vertex $p$ in $N_{h}$ such that $deg_{x_k}(p)=-d_{k}(\rho_h)$. Let $K$ be the maximal section of $N_{h}$ at $p$ satisfying that $deg_{x_{k}}(q)=deg_{x_{k}}(p)$ for any point $q\in K$.

  If $dim(N_h)\leqslant2$, due to Proposition \ref{polytope basis equals greedy basis}, $d_{1}(\rho_{h})=-h_{1}$ and $d_{2}(\rho_{h})=-h_{2}+c[-h_{1}]_{+}$. On the other hand, according to the construction of $N_{h}$ and Remark \ref{shape for rank 2}, the maximal lengths of edges of $N_{h}$ parallel to $e_{1}$ and $e_{2}$ are $[-h_{1}]_{+}$ and $[-h_{2}+c[-h_{1}]_{+}]_{+}$ respectively. So the theorem holds in this case.

  If $dim(N_h)>2$, we take induction on the order of polytopes to prove this result. Assume it holds for all proper sub-polytopes of $N_h$. According to the construction of $N_h$, for any $i\in[1,n]$, due to Construction \ref{construction} (b), we have
  \[x_{i}\rho_{h}=\sum\limits_{N_{\alpha}[w]\in\bigcup\limits_{l} U_{h}^{l}}Y^{w}\rho_{\alpha}\]
  and hence
  \[N_{h}=\sum\limits_{N_{\alpha_{j}}[w_{j}]\in\bigcup\limits_{l}U_{h}^{l}}N_{\alpha_{j}}[w_{j}].\]
  Then because of the inductive assumption, when $i\neq k$, the maximal length of edges of $N_{\alpha_{j}}[w_{j}]$ parallel to the $k$-th coordinate axis is not larger than $[d_{k}(\rho_{h})]_{+}$. So $l_k\leqslant[d_{k}(\rho_{h})]_{+}$.

  Since $\hat{Y}^pX^h$ has the minimal $x_k$-degree among monomial summands of $\rho_h$, $K$ must be a face of $N_h$ by the definition of faces in (\ref{face}). Moreover, due to the definition of $K$, it can be seen that $K=K\'\oplus K^{\prime\prime}$, where $K\'$ is the ${k}$-section of $K$ at the minimal point in $K$ while $K^{\prime\prime}$ is the $[1,n]\setminus\{k\}$-section of $K$ at the same point. Therefore because $p\in V(N_{h})$, the ${k}$-section of $K$ at $p$ is an edge of $K$, hence an edge of $N_h$. If the length of the ${k}$-section of $K$ at $p$ is less than $[-deg_{x_k}(p)]_+$, then $\rho_{h}$ is not a formal Laurent polynomial in $X_{t_{k}}$ according to the calculation in Subsection 3.1, where $t_{k}$ is the vertex connected to $t_{0}$ by an edge labeled $k$. Hence by Lemma \ref{non-maximal and non-minimal}, the length of the ${k}$-section of $K$ at $p$ equals $[-deg_{x_k}(p)]_+=[d_{k}(\rho_h)]_+$. So $l_{k}\geqslant[d_{k}(\rho_h)]_+$.

  Therefore, $l_k=[d_{k}(\rho_{h})]_{+}$.
\end{Proof}

The above relation between $d$-vectors and polytopes induces the positivity of $d$-vectors associated to non-initial cluster variables which was first come up as a conjecture in \cite{FZ4} and then proved in \cite{CL} for skew-symmetrizable case.
\begin{Theorem}\label{positivity of d-vectors}
  Let $\A$ be a TSSS cluster algebra and $x_{l;t}$ be a non-initial cluster variable in $\A$ with $l\in[1,n]$, $t\in\T_{n}$. Then $d_{l;t}\in\N^{n}$. More precisely, for any $k\in[1,n]$,
  $$d_{k}(x_{l;t})=\widetilde{deg}_{k}(P_{l;t})=\widetilde{deg}_{k}(\phi_{k}(P_{l;t})),$$
  where $P_{l;t}$ is the enumerator of the Laurent expression of $x_{l;t}$ in $X$.
\end{Theorem}
\begin{Proof}
  Since this result is independent of the choice of the semifield $\P$, we can assume $\A$ has principal coefficients. If $d_{l;t}\not\in\N^{n}$, then there is $i\in[1,n]$ such that $d_{i}(x_{l;t})<0$.

  Thus, by Theorem \ref{d-vector and lengths}, $N_{g_{l;t}}$ lies in the hyperplane $z_{i}=0$ and $deg_{x_{i}}(\hat{Y}^{p}X^{g_{l;t}})\geqslant -d_{i}(x_{l;t})$ for any point $p\in N_{g_{l;t}}$. So according to the construction of $N_{g_{l;t}}$ and $N_{g_{l;t}+d_{i}(x_{l;t})e_{i}}$, we have \[N_{g_{l;t}}=\gamma_{i;0}(N_{\pi_{i}(g_{l;t})})=\gamma_{i;0}(N_{\pi_{i}(g_{l;t}+d_{i}(x_{l;t})e_{i})})=N_{g_{l;t}+d_{i}(x_{l;t})e_{i}}.\]
  Hence $x_{l;t}=\rho_{g_{l;t}}=x_{i}^{-d_{i}(x_{l;t})}\rho_{g_{l;t}+d_{i}(x_{l;t})e_{i}}$. So
  \begin{equation}\label{equation: in the proof of positivity of d-vectors}
    x_{l;t}=L^{t}(x_{i}^{-d_{i}(x_{l;t})}\rho_{g_{l;t}+d_{i}(x_{l;t})e_{i}})=L^{t}(x_{i}^{-d_{i}(x_{l;t})})L^{t}(\rho_{g_{l;t}+d_{i}(x_{l;t})e_{i}}) =(\rho^{t}_{g_{i;t_{0}}^{t}})^{-d_{i}(x_{l;t})}\rho^{t}_{h},
  \end{equation}
  where $h=(g_{l;t}+d_{i}(x_{l;t})e_{i})^{t}$.

  According to Theorem \ref{properties in case tsss}, $\rho^{t}_{g_{i;t_{0}}^{t}},\rho^{t}_{h}\in\N Trop(Y_{t})[X_{t}^{\pm}]$. So (\ref{equation: in the proof of positivity of d-vectors}) induces that both $\rho^{t}_{g_{i;t_{0}}^{t}}$ and $\rho^{t}_{h}$ are Laurent monomials, which means $g_{i;t_{0}}^{t},h\in\N^{n}$ and thus $x_{l;t}=X_{t}^{-d_{i}(x_{l;t})g_{i;t_{0}}^{t}+h}$. It holds only when $g_{i;t_{0}}^t=e_{l}$, $d_{i}(x_{l;t})=-1$ and $h=0$, which contradicts the non-initial assumption.

  Therefore $d_{l;t}\in\N^{n}$.

  Thus by Theorem \ref{d-vector and lengths} and its proof we have $d_{k}^{t\'}(x_{l;t})=[d_{k}^{t\'}(x_{l;t})]_{+}=\widetilde{deg}_{k}^{t\'}(\phi_{k}(P_{l;t}))$. Moreover, by Lemma \ref{general deg > d} and the definition of general degree we have
  \[d_{k}^{t\'}(x_{l;t})\leqslant\widetilde{deg}_{k}^{t\'}(P_{l;t})\leqslant\widetilde{deg}_{k}^{t\'}(\phi_{k}(P_{l;t})),\]
  so they are all equal.
\end{Proof}

According to Theorem \ref{positivity of d-vectors}, we get that $d_{l;t}\in\N^{n}$ when $x_{l;t}$ is a non-initial cluster variable. Therefore, $[d_{k}(x_{l;t})]_{+}=d_{k}(x_{l;t})$ in this case.  Hence Theorem \ref{d-vector and lengths} provides a geometric realization of $d$-vectors associated to non-initial cluster variables.

For a cluster algebra $\A$ having principal coefficients, define a map
\[\psi:\; \Z\P[x_{1}^{\pm 1},\cdots,x_{n}^{\pm 1}]\longrightarrow\Z\P[x_{1},\cdots,x_{n}]\]
such that $\psi(x)=P$ when $x=PX^{\alpha}$, where $P\in\Z\P[x_{1},\cdots,x_{n}]$, $\alpha\in\Z^{n}$ and $P$ is coprime with $x_{i},i\in[1,n]$. Theorem \ref{expression of a cluster variable} shows that
\[x_{l;t}=\frac{F_{l;t}|_{\F}(\hat{y}_{1},\hat{y}_{2},\cdots,\hat{y}_{n})}{F_{l;t}|_{\P}(y_{1},y_{2},\cdots,y_{n})}X^{g_{l;t}}
=F_{l;t}|_{\F}(\hat{y}_{1},\hat{y}_{2},\cdots,\hat{y}_{n})X^{g_{l;t}}.\]
Due to Theorem \ref{positivity of d-vectors}, for a non-initial cluster variable $x_{l;t}$,
\[P_{l;t}=\psi(F_{l;t}|_{\F}(\hat{y}_{1},\hat{y}_{2},\cdots,\hat{y}_{n})).\]
We obtain a map from theorem \ref{positivity of d-vectors}
\begin{equation*}
  \varphi:\; \{\text{non-initial $F$-polynomials of }\A\}\longrightarrow\{\text{non-initial $d$-vectors of }\A\}
\end{equation*}
such that \[\varphi(F_{l;t})=(\widetilde{deg}_{1}\circ\phi_{1}\circ\psi(F_{l;t}|_{\F}(\hat{y}_{1},\hat{y}_{2},\cdots,\hat{y}_{n})),\cdots,\widetilde{deg}_{n}\circ\phi_{n}\circ\psi(F_{l;t}|_{\F}(\hat{y}_{1},\hat{y}_{2},\cdots,\hat{y}_{n}))),\] which maps a $F$-polynomial associated to a non-initial cluster variable to the $d$-vector associated to the same cluster variable.

Again by Theorem \ref{expression of a cluster variable}, we can see that $X^{g_{l;t}}=\frac{p_{0}}{X^{d_{l;t}}}$ for any $l\in[1,n],t\in\T_{n}$, where $p_{0}$ is the unique monomial summand of $P_{l;t}$ with coefficient $1$. So we can also define a map from a $F$-polynomial associated to a non-initial cluster variable to the $g$-vector associated to the same cluster variable:
\begin{equation*}
  \theta_{1}:\; \{\text{non-initial $F$-polynomials of }\A\}\longrightarrow\{\text{non-initial $g$-vectors of }\A\}
\end{equation*}
such that
\[\theta_{1}(F_{l;t})=(deg_{x_{1}}(\frac{\psi(F_{l;t}|_{\F}(\hat{y}_{1},\hat{y}_{2},\cdots,\hat{y}_{n}))\mid_{y_{1}=\cdots=y_{n}=0}}{X^{\varphi(F_{l;t})}}),\cdots,deg_{x_{n}}(\frac{\psi(F_{l;t}|_{\F}(\hat{y}_{1},\hat{y}_{2},\cdots,\hat{y}_{n}))\mid_{y_{1}=\cdots=y_{n}=0}}{X^{\varphi(F_{l;t})}}))\]

In conclusion, we have the following theorem.

\begin{Theorem}\label{maps from $F$-polynomials}
  Let $\A$ be a TSSS cluster algebra with principal coefficients. Then, for any $l=1,\cdots,n$ and $t\in \T_n$,

  (i)\; there is a surjective map
  \begin{equation*}
    \varphi:\; \{\text{non-initial $F$-polynomials of }\A\}\longrightarrow\{\text{non-initial $d$-vectors of }\A\}
  \end{equation*}
  such that \[\varphi(F_{l;t})=(\widetilde{deg}_{1}\circ\phi_{1}\circ\psi(F_{l;t}|_{\F}(\hat{y}_{1},\hat{y}_{2},\cdots,\hat{y}_{n})),\cdots,\widetilde{deg}_{n}\circ\phi_{n}\circ\psi(F_{l;t}|_{\F}(\hat{y}_{1},\hat{y}_{2},\cdots,\hat{y}_{n})));\]

  (ii)\; there is a bijective map
  \begin{equation*}
    \theta_{1}:\; \{\text{non-initial $F$-polynomials of }\A\}\longrightarrow\{\text{non-initial $g$-vectors of }\A\}
  \end{equation*}
  such that
  \[\theta_{1}(F_{l;t})=(deg_{x_{1}}(\frac{\psi(F_{l;t}|_{\F}(\hat{y}_{1},\hat{y}_{2},\cdots,\hat{y}_{n}))\mid_{y_{1}=\cdots=y_{n}=0}}{X^{\varphi(F_{l;t})}}),\cdots,deg_{x_{n}}(\frac{\psi(F_{l;t}|_{\F}(\hat{y}_{1},\hat{y}_{2},\cdots,\hat{y}_{n}))\mid_{y_{1}=\cdots=y_{n}=0}}{X^{\varphi(F_{l;t})}})).\]

  (iii)\; there is a bijective map
  \begin{equation*}
    \chi_{1}:\; \{\text{non-initial $F$-polynomials of }\A\}\longrightarrow\{\text{non-initial cluster variables of }\A\}
  \end{equation*}
  such that
  \[x_{l;t}=\chi_{1}(F_{l;t})=\frac{\psi(F_{l;t}|_{\F}(\hat{y}_{1},\hat{y}_{2},\cdots,\hat{y}_{n}))}{X^{\varphi(F_{l;t})}}.\]
\end{Theorem}
\begin{Proof}
  (i)\; It is directly induced by Theorem \ref{positivity of d-vectors}.

  (ii)\; The surjectivity naturally holds by the definition of $\theta_{1}$ and that of $g$-vectors.

  On the other hand, if there are $l,l\'\in [1,n]$ and $t,t\'\in\T_{n}$ satisfying $\theta_{1}(F_{l;t})=\theta_{1}(F_{l\';t\'})$, then $F_{l;t}=\rho_{\theta_{1}(F_{l;t})}|_{x_{i}\rightarrow 1,\forall i\in[1,n]}=\rho_{\theta_{1}(F_{l\';t\'})}|_{x_{i}\rightarrow 1,\forall i\in[1,n]}=F_{l\';t\'}$. Hence $\theta_{1}$ is also injective.

  (iii)\; Following from (i), we have $x_{l;t}=\frac{P_{l;t}}{X^{d_{l;t}}}=\frac{\psi(F_{l;t}|_{\F}(\hat{y}_{1},\hat{y}_{2},\cdots,\hat{y}_{n}))}{X^{\varphi(F_{l;t})}}$, which shows the surjectivity.

   On the other hand, since $F_{l;t}=P_{l;t}|_{x_{i}\rightarrow 1,\forall i\in[1,n]}$, the injectivity follows.
\end{Proof}

The map $\chi_{1}$ directly leads to the following result.

\begin{Corollary}\label{$F$-polynomial uniquely determines cluater variable}
  Let $\A$ be a TSSS cluster algebra having principal coefficients with two non-initial cluster variables $x_{l;t}$, $x_{l\';t\'}$ and $F_{l;t}$, $F_{l\';t\'}$ be the $F$-polynomials associated to $x_{l;t}$, $x_{l\';t\'}$ respectively. If $F_{l;t}=F_{l\';t\'}$, then $x_{l;t}=x_{l\';t\'}$.
\end{Corollary}

\begin{Proposition}\label{cluster monomial}
  \cite{CL}Let $\A$ be a skew-symmetrizable cluster algebra having principal coefficients at some vertex and $t,t\'\in\T_{n}$. If $\prod\limits_{i=1}^{n}x_{i;t}^{a_{i}}=\prod\limits_{i=1}^{n}x_{i;t\'}^{a_{i}\'}$ and they are cluster monomials in $X_{t}$ and $X_{t\'}$ respectively, then for every $a_{i}\neq 0$, there is $j\in[1,n]$ such that $x_{i;t}=x_{j;t\'}$ and $a_{i}=a_{j}\'$.
\end{Proposition}

The following Corollary is a generalization of Proposition \ref{cluster monomial}.
\begin{Corollary}\label{cluster laurent monomial}
  Let $\A$ be a TSSS cluster algebra having principal coefficients, $t,t\'\in\T_{n}$ and $\alpha,\beta\in\Z^{n}$ are non-zero. If $X_{t}^{\alpha}=X_{t\'}^{\beta}$, then this is a permutation $\sigma$ of $[1,n]$ such that $x_{i;t}=x_{\sigma(i);t\'}$ and $\alpha_{i}=\beta_{\sigma(i)}$ for any $\alpha_{i}\neq 0$.
\end{Corollary}
\begin{Proof}
  \[X_{t\'}^{\beta}=X_{t}^{\alpha}=\prod\limits_{i=1}^{n}(\frac{P_{i;t}^{t\'}}{X_{t\'}^{d_{i;t}^{t\'}}})^{\alpha_{i}}\]
  is a laurent monomial in $X_{t\'}$. So $\prod\limits_{i=1}^{n}(P_{i;t}^{t\'})^{\alpha_{i}}$ is a laurent monomial. Let $I_{1}=\{i\in[1,n]\mid x_{i;t}\in X_{t\'}\}$, $I_{2}=[1,n]\setminus I_{1}$. By Corollary \ref{$F$-polynomial uniquely determines cluater variable}, $F_{i;t}^{t\'}\neq F_{j;t}^{t\'}$ , hence $P_{i;t}\neq P_{j;t}$ if $i\neq j$ and $i\in I_{2}$. So $\alpha_{i}=0$ for $i\in I_{2}$. Then there is a permutation $\sigma$ of $[1,n]$ such that $x_{i;t}=x_{\sigma(i);t\'}$ for $i\in I_{1}$ and
  \[\prod\limits_{i\in I_{1}}x_{\sigma(i);t\'}^{\alpha_{i}}=\prod\limits_{i=1}^{n}x_{i;t\'}^{\beta_{\sigma(i)}}.\]
  Therefore $\alpha_{i}=\beta_{\sigma(i)}$ for $i\in I_{1}$ and $\alpha_{i}=0$ for $i\in I_{2}$.
\end{Proof}

\begin{Theorem}\label{CL:compatible set}
  \cite{CL} \;For any skew-symmetrizable cluster algebra $\A$ and any collection $U$ of cluster variables in $\A$, if each pair in $U$ is contained in some cluster of $\A$, then there is a cluster of $\A$ containing $U$ as a subset.
\end{Theorem}

\begin{Lemma}\label{F=1}
  For any $l\in [1,n]$ and $t\in \T_{n}$, $x_{l;t}$ is an initial cluster variable if and only if $F_{l;t}=1$.
\end{Lemma}
\begin{proof}
  The necessity is directly from the definition of $F$-polynomials.

  By Theorem \ref{expression of a cluster variable}, we have
  \begin{equation*}
    x_{l;t}=\frac{F_{l;t}|_{\F}(\hat{y}_{1},\hat{y}_{2},\cdots,\hat{y}_{n})}{F_{l;t}|_{\P}(y_{1},y_{2},\cdots,y_{n})}X^{g_{l;t}}.
  \end{equation*}
  So when $F_{l;t}=1$, $x_{l;t}$ is  a Laurent monomial of $X_{t_{0}}$ and of $X_{t}$ respectively. According to Corollary \ref{cluster laurent monomial}, $x_{l;t}=x_{j}$ for some $j$, which shows the sufficiency.
\end{proof}

Then by Lemma \ref{F=1}, Theorem \ref{CL:compatible set} and Corollary \ref{$F$-polynomial uniquely determines cluater variable}, we have the following corollary:

\begin{Corollary}
  Let $\A$ be a skew-symmetrizable cluster algebra with $X_{t}$, $X_{t\'}$ two clusters. If $\{F_{i;t}\}_{i\in [1,n]}=\{F_{i;t'}\}_{i\in [1,n]}$, then $X_{t}=X_{t'}$.
\end{Corollary}
\begin{Proof}
  For convenience we order two sets of $F$-polynomials such that $F_{i;t}=F_{i;t\'}$ for any $i\in[1,n]$. If $F_{i;t}\neq 1$, by the last Lemma, $x_{i;t}$ is non-initial and so following Corollary \ref{$F$-polynomial uniquely determines cluater variable} we have $x_{i;t}=x_{i;t\'}$. If $F_{i;t}=1$, then $x_{i;t}$ is an initial cluster variable. Without lose of generality, we can order $F$-polynomials such that $F_{i;t}=1$ when $i\in[1,k]$ and $F_{i;t}\neq 1$ otherwise. Then it is enough to prove that given a set of cluster variables $x_{k+1;t},\cdots,x_{n;t}$, there is at most one $k$-set of initial cluster variables $x_{i_{1}},\cdots,x_{i_{k}}$ such that $(x_{i_{1}},\cdots,x_{i_{k}},x_{k+1;t},\cdots,x_{n;t})$ is a cluster.

  Assume there are two different $k$-set satisfying above condition, then by Theorem \ref{CL:compatible set} there is a cluster containing all of them. However there are totally at least $n+1$ variables. The contradiction completes the proof.
\end{Proof}

\subsection{From $g$-vectors to $F$-polynomials, cluster variables and $d$-vectors}\quad

The fact that a $g$-vector uniquely determines its corresponding cluster variable is already proved in \cite{GHKK} for skew-symmetrizable case, here in this subsection we would like to express the maps from a $g$-vector associated to a cluster variable to the $F$-polynomial and the $d$-vector associated to the same cluster variable respectively.

\begin{Theorem}\label{maps from $g$-vectors}
  Let $\A$ be a TSSS cluster algebra with principal coefficients. Then, for any $l\in [1,n]$ and $t\in \T_n$,

  (i)\; there is a surjective map
  \begin{equation*}
    \theta_{2}:\; \{\text{$g$-vectors of }\A\}\longrightarrow\{\text{$F$-polynomials of }\A\}\;\;\; \text{via}\;\;\;\theta_{2}(g_{l;t})=\rho_{g_{l;t}}\mid_{x_{i}\rightarrow 1,\forall i}.
  \end{equation*}
  which is bijective when restricted to non-initial subsets. In this case, we have $\theta_1=\theta_2^{-1}$ for the map $\theta_1$ in Theorem \ref{maps from $F$-polynomials};

  (ii)\; there is a surjective map
  \begin{equation*}
    \eta:\; \{\text{$g$-vectors of }\A\}\longrightarrow\{\text{$d$-vectors of }\A\}\;\;\; \text{via}\;\;\;\eta(g_{l;t})=\text{denominator vector of } \rho_{g_{l;t}}.
  \end{equation*}

  (iii)\; there is a bijective map
  \begin{equation*}
    \chi_{2}:\; \{\text{$g$-vectors of }\A\}\longrightarrow\{\text{cluster variables of }\A\}
  \end{equation*}
  via
  \[x_{l;t}=\chi_{2}(g_{l;t})=\rho_{g_{l;t}}.\]
\end{Theorem}
\begin{Proof}
  According to Theorem \ref{from general to cluster variables}, we have $x_{l;t}=\rho_{g_{l;t}}$ and $deg(\rho_{g_{l;t}})=g_{l;t}$, which lead to (iii). Moreover, (i) and (ii) follow from (iii), due to Theorem \ref{maps from $F$-polynomials} (ii) and the definitions of $F$-polynomials and $d$-vectors.
\end{Proof}

\begin{Remark}
  When $\A$ is a cluster algebra over a semifield $\P$, by Theorem \ref{maps from $F$-polynomials} (iii) and Theorem \ref{maps from $g$-vectors} (iii), for a cluster variable $x_{l;t}$ in $\A$, we have
  \begin{equation}\label{cluster variables and g-vectors in general semifield}
    x_{l;t}=\frac{\psi(F_{l;t}|_{\F}(\hat{y}_{1},\hat{y}_{2},\cdots,\hat{y}_{n}))}{F_{l;t}|_{\P}(y_{1},y_{2},\cdots,y_{n})X^{\varphi(F_{l;t})}}=\frac{\rho_{g_{l;t}}}{\theta_{2}(g_{l;t})|_{\P}(y_{1},y_{2},\cdots,y_{n})},
  \end{equation}
  where $F_{l;t}$ is the $F$-polynomial associated to $x_{l;t}$ in the corresponding principal coefficients cluster algebra $\A_{prin}$  of $\mathcal A$.

  Thanks to Siyang Liu for pointing out that because of Corollary \ref{TSSS positivity}, the proof of Theorem 7.2 in \cite{N} holds for TSSS cluster algebras. Hence (\ref{cluster variables and g-vectors in general semifield}) induces a bijection between cluster variables and $g$-vectors, which confirms a conjecture in \cite{FZ} claiming that the exchange graph of a TSSS cluster algebra only depends on the exchange matrix $B$.
\end{Remark}

Thus we obtain the relations among cluster variables, $g$-vectors, $F$-polynomials and $d$-vectors as the following diagram. Note that the maps from $\{F-polynomials\}$ are restricted in the subset consisting of non-initial $F$-polynomials since all initial $F$-polynomials equal to $1$.

\begin{figure}[H]
  \centering
  \includegraphics[width=80mm]{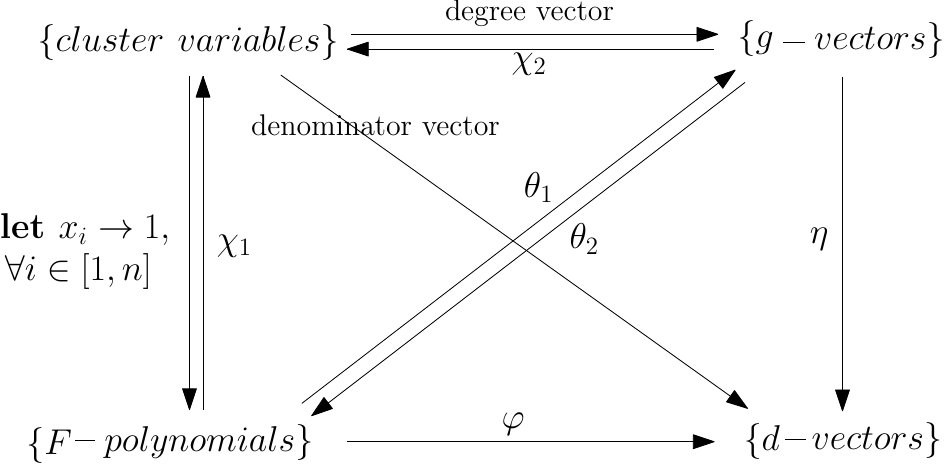}
  \caption{Relation Diagram}\label{fig7}
\end{figure}

 A natural question is that whether we can construct a map from the set of $d$-vectors to one of the other three sets in Figure \ref{fig7}.
\begin{Problem}
  In a cluster algebra $\A$, is an element in $\mathcal{P}$ uniquely determined by its denominator vector in $\Z^{n}$? In particular, is the $g$-vector $g_{l,t}$ associated to a cluster variable $x_{l;t}$ determined uniquely by its denominator vector $d_{l;t}$?
\end{Problem}
The answer is positive in some special cases. For example, when $\A$ is of rank 2, according to Proposition \ref{polytope basis equals greedy basis}, $\mathcal{P}$ is the greedy basis whose elements are parameterized by denominator vectors. However, in general this is not true even for cluster variables. We think that the polytope method might be helpful in considering this problem.
\vspace{4mm}

\section{Polytope basis for an upper cluster algebra}
Recall that in Theorem \ref{properties in case tsss}, we associate a (formal) Laurent polynomial $\rho_{h}^{pr}$ to each $h\in\Z^{n}$ in a cluster algebra $\A$ with principal coefficients. While in a cluster algebra over an arbitrary semifield $\P$, let $F_{h}=\rho^{pr}_{h}|_{x_{i}\rightarrow 1,\forall i\in[1,n]}$ and define a formal Laurent polynomial
\[\rho_{h}|_{\P}:=\frac{F_{h}|_{\F}(\hat{Y})}{F_{h}|_{\P}(Y)}X^{h}\in\N\P[[X^{\pm 1}]].\]
We denote by $\widehat{\mathcal{P}}|_{\P}$ the set consisting of all such formal Laurent polynomials $\rho_h|_{\P}$, i.e.,
\[\widehat{\mathcal{P}}|_{\P}=\{\rho_{h}|_{\P}\in\N\P[[X^{\pm 1}]]\mid h\in\Z^{n}\},\]
and $\mathcal{P}|_{\P}=\{\rho_{h}|_{\P}\in\N\P[X^{\pm 1}]\mid h\in H\}$.

In this section, the subscript of semifield is always compatible with the cluster algebra we are talking about, so we omit the subscript for convenience. We want to take further discussion about $\rho_{h}$ to construct a basis of $\mathcal{U}(\A)$ for a TSSS cluster algebra $\A$.

\begin{Lemma}\label{universally positive elements}
   For a TSSS cluster algebra $\A$ with principal coefficients, a universally positive elements $f$ in $\mathcal{U}(\A)$ can be expressed as a $\Z Trop(Y)$-linear combination of $\mathcal{P}$, that is, $f=\sum\limits_{h\in H}a_{h}\rho_{h}$ with finitely many nonzero $a_{h}\in\Z Trop(Y)$.
\end{Lemma}

\begin{Proof}
  Without loss of generality we can assume $f$ is universally indecomposable and it is written as a Laurent polynomial in $X_{t_{0}}$.

  The universal indecomposability of $f$ leads to the fact that for any two constant coefficient-free monomial summands $p$ and $p\'$ of $f$, there is a sequence of constant coefficient-free monomial summands $p=p_0,p_1,\cdots,p_s=p\'$ of $f$ satisfying $\frac{p_i}{p_{i-1}}= \hat{y}_{j_i}^{a_i}$ for some $j_i\in[1,n]$ and $a_i\in\Z$. Therefore, $f$ should be homogeneous under canonical $\Z^{n}$-grading since $\hat{y}_{j}$ is homogeneous with degree $0$ for any $j\in[1,n]$. Denote this degree by $h(f)$.

  We will show that we can find $a\in\Z$, $w\in\N^n$ and $h\in\Z^n$ such that $aY^w\rho_{h}$ is a summand of $f$ and the polytope corresponding to $f-aY^w\rho_{h}$ is a sub-polytope of that corresponding to $f$. Thus by iteratively finding such summands, the corresponding polytope finally becomes an empty set and at the same time the summation of these summands is a $\Z Trop(Y)$-linear combination we want.

  Let $N(f)$ be the corresponding polytope of $f$. Choose a minimal vector $w$ in $N(f)$. It corresponds to a monomial summand $\hat{Y}^{w}X^{h(f)}$ of $f$. Denote $h=h(f)+wB^{\top}$. Let $t_{k}$ be the vertex connected to $t_{0}$ by an edge labelled $k$ in $\T_{n}$ for any $k\in[1,n]$. Since $f$ is univerally positive, $L^{t_{k}}(f)$ is a positive Laurent polynomial in $X_{t_{k}}$. Then for any monomial summand $p$ of $f$, the sum of all monomial summands of $f$ having $x_{k}$-degree $deg_{x_{k}}(p)$ must be of the form
  \begin{equation}\label{equation in pf of univerally positive}
    x_{k}^{deg_{x_{k}}(p)}(1+\hat{y}_{k})^{[-deg_{x_{k}}(p)]_{+}}\sum\limits_{j}m_{j},
  \end{equation}
  where $m_{j}$ is a Laurent monomial in $X\setminus\{x_{k}\}$.

  Let $w\'\in V_{h}$. Then according to the definition of $V_{h}$, in $N_{h}$  there is a sequence $p_{0}=X^{h},p_{1},\cdots,p_{r}=\hat{Y}^{w\'}X^{h}$ and a sequence $i_{1},\cdots,i_{r}$ satisfying the conditions listed in the definition of $V_{h}$.

  Assume $p=\hat{Y}^{w+w\'}X^{h(f)}$ and $i_{r}=k$ (thus $deg_{x_{k}}(p)<0$). We then use induction on $r>0$ to prove that there is some $m_{j}$ in (\ref{equation in pf of univerally positive}) such that
  $$x_{k}^{deg_{x_{k}}(p)}\hat{y}_{k}^{-[\epsilon_{r}]_{+}deg_{x_{k}}(p)}m_{j}=\hat{Y}^{w+w\'}X^{h(f)},$$
  where $\epsilon_{r}$ is the same as that given in the definition of $V_h$.

  When $r=1$, if there is no $m_{j}$ in (\ref{equation in pf of univerally positive}) such that $x_{k}^{deg_{x_{k}}(p)}\hat{y}_{k}^{-deg_{x_{k}}(p)}m_{j}=\hat{Y}^{w+w\'}X^{h(f)}$, which is equivalent to there being no $m_{j}$ in (\ref{equation in pf of univerally positive}) such that $x_{k}^{deg_{x_{k}}(p)}m_{j}=\hat{Y}^{w}X^{h(f)}$. Then because $f$ is universally positive, there is some $m_{j}$ and $s\in[1,-deg_{x_{k}}(p)]$ such that $\hat{Y}^{w}X^{h(f)}=x_{k}^{deg_{x_{k}}(p)}\hat{y}_{k}^{s}m_{j}$ (otherwise $L^{t_k}(f)$ can not be positive), which means there is a vector $w-se_{k}$ in $N(f)$ less than $w$. This contradicts our choice of $w$. Hence there is some $m_{j}$ in  (\ref{equation in pf of univerally positive}) such that $$x_{k}^{deg_{x_{k}}(p)}\hat{y}_{k}^{-deg_{x_{k}}(p)}m_{j} = \hat{Y}^{w+w\'}X^{h(f)}.$$

  Suppose when $r<l$, there is some $m_{j}$ in  (\ref{equation in pf of univerally positive}) such that $$x_{k}^{deg_{x_{k}}(p)}\hat{y}_{k}^{-[\epsilon_{r}]_{+}deg_{x_{k}}(p)}m_{j}=\hat{Y}^{w+w\'}X^{h(f)}.$$

  When $r=l$, if there is no $m_{j}$ in  (\ref{equation in pf of univerally positive}) such that $$x_{k}^{deg_{x_{k}}(p)}\hat{y}_{k}^{-[\epsilon_{r}]_{+}deg_{x_{k}}(p)}m_{j}=\hat{Y}^{w+w\'}X^{h(f)},$$
  then according to the mutation formula, the Laurent monomial corresponding to the cross is not a summand of $L^{t_{k}}(f)$ as showed in Figure \ref{figure: universally positive}. Here the Laurent polynomials corresponding to the red line in the left-hand side correlates to the Laurent monomial corresponding to the red point in the right-hand side under the mutation in direction $k$.
  \begin{figure}[H]
    \centering
    \includegraphics[width=40mm]{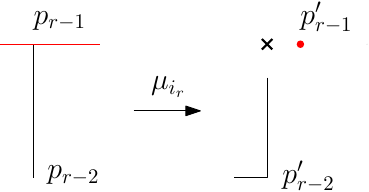}
    \caption{A part of $N(f)$ under mutation.}\label{figure: universally positive}
  \end{figure}
  Let $t\'$ be the vertex connected to $t_{k}$ by an edge labelled $i_{r-1}$ in $\T_{n}$. Since $L^{t\'}(f)$ is a positive Laurent polynomial in $X_{t\'}$, the Laurent monomial $Y^{w}p_{r-2}$ is a summand of $$x_{i_{r-1}}^{deg_{x_{i_{r-1}}}(p_{r-2})}\hat{y}_{i_{r-1}}^{-[\epsilon_{r-1}]_{+}(deg_{x_{i_{r-1}}}(p_{r-2})+s)}m_{j}$$
  for some $j$, where $s\in[1,-deg_{x_{i_{r-1}}}(p_{r-2})]$ rather than $s=0$. This contradicts the inductive assumption.

  Hence there is some $m_{j}$ in  (\ref{equation in pf of univerally positive}) such that $$x_{k}^{deg_{x_{k}}(p)}\hat{y}_{k}^{-[\epsilon_{r}]_{+}deg_{x_{k}}(p)}m_{j}=\hat{Y}^{w+w\'}X^{h(f)}.$$

  Moreover, since $f$ is universally positive, the above claim holds for $L^t(f)$ with any $t\in\T_n$. Then according to the description of edges as well as 2-dimensional faces of $N_h$ for any $h\in\Z^n$ in the proof of Theorem \ref{properties in case tsss} (v), for any $w\'\in V(N_h)$ and $k\in[1,n]$, there is $\epsilon_{w\',k}\in\{\pm 1\}$ such that
  $$m_{k}(w\')\hat{Y}^{w\'}X^{h}(1+\hat{y}_{k}^{\epsilon_{w\',k}})^{[-deg_{x_{k}}(\hat{Y}^{w\'}X^{h})]_+}$$ is a summand of $\rho_{h}$ as a complement of $m_{k}(w\')\hat{Y}^{w\'}X^{h}$ in direction $k$, and $$m_{k}(w\')\hat{Y}^{w+w\'}X^{h(f)}(1+\hat{y}_{k}^{\epsilon_{w\',k}})^{[-deg_{x_{k}}(\hat{Y}^{w\'}X^{h})]_+}$$
  is a summand of $f$.

  Then in $f-co_{\hat{Y}^{w}X^{h}}(f)Y^{w}\rho_{h}$, the vectors corresponding to Laurent monomials with negative constant coefficients are in $(N_{h}\setminus V(N_h)[w]$. Denote by $W_{h}$ the set consisting of these vectors. Since the support of $f-co_{\hat{Y}^{w}X^{h}}(f)Y^{w}\rho_{h}$ is finite, we can find a minimal Laurent polynomial $\sum\limits_{u\in W_{h}}a_{u}\rho_{h(f)+uB}$ for some $a_{u}\in\N\P$ such that
  $$f-co_{\hat{Y}^{w}X^{h(f)}}(f)Y^{w}\rho_{h}+\sum\limits_{u\in W_{h}}a_{u}\rho_{h(f)+uB^{\top}}$$
  is universally positive. Denote it by $f_{1}$ and its corresponding polytope by $N(f_{1})$. Note that
  $$co_{\hat{Y}^{w}X^{h(f)}}(f_{1})=0,$$
  and $co_{\hat{Y}^{u}X^{h(f)}}(f_{1})=0$ for any $u\in W_{h}$ with $a_{u}\neq 0$ due to the minimality of $\sum\limits_{u\in W_{h}}a_{u}\rho_{h(f)+uB^{\top}}$.

  Since the origin is the unique minimal point in $N_{h}$, $u> w$ for any $u\in W_{h}$. So any minimal point $p$ in $N(f_{1})$, is either a minimal point in $N(f)$ or $p> w$. Moreover, if there is $u\in W_{h}$ such that the maximal point $p$ in $N_{h(f)+uB^{\top}}[u]$ is a maximal point in $N(f_{1})$, then similar to our above discussion (the maximal point case is dual to the minimal point case in the above), $co_{p}(N(f_{1}))\neq 0$ for any point $p\in V(N_{h(f)+w\'B^{\top}})[w\']$, in particular $co_{\hat{Y}^{w\'}X^{h(f)}}(f_{1})\neq 0$, which contradicts the fact. Hence there is no maximal point $p\'$ in $N(f_{1})$ which is not a maximal point in $N(f)$. So the set consisting of minimal points in $N(f_{1})$ is a union of a proper subset of minimal point set of $N(f)$ and a set of points strictly lager than some minimal points in $N(f)$, while the set consisting of maximal points in $N(f_{1})$ is a subset of that of maximal points in $N(f)$.

  Therefore, replacing $f$ by $f_{1},\cdots$, we can continue the above way to iteratively produce new universally positive elements $f_{2},\cdots$, and will finally obtain 0 in finitely many times as the minimal points becoming larger while the maximal points becoming smaller. Summing these together we have an equation $f=\sum\limits_{h\in H}a_{h}\rho_{h}$ for some $a_{h}\in\Z\P$.
\end{Proof}

In Lemma \ref{universally positive elements}, $a_{h}$ is not necessarily in $\N Trop(Y)$. There is a characterization about  when $a_{h}$ is always in $\N Trop(Y)$ for a cluster algebra without coefficients of rank 2 in \cite{LLZ2}. So we may ask naturally how  it is in general case.
\begin{Problem}
  In Lemma \ref{universally positive elements}, for what kind of cluster algebras, is $a_{h}$ always in $\N Trop(Y)$?
\end{Problem}

\begin{Theorem}\label{positive}
  Let $\A$ be a TSSS cluster algebra with principal coefficients. Then $\mathcal{P}$ is a strongly positive $\Z Trop(Y)$-basis for the upper cluster algebra $\mathcal{U}(\A)$.
\end{Theorem}
\begin{Proof}
  First, we prove  that $\mathcal{U}(\A)$ is linearly generated by $\mathcal{P}$ over $\Z Trop(Y)$.

  Let $f$ be an element in $\mathcal{U}(\A)$ as a Laurent polynomial in $X$. Then $f$ has a decomposition $f=\sum\limits_{i=1}^{l}f_{i}$, where $f_{i}\neq 0$ is a summand of $f$ such that $f_{i}$ is indecomposable in $\mathcal{U}(\A)$. Then as we said before, $f_{i}$ is homogeneous under canonical $\Z^{n}$-grading since the exchange binomials are all homogeneous with degree $0$. Denote $h_{i}=deg(f_{i})$.

  Denote by $N_{i}$ the Newton polytope of $f_{i}$. For a minimal lattice point $p$ in $N_{i}$, according to the proof of Lemma \ref{universally positive elements}, the maximal point in $N_{h_{i}-pB^{\top}}[p]$ is not larger than a maximal point of $N_{i}$. Let $f_{i}^{(1)}=f_{i}-co_{p}(N_{i})\hat{Y}^{p}\rho_{h_{i}-pB^{\top}}$ and repeat the above process on $f_{i}^{(j)}$ to get $f_{i}^{(j+1)}$. As explained in the proof of Lemma \ref{universally positive elements} it will stop in finitely many steps. Thus we decompose $f_{i}$ as a $\Z Trop(Y)$-linear combination of $\mathcal{P}$. Alternatively, once $a_{p}$ is large enough for each $p\in N_{i}$, $f_{i}+\sum\limits_{p\in N_{i}}a_{p}Y^{p}\rho_{h_{i}-pB^{\top}}$ is universally positive. Then Lemma \ref{universally positive elements} induces a $\Z Trop(Y)$-linear combination for $f_{i}$.

  Next, we need to show that $\mathcal{P}$ is linearly independent.

  Assume $\sum\limits_{i=1}^{l}a_{i}Y^{w_{i}}\rho_{h_i}=0$ for some $0\neq a_{i}\in\Z$, $w_{i}\in\Z^{n}$ and $h_i\in H$ with
  \[Y^{w_{i}}\rho_{h_i}\neq Y^{w_{j}}\rho_{h_j} \;\; \text{  when}\;\; i\neq j.\]
  Without loss of generality, we suppose all $a_{i}Y^{w_{i}}\rho_{h_i}$ have the same degree, then $-w_{i}B^{\top}+h_{i}=-w_{j}B^{\top}+h_{j}$ for any $i,j$. Hence if $w_i=w_j$, it follows that $h_{i}=h_{j}$. So, when $i\neq j$, there must be $w_i\neq w_j$. Choose a minimal $w_{j}$ among all $w_{i}$. Then $co_{w_{j}}(\sum\limits_{i=1}^{l}a_{i}N_{i}[w_{i}])=a_{j}\neq 0$, which contradicts to $\sum\limits_{i=1}^{l}a_{i}Y^{w_{i}}\rho_{h_i}=0$. So $\mathcal{P}$ is linearly independent.

  In summary, we get that $\mathcal{P}$ is a $\Z Trop(Y)$-basis of $\mathcal{U}(\A)$.

  Finally we show that $\mathcal{P}$ is strongly positive.

  Choose arbitrary  $\rho_{h},\rho_{h\'}\in\mathcal{P}$. Then $\rho_{h}\rho_{h\'}=\sum\limits_{f}a_{f}\rho_{f}$ with $a_{f}\in\Z[Y]$. Therefore, we only need to prove that $a_{f}\in\N[Y]$.

  If $h,h\'\in\N^{n}$, then $\rho_{h}=X^{h}$ and $\rho_{h\'}=X^{h\'}$. So $\rho_{h}\rho_{h\'}=X^{h+h\'}=\rho_{h+h\'}$, the claim holds.

  Otherwise, at least one of $h,h\'$ is in $\Z^{n}\setminus\N^{n}$. Without loss of generality, suppose $h\in\Z^{n}\setminus\N^{n}$. Now we take induction on the partial order $\leqslant$ in the set of polytopes. Assume
  \begin{equation}\label{inductionpartial}
  \rho_{r}\rho_{h\'}=\sum\limits_{r\'\in\Z^{n}} c_{r\'}\rho_{r\'}
  \end{equation}
  where $c_{r\'}\in\N[Y]$ for any $\rho_{r},\rho_{h\'}\in\mathcal{P}$ with $N_{r}<N_{h}$.

  According to (\ref{equation: decomposition of rho_h for x_i}), we have a decomposition
  \[x_{i}\rho_{h}=\mathop{\sum}_{N_{\alpha}[w]\in\bigcup\limits_{s} U_{h}^{s},\atop \alpha\in \Z^n, w\in\N^n}Y^{w}\rho_{\alpha}=\sum\limits_{r\in\Z^{n}}b_{r}\rho_{r},\]
  where $b_{r}\in\N[Y]$ and there are only finitely many nonzero. Therefore,
  \begin{equation}\label{nonhyper}
  x_{i}\rho_{h}\rho_{h\'}=\sum\limits_{r\in\Z^{n}} b_{r}\rho_{r}\rho_{h\'}.
  \end{equation}
  Because $h\in\Z^{n}\backslash\N^{n}$, there is $i\in[1,n]$ such that $N_{h}$ is not in the hyperpane $z_{i}=0$. Then there must be more than one terms in the right-hand side of (\ref{nonhyper}). Then, combining (\ref{inductionpartial}) and (\ref{nonhyper}), we have
  $$x_{i}\rho_{h}\rho_{h\'}=\sum\limits_{r\in\Z^{n}}\sum\limits_{r\'\in\Z^{n}} b_{r}c_{r,r\'}\rho_{r\'}.$$

  Again due to the construction of strata, we claim that there is a decomposition $\sum\limits_{r\in\Z^{n}}\sum\limits_{r\'\in\Z^{n}} b_{r}c_{r,r\'}\rho_{r\'}=\sum\limits_{\lambda}\sum\limits_{(a_{r\'},r\')\in J_{\lambda}}a_{r\'}\rho_{r\'}$ satisfying $a_{r\'}\in\N[Y]$ and
  \begin{equation*}
   \sum\limits_{(a_{r\'},r\')\in J_{\lambda}}a_{r\'}\rho_{r\'}=x_{i}Y^{w_{\lambda}}\rho_{f_{\lambda}}
  \end{equation*}
  for some $w_{\lambda}\in\N^{n}$ and $f_{\lambda}\in\Z^{n}$. This can be done in the following way:

  Choose a minimal point $p\in N$, where $N$ is the polytope corresponding to $\rho_{h}\rho_{h\'}$, then with the help of the strata of $N_{h}$ for $x_{i}$ and inductive assumption, it can be checked that the strata of $N_{h+h\'-pB^{\top}}[p]$ for $x_{i}$ are in $N$. Hence it follows that $N_{h+h\'-pB^{\top}}[p]$ is a sub-polytope of $N$. Let $\{(N_{r\'}[w_{r\'}])^{a_{r\'}}\}$ be the set of all strata of $N_{h+h\'-pB^{\top}}[p]$ for $x_{i}$. Then let $J_{\lambda_0}$ be a subset of $J$ consisting of those $(a_{r\'}, r')$, $w_{\lambda_0}=p$ and $f_{\lambda_0}$, we have
  $$\sum\limits_{(a_{r\'},r\')\in J_{\lambda_0}}a_{r\'}\rho_{r\'}=x_{i}Y^{w_{\lambda_0}}\rho_{f_{\lambda_{0}}}.$$
  Let $N_{1}$ denote the sub-polytope of $N$ corresponding to $\sum\limits_{r\in\Z^{n}}\sum\limits_{r\'\in\Z^{n}} b_{r}c_{r,r\'}\rho_{r\'}-\sum\limits_{(a_{r\'},r\')\in J_{\lambda}}a_{r\'}\rho_{r\'}$. Repeating the above process for $N_1,\cdots$ until $N_{j}=\emptyset$ for some $j$, we can determine all $J_{\lambda}$ and thus obtain the decomposition we want.

  So in conclusion, we have $\rho_{h}\rho_{h\'}=\sum\limits_{\lambda}\sum\limits_{(a_{r\'},r\')\in J_{\lambda}}x_{i}^{-1}a_{r\'}\rho_{r\'}=\sum\limits_{\lambda}Y^{w_{\lambda}}\rho_{f_{\lambda}}$, thus $a_{f}=\sum\limits_{f_{\lambda}=f}Y^{w_{\lambda}}\in\N[Y]$ for $\rho_{h}\rho_{h\'}=\sum\limits_{f}a_{f}\rho_{f}$.
\end{Proof}

Due to Theorem \ref{positive}, we call $\mathcal{P}$ the {\bf polytope basis} of $\mathcal{U}(\A)$ since the method of polytopes is used effectively in the construction of $\mathcal{P}$.

However in general for a cluster algebra $\A$ over an arbitrary semifield $\P$, we find that $\mathcal{P}$ is a $\Z\P$-basis for a subalgebra of $\mathcal{U(A)}$ rather than $\mathcal{U(A)}$ itself. Therefore, we call the $\Z\P$-subalgebra generated by $\mathcal P$ the \textbf{intermediate cluster algebra} associated to $\A$, and denote it by $\mathcal{I_{P}(A)}$. In fact, we have:

\begin{Theorem}\label{positive over arbitray semifield}
  Let $\A$ be a cluster algebra over a semifield $\P$. Then $\mathcal{P}$ is a strongly positive $\Z\P$-basis for the intermediate cluster algebra $\mathcal{I_{P}(A)}$.
\end{Theorem}
\begin{Proof}
  First we prove that $\mathcal{P}$ is $\Z\P$-linearly independent. Let $P\in \mathcal{U}(\A)$ be a $\N\P$-linear combination of $\mathcal{P}$ as $P=\sum\limits_{i=1}^{l}a_{i}P_{i}$, where $a_{i}\in\N\P$ and $P_{i}\in\mathcal{P}$ is a universally indecomposable summand of $P$ for each $i\in[1,l]$. We claim such decomposition is unique.

  Assume on the contrary there are two decompositions $P=\sum\limits_{i=1}^{l}a_{i}P_{i}=\sum\limits_{j=1}^{l\'}a\'_{j}P\'_{j}$, where $a_{i},a\'_{j}\in\N\P$ and $P_{i}=\rho_{h_{i}},P\'_{j}=\rho_{h\'_{j}}$ for some $h_{i},h\'_{j}\in H$ for each $i\in[1,l],j\in[1,l\']$. For any $k\in[1,n]$ and any monomial summand $p$ of $P$ with constant coefficient $1$, let $P\'$ be the sum of all monomial summands of $P$ with the same $x_{k}$-degree as that of $p$. So $P\'$ is of the form $P\'=x_{k}^{deg_{x_{k}}(p)}M_{k}^{[-deg_{x_{k}}(p)]_{+}}L_{k}$, where $L_{k}$ is a Laurent polynomial in $\N\P[X^{\pm 1}]$, and there is a (not necessarily unique) monomial summand $r_{k}$ of $L_{k}$ such that $p$ is a summand of $x_{k}^{deg_{x_{k}}(p)}M_{k}^{[-deg_{x_{k}}(p)]_{+}}r_{k}$. For any $t\in\T_{n}$, denote
  \[J^{(t)}=\{p\mid p\text{ is a monomial summand of $P^{(t)}$ with constant coefficient $1$ and}\qquad\]
  \[\qquad \text{there is a summand }r_{k}\text{ of }L_{k}^{(t)}\text{ such that }p=x_{k;t}^{deg_{x_{k;t}}(p)}(\frac{\prod\limits_{s=1}^{n}x_{s;t}^{[-b^{t}_{sk}]_{+}}}{y_{k;t}\oplus 1})^{[-deg_{x_{k;t}}(p)]_{+}}r_{k}, \forall k\in[1,n]\},\]
  where $P^{(t)}$ is the Laurent expression of $P$ in $X_{t}$ and $L_{k}^{(t)}$ is similarly defined as above for $P^{(t)}$. Then Let
  \[J=\{p\in J^{(t_{0})}|\text{there is }p^{(t_{i})}\in J^{(t_i)}\text{ correlated to }p\text{ during }\mu_i\text{ for any }i\in[1,n]\},\]
  where $t_i$ is the vertex adjacent to $t_0$ by an edge labeled $i$.

  Recall that in principal coefficients case, $X^{h}$ is the only minimal monomial summand in $\rho_{h}$ and $\rho_{h}$ is universally indecomposable for any $h\in H$. Hence because of the decomposition $P=\sum\limits_{i=1}^{l}a_{i}P_{i}$, we have $J=\{\frac{a_{i}}{co_{X^{h_{i}}}(P)} X^{h_{i}}\mid i\in[1,l]\}$. On the other hand, $J=\{\frac{a\'_{j}}{co_{X^{h_{j}}}(P)} X^{h\'_{j}}\mid j\in[1,l\']\}$ due to the decomposition $P=\sum\limits_{j=1}^{l\'}a\'_{j}P\'_{j}$. Therefore the above two decompositions must be the same.

  So for any finite equation $\sum\limits_{h\in H}a_{h}\rho_{h}=0$, with $a_{h}\in\Z\P$, we get $\sum\limits_{h\in H}a\'_{h}\rho_{h}=\sum\limits_{h\in H}a^{\prime\prime}_{h}\rho_{h}$ with $a\'_{h},a^{\prime\prime}_{h}\in\N\P$. Following the above discussion, $a\'_{h}=a^{\prime\prime}_{h}$, i.e., $a_{h}=0$ for each $h\in H$. So $\mathcal{P}$ is $\Z\P$-linear independent.

  Due to $\rho_{h}|_{\P}=\frac{F_{h}|_{\F}(\hat{Y})}{F_{h}|_{\P}}X^{h}$, the strongly positivity in principal coefficients case leads to that in other semifield cases. Hence $\mathcal{P}$ is a strongly positive $\Z\P$-basis for $\mathcal{I_{P}(A)}$.
\end{Proof}

In particular when $\A=\mathcal{U(A)}$, $\mathcal{P}$ is a strongly positive $\Z\P$-basis for $\A$.

We wonder when $\mathcal{I_{P}(A)}$ coincides with $\mathcal{U(A)}$. The following corollary of Theorem \ref{positive} and Theorem \ref{positive over arbitray semifield} provides an equivalent condition of $\mathcal{I_{P}(A)=U(A)}$.

\begin{Corollary}\label{equivalent condition for arbitrary semifield}
  Let $\A$ be a cluster algebra over a semifield $\P$. Then $\mathcal{I_{P}(A)=U(A)}$ if and only if for any universally indecomposable $f\in\mathcal{U(A)}$, there is a homogenous element $\tilde{f}\in\mathcal{U}(A_{prin})$ satisfying $a\frac{F|_{\F}(\hat{Y})}{F|_{\P}(Y)}X^{g}=f$, where $a\in\N\P$, $F=L(\tilde{f})|_{x_{i}\rightarrow 1,\forall i\in[1,n]}$ and $g$ is the degree of $\tilde{f}$.

  In particular, $\mathcal{I_{P}(A)=U(A)}$ if $B$ (or $\tilde{B}$ for geometric type) has full rank.
\end{Corollary}
\begin{Proof}
  (``only if''):\; Because $\mathcal{I_{P}(A)=U(A)}$, $\mathcal{P}$ is a $\Z\P$-basis of $\mathcal{U(A)}$. So for any universally indecomposable element $f\in\mathcal{U(A)}$, there is a unique $\Z\P$-linear combination $f=\sum\limits_{i} a_{i}\rho_{h_{i}}$ with finitely many $a_{i}\neq 0\in\Z\P$. Then because $f$ is universally indecomposable, we can find $\tilde{a}_{i}\in\Z Trop(Y)$ such that $\tilde{f}=\sum\limits_{i} \tilde{a}_{i}\rho_{h_{i}}\in\mathcal{U}(\A_{prin})$ is homogenous and universally indecomposable, and $a\frac{F|_{\F}(\hat{Y})}{F|_{\P}(Y)}X^{g}=f$ for some $a\in\N\P$, $F=L(\tilde{f})|_{x_{i}\rightarrow 1,\forall i\in[1,n]}$ and $g$ being the degree of $\tilde{f}$.

  (``if''):\; Because for any universally indecomposable $f\in\mathcal{U(A)}$, there is a homogenous element $\tilde{f}\in\mathcal{U}(A_{prin})$ satisfying $a\frac{F|_{\F}(\hat{Y})}{F|_{\P}(Y)}X^{g}=f$, where $a\in\N\P$, $F=L(\tilde{f})|_{x_{i}\rightarrow 1,\forall i\in[1,n]}$ and $g$ is the degree of $\tilde{f}$, so for any $f\in\mathcal{U(A)}$, similar to the proof of Theorem \ref{positive}, we can find a $\Z\P$-linear combination of $f$ in universally indecomposable elements with the help of the set analogous to $J$ in Theorem \ref{positive over arbitray semifield}. So we only need to deal with the case where $f$ is universally indecomposable. Then there is $\tilde{f}\in\mathcal{U}(A_{prin})$ satisfying $a\frac{F|_{\F}(\hat{Y})}{F|_{\P}(Y)}X^{g}=f$, where $a\in\N\P$, $F=L(\tilde{f})|_{x_{i}\rightarrow 1,\forall i\in[1,n]}$ and $g$ is the degree of $\tilde{f}$.

  Since $f$ is universally indecomposable, so is $\tilde{f}$. Then by Lemma \ref{universally positive elements} or Theorem \ref{positive}, we have a $\Z Trop(Y)$-linear combination of $\tilde{f}$ as $\tilde{f}=\sum\limits_{j}\tilde{c}_{j}\rho_{h_{j}}$ in $\mathcal{U}(\A_{prin})$ for some $\tilde{c}_{j}\in\Z Trop(Y)$. Hence $f=a\frac{F|_{\F}(\hat{Y})}{F|_{\P}(Y)}X^{g}=\sum\limits_{j}c_{j}\rho_{h_{j}}$ in $\mathcal{U(A)}$ for some $c_{j}\in\Z\P$. So $\mathcal{I_{P}(A)=U(A)}$.

  In particular, when $B$ or $\tilde{B}$ has full rank, for any universally indecomposable $f\in\mathcal{U(A)}$, there is unique $\tilde{f}\in\mathcal{U}(A_{prin})$ up to multiplying a Laurent monomial in $Y$ satisfying $a\frac{F|_{\F}(\hat{Y})}{F|_{\P}(Y)}X^{g}=f$, where $a\in\N\P$, $F=L(\tilde{f})|_{x_{i}\rightarrow 1,\forall i\in[1,n]}$ and $g$ is the degree of $\tilde{f}$. So in this case $\mathcal{I_{P}(A)=U(A)}$.
\end{Proof}

The following example given by Yan Zhou in \cite{Zh} is a counterexample where the condition in Corollary \ref{equivalent condition for arbitrary semifield} fails, thus $\mathcal{I_{P}(A)}\subsetneqq\mathcal{U(A)}$.
\begin{Example}[\cite{Zh}]
  Let $\A$ be the cluster algebra without coefficients associated to the exchange matrix
  \[B=\begin{pmatrix}
        0 & 2 & -2 \\
        -2 & 0 & 2 \\
        2 & -2 & 0
      \end{pmatrix}.\]
  Then $P:=\frac{x_{1}^{2}+x_{2}^{2}+x_{3}^{2}}{x_{1}x_{2}x_{3}}\in\mathcal{U}(\A)$ and it is universally indecomposable. But $P$ can not be written as a $\Z$-linear combination of $\mathcal{P}$. Otherwise by Corollary \ref{equivalent condition for arbitrary semifield} there must be $a_{1},a_{2},a_{3}\in\Z Trop(Y)$ satisfying $\tilde{P}=\frac{a_{1}x_{1}^{2}+a_{2}x_{2}^{2}+a_{3}x_{3}^{2}}{x_{1}x_{2}x_{3}}\in\mathcal{U}(\A_{prin})$ and $a_{i}|_{y_{j}\rightarrow 1,\forall j\in[1,n]}=1$ for $i\in\{1,2,3\}$. However, by calculating the Laurent expression of $\tilde{P}$ in $X_{t_{i}}$, where $t_{i}$ is connected with $t_{0}$ by an edge labeled $i\in\{1,2,3\}$, we get that either $a_{3}y_{2}$ is a proper summand of $a_{1}$ and at the same time $a_{1}$ is a proper summand of $a_{3}y_{2}$ or $P$ has some monomial summand other than $x_{i}^{-1}x_{j}^{-1}x_{k}$, where $\{i,j,k\}=\{1,2,3\}$, which is impossible.
\end{Example}

\begin{Corollary}\label{semigroup}
  Let $\A$ be a cluster algebra over a semifield $\P$. Then for any $h,h\'\in\Z^{n}$, $\rho_{h+h\'}$ is a summand of $\rho_{h}\rho_{h\'}$. Hence $H$ is an additive sub-monoid of $\Z^{n}$.
\end{Corollary}
\begin{Proof}
  It is sufficient to deal with the principal coefficients case. Since $X^{h}$ and $X^{h\'}$ are summand of $\rho_{h}$ and $\rho_{h\'}$ respectively, $X^{h+h\'}$ is a summand of $\rho_{h}\rho_{h\'}$. Moreover, according to Theorem \ref{positive}, $\rho_{h}\rho_{h\'}=\sum\limits_{g}a_{g}\rho_{g}$ with $a_{g}\in\N[Y^{\pm}]$. Then because $\rho_{h+h\'}$ is the unique element in $\widehat{\mathcal{P}}$ having $X^{h+h\'}$ as a summand, $\rho_{h+h\'}$ is a summand of $\rho_{h}\rho_{h\'}$.
\end{Proof}

We would like to end this section with the following results showing that in many ``good" situations, $\rho_{h}$ is a Laurent polynomial for any $h\in\Z^{n}$ and hence $\mathcal{P}=\widehat{\mathcal{P}}$.

\begin{Proposition}
  Let $\A$ be a cluster algebra over a semifield $\P$. Then $\mathcal{P}=\widehat{\mathcal{P}}$ if and only if $\rho_{-e_{i}}\in\mathcal{P}$ for any $i\in[1,n]$.
\end{Proposition}
\begin{Proof}
  The necessity is due to the definition of $\widehat{\mathcal{P}}$. So, we only need to show the sufficiency. It is enough to deal with the case where $\A$ has principal coefficients. Assume $\rho_{-e_{i}}\in\mathcal{P}$ for any $i\in[1,n]$. $\rho_{e_{i}}=x_{i}\in\mathcal{P}$ for any $i\in[1,n]$. For any $h\in\Z^{n}$, denote $h=(h_{1},\cdots,h_{n})$. Hence $$h=\sum\limits_{i=1}^{n}h_{i}e_{i}=\sum\limits_{i=1}^{n}[h_{i}]_{+}e_{i}+\sum\limits_{i=1}^{n}[-h_{i}]_{+}(-e_{i}).$$ So according to Corollary \ref{semigroup}, we have $h\in H$. Hence, $\rho_{h}\in\mathcal{P}$, and thus $\mathcal{P}=\widehat{\mathcal{P}}$.
\end{Proof}

Motivated by applications to non-commutative Donaldson-Thomas theory,  Keller introduced the idea of maximal green sequences and reddening sequences in \cite{K,K2}. An equivalent definition for a skew-symmetric cluster algebra to admit reddening sequences is that it has $-e_{i}$ as a $g$-vector associated to some cluster variable for any $i\in[1,n]$ (see \cite{M} for more details). Hence the following result is a direct corollary of the above proposition.
\begin{Corollary}
  Let $\A$ be a skew-symmetric cluster algebra over a semifield $\P$. If $\A$ admits reddening sequences, then $\mathcal{P}=\widehat{\mathcal{P}}$.
\end{Corollary}

\vspace{4mm}

\section{Construction of faces of polytopes in a skew-symmetrizable cluster algebra}\quad

In this section we assume $\A$ is a skew-symmetrizable cluster algebra. Recall that in this paper we denote the skew-symmetrizer $D=diag(d_{1},\cdots,d_{n})$.
Then for the exchange matrix $B_{t}$ of $\A$ at any vertex $t$, $DB_{t}$ is skew-symmetric, which will help us to construct the proper faces of $N_{h}$ more explicitly.

When $\A$ is of rank 2, we have listed the possible shape of all $N_{h}$ of dimension 2 in Figure \ref{figure: shape for rank 2} (3)-(7). Any polytope $N_{h}$ satisfies one of the two following conditions (a) and (b) from the proof of Theorem \ref{properties in case tsss}:

(a)\;There is $t\in\T_{n}$ such that $N_{h^{t}}^{t}$ is the origin, i.e., the polytope has dimension 0.

(b)\;For any $t\in\T_{n}$, $N_{h^{t}}^{t}$ has dimension 2.

A polytope $N_{h}$ whose shape is one of (3)-(6) satisfies the condition (a) while that with shape (7) may satisfy (a) or (b). In order to determine which condition such a polytope satisfies, we need to analyze the lengths of two orthogonal edges parallel to the coordinate axis. Assume the initial exchange matrix is
\[B\'=\begin{pmatrix}
     0 & \epsilon b \\
     -\epsilon c & 0
   \end{pmatrix}\]
with $b,c\in\N,\epsilon\in\{\pm 1\}$ and the lengths of two orthogonal edges parallel to the coordinate axis in $N_{h}$ are $aw_{1}$ and $aw_{2}$ respectively with $(w_1,w_2)=1$.

A polytope satisfying (a) is mutation equivalent to $N^{t}_{\alpha}$ for some $\alpha\in\N^{2}$. So by calculating the polytope $N_{e_{i}^{t}}^{t}$ for $i=1,2$ and $t\in\T_{n}$,  a polytope with shape (7) satisfies (a) if and only if $w_{1}$ and $w_{2}$ satisfy the following Condition \ref{condition star} for five cases, where all sets are {\em multi-sets}:

\begin{Condition}\label{condition star}
  Case 1: when $bc=1$, it holds that $\{w_1, w_2\}=\{1,1\}$ and let $s=1$ or $2$;

  Case 2: when $bc=2$, it holds that $\{w_{1},w_{2}\}$ equals $\{1,1\}$ or $\{1,2\}$ and let $s$ satisfy $|b\'_{rs}|=max\{w_{1},w_{2}\}$ for some $r$;

  Case 3: when $bc=3$, it holds that $\{w_{1},w_{2}\}$ equals $\{1,1\}$, $\{1,2\}$, $\{1,3\}$ or $\{2,3\}$ and let $s$ satisfy $|b\'_{rs}|=3$ if $3\in\{w_{1},w_{2}\}$ while $|b\'_{rs}|=1$ if $3\notin\{w_{1},w_{2}\}$ for some $r$;

  Case 4: when $bc=4$,  there is $i\in\Z_{>0}$ such that $\{w_{1},w_{2}\}$ equals $\{2i-1,4i\}$, $\{2i+1,4i\}$, $\{i,i+1\}$, $\{2i-1,i\}$ or $\{2i+1,i\}$ and let $s\in[1,2]$ satisfy that either both $w_{1}$ and $w_{2}$ are odd and $|b\'_{rs}|=1$ for some $r$ or $w_{s}$ is the only odd in $\{w_{1},w_{2}\}$;

  Case 5: when $bc>4$,  there is $i,j\in\N$ and $s\neq r\in\{1,2\}$ such that $j-i\in \{0,1\}$, $$\{w_{1},w_{2}\}=\{\frac{f^{j+1}+f^{j}-k^{j}-k^{j+1}}{f-k},\frac{(f^{i+1}-k^{i+1})}{f-k}|b\'_{rs}|\}$$
  and $s$ satisfies that $w_{s}=\frac{f^{j+1}+f^{j}-k^{j}-k^{j+1}}{f-k}$, where
  $f=\frac{\sqrt{b^{2}c^{2}-4bc}+bc-2}{2}$ and $k=\frac{-\sqrt{b^{2}c^{2}-4bc}+bc-2}{2}$.
\end{Condition}

It can be checked that both $\frac{f^{j+1}+f^{j}-k^{j}-k^{j+1}}{f-k}$ and $\frac{(f^{j+1}-k^{j+1})}{f-k}|b\'_{rs}|$ increase along the increasing of $j$ when $bc>4$. So in practice Condition \ref{condition star} is not so hard to verify.

We assign a label to each edge in $E(N_{h})$. Denote $E^{(0)}=E_{h}\cap E(N_{h})$. And to each edge in $E^{(0)}$ assign a label by the way introduced in the definition of $E_{h}$. Then inductively let $E^{(r+1)}$ consist of the edges of $N_h$ not in $\bigcup\limits_{i=0}^{r}E^{(i)}$ and each of them is an edge of a 2-dimensional face $S$ of $N_{h}$ which contains an edge in $\bigcup\limits_{i=0}^{r}E^{i}$ parallel to $v$ for each vector $v$ in its lattice generating set based on its minimal vertex. According to Theorem \ref{properties in case tsss} (v), there is $h_{S}\in\Z^{ldim(S)}$, a cluster algebra $\A\'$ and a non-negative polytope projection $\tau$ from $N_{h_{S}}|_{\A\'}$ to $S$. Denote

For each edge $l\in S\cap E^{(r+1)}$, focus on the 2-dimensional face $S\'$ of $N_{h_{S}}|_{\A\'}$ such that $\tau$ maps an edges of $S\'$ not parallel to any coordinate axis to $l$ and denote the matrix obtained from the initial exchange matrix of $\A\'$ via deleting all rows and columns paralyzed by indices $k$ such that there is no edge of $S\'$ parallel to $e_k$ as
$(b\'_{uv})_{2\times2}=\begin{pmatrix}
     0 & \epsilon b \\
     -\epsilon c & 0
   \end{pmatrix}$ with $b,c\in\N,\epsilon\in\{\pm 1\}$. Assign to $l$ the label $j$ if it parallels to $e_j$, otherwise assign the label $\kappa(l)$ by labels of edges in $\bigcup\limits_{i=0}^{r}E^{(i)}$ as
\begin{equation}\label{label of an edge}
  \kappa(l)=\left\{\begin{array}{ll}
                   \text{the label assigned to the edge}\\\text{ of }S\text{ not intersecting with }l, & \text{if }S\'\text{ is of the form (3), (4) or (5) in Figure \ref{figure: shape for rank 2}}; \\\\
                   \text{the label assigned to the}&\text{if }S\'\text{ is of the form (6) or(7) in Figure \ref{figure: shape for rank 2}}\\
                   \text{ edge of }S\text{ parallel to }\tilde{\tau}(e_{s}),&  \text{and }w_{1},w_{2},s\text{ satisfy Condition \ref{condition star}};\\\\
                   0,& \text{otherwise}.
                 \end{array}
  \right.
\end{equation}
where $w_{1}$ and $w_{2}$ satisfy $\tilde{\tau}^{-1}(p-q)=w_{1}e_{1}+w_{2}e_{2}$ for arbitrary two points $p>q$ in $l$ such that there is no other point in $\overline{pq}$. Note that the edges in the right-hand side of (\ref{label of an edge}) are those in $\bigcup\limits_{i=0}^{r}E^{(i)}$, so the inductive definition is well-defined.

In the skew-symmetrizable case, the following result helps us to find the cluster algebra $\A'$ so as to construct the proper face of $N_{h}$ and hence the support of it more conveniently.

\begin{Theorem}\label{properties of N_h for skew-symmetrizable}
  In Theorem \ref{properties in case tsss} (v), if $\A$ is a skew-symmetrizable cluster algebra with principal coefficients, and denote by $B\'$ the initial exchange matrix of the cluster algebra $\A\'$, then $B\'= \overline{W}^{\top}BW$, where $W=(\tilde{\tau}(e_{1})^{\top},\cdots,\tilde{\tau}(e_{r})^{\top})$, $\overline{W}=(\overline{w}_1^{\top},\cdots,\overline{w}_r^{\top})$ are $n\times r$ integer matrices, $\overline{w}_i=\sum\limits_{j=1}^{n}\frac{d_{j}}{d_{s}}w_{ji}e_{j}$ for $\tilde{\tau}(e_{i})=\sum\limits_{j=1}^{n}w_{ji}e_{j}$, with $s\neq 0$ being the label of the edge in $S$ parallel to $\tilde{\tau}(e_{i})$ while $\overline{w}_i=\sum\limits_{j=1}^{n}d_{j}w_{ji}e_{j}$ when the label is 0.
\end{Theorem}

\begin{Proof}
  We denote $W =(w_{ji})_{n\times r}$ and $w_i$ the $i$-th column of $W$ for $i\in [1,r]$. If there is $i\in[1,r]$ such that the label of the edge in $S$ parallel to $\tilde{\tau}(e_{i})$ is 0, then there is no interior point in $N_{h\'}|_{\A\'}$ and in the construction of $N_{h\'}|_{\A\'}$, the $i$-th row of $B\'$ is not used. Hence the $i$-th row of $B\'$ can be arbitrary so long as $B\'$ is skew-symmetrizable. So we can let $\overline{w}_i=\sum\limits_{j=1}^{n}d_{j}w_{ji}e_{j}$ and then focus on proper faces of $S$. We can deal with all $i\in[1,n]$ such that the label of the edge in $S$ parallel to $\tilde{\tau}(e_{i})$ is 0 in the above way. Hence in the following we assume the label of the edge in $S$ parallel to $\tilde{\tau}(e_{i})$ is not 0 for any $i\in[1,r]$.

  According to the discussion in the proof of Theorem \ref{properties in case tsss}, for each edge $l$ in $S$ parallel to some vector in the lattice generating set based on the minimal point of $S$, there is some $t\in\T_{n}$ and a face $S^{t}$ of $N^{t}_{h^{t}}$ which correlates to $S$ with an edge $l^t$ of $S^t$ parallel to $e_i$ correlated to $l$. Then there is $h\'_{t}\in\Z^{ldim(S^t)}$ and a non-negative projection $\tau^t: N_{h\'_t}|_{\A\'}\rightarrow S^t$, where $\A\'_{t}$ is a cluster algebra associated to $B\'_{t}=(b_{ij}\')$. Let $l\'$ be another edge of $S^t$ containing a common vertex of $L^t$ with two points $p<q$ on it such that there does not exist other points in $\overline{pq}$. Assume $i$ and $k$ are the indices satisfying $\tilde{\tau}^t(e_i)$ and $\tilde{\tau}^t(e_k)$ parallel to $l^t$ and $l\'$ respectively. Then due to the definition of polytope projection and Construction \ref{construction}, we could choose $b_{ik}\'$ equal to $deg_{x_i}(q)-deg_{x_i}(p)$, which is $\sum\limits_{s=1}^n b_{is}^t\alpha_s$, where $(\alpha_1,\cdots,\alpha_n)=q-p=\tilde{\tau}^t(e_k)$. Hence the $i$-th row of the desired equation holds for $t$.

  Then it suffices to prove the aimed formula will be maintained under one step of mutations. Due to the independence of each row in the desired equation, we may assume the result holds for $t_{1}$ for any row and $t_{2}$ is connected to $t_{1}$ by an edge labeled $k$. For any face $S^{t_{2}}$ of $N^{t_{2}}_{h^{t_{2}}}$, there is a face $S^{t_{1}}$ of $N^{t_{1}}_{h^{t_{1}}}$ correlated to it according to Theorem \ref{properties in case tsss} (iii).

  If $ldim(S^{t_{1}})<ldim(S^{t_{2}})$, we can embed $N_{h\'_{t_{1}}}$ to a higher space with an extra coordinate $z_{k\'}$ and extend $\tilde{\tau}^{t_{1}}$ as well as $B\'_{t_{1}}$ by setting $\tilde{\tau}^{t_{1}}(e_{k\'})=e_{k}$ and adding the $k\'$-th row and column to $B\'_{t_{1}}$ according to the $k$-th row and column of $B_{t_{1}}$ respectively. Dually when $ldim(S^{t_{1}})>ldim(S^{t_{2}})$. So in the following we may assume $\A\'_{t_1}$ and $\A\'_{t_2}$ have the same rank.

  If there is no segment $e\in N_{h\'_{t_{1}}}|_{\A\'_{t_{1}}}$ or $e\in N_{h\'_{t_{2}}}|_{\A\'_{t_{2}}}$ satisfying $\tilde{\tau}^{t_{1}}(e)= e_{k}$ or $\tilde{\tau}^{t_{2}}(e)= e_{k}$, then according to the mutation formula (\ref{equation: mutation of x}) and (\ref{equation: mutation of y}), it can be seen that $S^{t_{2}}$ is isomorphic to $S^{t_{1}}$. Hence we can find an isomorphism $\tau^{t_{2}}$ from $N_{h\'_{t_{1}}}|_{\A\'_{t_{1}}}$ to $S^{t_{2}}$, i.e., in this case $h\'_{t_{2}}=h\'_{t_{1}}$ and $B\'_{t_{2}}=B\'_{t_{1}}$. For simplicity, assume there are correlative edges in $S^{t_1}$ and $S^{t_2}$ under $\mu_k$ parallel to $\tilde{\tau}^{t_1}(e_i)$ and $\tilde{\tau}^{t_2}(e_i)$ respectively for any $i$. Denote $B\'_{t_{1}}=(b_{ij}\')$.

  Due to the mutation formula (\ref{equation: mutation of x}) and (\ref{equation: mutation of y}), we can calculate that $w^{t_{2}}_{ji}=w^{t_{1}}_{ji}$ if $j\neq k$, while $w^{t_{1}}_{ki}=\sum\limits_{l}w^{t_{1}}_{li}[\epsilon_{k} b^{t_{1}}_{kl}]_{+}-w^{t_{1}}_{ki}$ for some $\epsilon_{k}\in\{\pm 1\}$ (here the choice of $\epsilon$ is according to that $S^{t_{2}}$ is an upper face or a bottom face with respect to the $k$-th coordinate). Hence we also have $\overline{w}^{t_{2}}_{ji}=\overline{w}^{t_{1}}_{ji}$ if $j\neq k$, while $$\overline{w}^{t_{1}}_{ki}=\sum\limits_{l}\frac{d_{k}}{d_{l}}\overline{w}^{t_{1}}_{li}[\epsilon_{k} b^{t_{1}}_{kl}]_{+}-\overline{w}^{t_{1}}_{ki}=\sum\limits_{l}\overline{w}^{t_{1}}_{li}[-\epsilon_{k} b^{t_{1}}_{lk}]_{+}-\overline{w}^{t_{1}}_{ki}.$$ Therefore, it can be checked that
  \begin{equation*}
    \begin{array}{ll}
       & \sum\limits_{l,s}\overline{w}^{t_{2}}_{li}b^{t_{2}}_{ls}w^{t_{2}}_{sj} \\
      = & \sum\limits_{l,s\neq k}\overline{w}^{t_{1}}_{li}(b^{t_{1}}_{ls}+[-\epsilon_{k}b^{t_{1}}_{lk}]_{+}b^{t_{1}}_{ks}+b^{t_{1}}_{lk}[\epsilon_{k} b^{t_{1}}_{ks}]_{+})w^{t_{1}}_{sj}\\
       &\quad-\sum\limits_{s\neq k}(\sum\limits_{u}\overline{w}^{t_{1}}_{ui}[-\epsilon_{k} b^{t_{1}}_{uk}]_{+}-\overline{w}^{t_{1}}_{ki})b^{t_{1}}_{ks}w^{t_{1}}_{sj}-\sum\limits_{l\neq k}\overline{w}^{t_{1}}_{li}b^{t_{1}}_{lk}(\sum\limits_{v}w^{t_{1}}_{vj}[\epsilon_{k} b^{t_{1}}_{kv}]_{+}-w^{t_{1}}_{kj}) \\
      = & \sum\limits_{l,s}\overline{w}^{t_{1}}_{li}b^{t_{1}}_{ls}w^{t_{1}}_{sj} \\
      = & b_{ij}\'
    \end{array}
  \end{equation*}
  So $B_{t_{2}}\'=\overline{W}_{t_{2}}^{\top}B_{t_{2}}W_{t_{2}}$.

  Otherwise we may assume there is a segment $e\in N_{h\'_{t_{1}}}|_{\A\'_{t_{1}}}$ satisfying $\tilde{\tau}^{t_{1}}(e)=e_{k}$, the other case is dual. In this case we also use $k\'$ to denote the label of $e$, then $S^{t_{2}}$ is isomorphic to $\mu_{k\'}(N_{h\'_{t_{1}}})|_{\A\'}$. Hence $B\'_{t_{2}}=\mu_{k\'}(B\'_{t_{1}})$. On the other hand, since $\tilde{\tau}^{t_{1}}(e)=e_{k}$, we have $w^{t_{1}}_{jk\'}=\delta_{jk}$. And similar to the first case, we can also calculate to see that
  \begin{equation*}
    w_{ji}^{t_{2}}=\left\{\begin{array}{ll}
                                       \sum\limits_{l}w^{t_{1}}_{li}[sgn(b_{ik\'}\') b^{t_{1}}_{kl}]_{+}-w^{t_{1}}_{ki}, & \text{if }j=k; \\
                                       w^{t_{1}}_{ji}, & \text{otherwise.}
                                     \end{array}\right.
  \end{equation*}
  and
  \begin{equation*}
    \overline{w}_{ji}^{t_{2}}=\left\{\begin{array}{ll}
                                       \sum\limits_{l}\overline{w}^{t_{1}}_{li}[-sgn(b_{ik\'}\') b^{t_{1}}_{lk}]_{+}-\overline{w}^{t_{1}}_{ki}, & \text{if }j=k;\\
                                       \overline{w}^{t_{1}}_{ji}, & \text{otherwise.}
                                     \end{array}\right.
  \end{equation*}
  Therefore, for any $i,j\in I_{h_{t_{2}}\'}\setminus\{k\'\}$,
  \begin{equation*}
    \begin{array}{rl}
       & \sum\limits_{l,s}\overline{w}^{t_{2}}_{li}b^{t_{2}}_{ls}w^{t_{2}}_{sj} \\
      = & \sum\limits_{l,s\neq k}\overline{w}^{t_{1}}_{li}(b^{t_{1}}_{ls}+[b^{t_{1}}_{lk}]_{+}b^{t_{1}}_{ks}+b^{t_{1}}_{lk}[b^{t_{1}}_{ks}]_{+})w^{t_{1}}_{sj} \\
       & \quad -\sum\limits_{s\neq k}(\sum\limits_{u}\overline{w}^{t_{1}}_{ui}[-sgn(b_{ik\'}\') b^{t_{1}}_{uk}]_{+}-\overline{w}^{t_{1}}_{ki})b^{t_{1}}_{ks}w^{t_{1}}_{sj}-\sum\limits_{l\neq k}\overline{w}^{t_{1}}_{li}b^{t_{1}}_{lk}(\sum\limits_{v}w^{t_{1}}_{vj}[sgn(b_{jk\'}\') b^{t_{1}}_{kv}]_{+}-w^{t_{1}}_{kj})\\
      = & \sum\limits_{l,s}\overline{w}^{t_{1}}_{li}b^{t_{1}}_{ls}w^{t_{1}}_{sj}\\
       & \quad+(\sum\limits_{l}\overline{w}^{t_{1}}_{li}([b^{t_{1}}_{lk}]_{+}-[-sgn(b_{ik\'}\') b^{t_{1}}_{lk}]_{+}))(\sum\limits_{s}b^{t_{1}}_{ks}w^{t_{1}}_{sj}) +(\sum\limits_{l}\overline{w}^{t_{1}}_{li}b^{t_{1}}_{lk})(\sum\limits_{s}([b^{t_{1}}_{ks}]_{+}-[-sgn(b_{k\'j}\') b^{t_{1}}_{ks}]_{+})w^{t_{1}}_{sj})\\
      = & \sum\limits_{l,s}\overline{w}^{t_{1}}_{li}b^{t_{1}}_{ls}w^{t_{1}}_{sj} +[\sum\limits_{l}\overline{w}^{t_{1}}_{li}b^{t_{1}}_{lk}]_{+}(\sum\limits_{s}b^{t_{1}}_{ks}w^{t_{1}}_{sj}) +(\sum\limits_{l}\overline{w}^{t_{1}}_{li}b^{t_{1}}_{lk})[\sum\limits_{s}b^{t_{1}}_{ks}w^{t_{1}}_{sj}]_{+}\\
     = & b_{ij}\'+[b_{ik}\']_{+}b_{kj}\'+b_{ik}\'[b_{kj}\']_{+},
    \end{array}
  \end{equation*}
  while
 $$        \sum\limits_{l,s}\overline{w}^{t_{2}}_{lk\'}b^{t_{2}}_{ls}w^{t_{2}}_{sj} \\
      =  \sum\limits_{s}b^{t_{2}}_{ks}w^{t_{2}}_{sj} \\
      =  \sum\limits_{s} -b^{t_{1}}_{ks}w^{t_{2}}_{sj}\\
      =  \sum\limits_{l,s}-\overline{w}^{t_{1}}_{lk\'}b^{t_{1}}_{ls}w^{t_{1}}_{sj}\\
      =  -b_{k\'j},$$
  and similarly $\sum\limits_{l,s}\overline{w}^{t_{2}}_{li}b^{t_{2}}_{ls}w^{t_{2}}_{sk\'}=-b_{ik\'}$. So we get $B\'_{t_{2}}=\mu_{k}(B\'_{t_{1}})=\overline{W}^{\top}_{t_{2}}B_{t_{2}}W_{t_{2}}$, which completes the proof.
\end{Proof}

According to Theorem \ref{properties in case tsss}, Theorem \ref{properties of N_h for skew-symmetrizable} and Construction \ref{construction}, when $\A$ is a skew-symmetrizable cluster algebra with principal coefficients, we can calculate the proper faces of $N_{h}$ in the following way. Here, we calculate faces with dimension 2 for example.

We define the sets $\overline E^{(i)}$, $\tilde{E}^{(i)}$ and $N^{(i)}$ recursively for $i\in \N$. First denote $\overline{E}^{(0)}=E_{h}$ and $N^{(-1)}=\emptyset$. Assign a label to each edge in $E_h$ as in the definition of $E_h$.

 Inductively assume $\overline E^{(j)}$ and $N^{(j-1)}$ has been determined for $j\in[0,t]$, while $\tilde{E}^{(j-1)}$ has been determined for $j\in[1,t]$. Then, define $\tilde{E}^{(t)}$ to be the union of $\overline{E}^{(t)}$ and the set consisting of segments $\overline{pq}$ satisfying that $p$ is a vertex in $\overline{E}^{(t)}\setminus(\overline{E}^{(t)}\cap N^{(t-1)})$, $q\notin conv(\bigcup\limits_{i=0}^{t-1}\tilde{E}^{(i)}\cup\overline{E}^{(t)})$ and
\[q=\left\{\begin{array}{ll}
           & \text{if there is a point }p\'\text{ lying in the same}\\
           p-[-deg_{x_{k}}(X^{h+pB^{\top}})]_{+}e_{k},& \text{2-dimensional face of }conv(\bigcup\limits_{i=0}^{t-1}\tilde{E}^{(i)}\cup\overline{E}^{(t)})\text{ with }\\
           &p \text{ such that }p\'-e_{k}\in conv(\bigcup\limits_{i=0}^{t-1}\tilde{E}^{(i)}\cup\overline{E}^{(t)});\\
           p+[-deg_{x_{k}}(X^{h+pB^{\top}})]_{+}e_{k}, & \text{otherwise.}
         \end{array}\right.\]
for some $k\in[1,n]$. Assign a label $\kappa(l)$ to an edge $l\in \tilde{E}^{(t)}$ if $l$ is parallel to $e_{\kappa(l)}$. Then define $N^{(t)}$ to be the convex hull of $\bigcup\limits_{i=0}^{t}\tilde{E}^{(i)}$.

For any 2-dimensional face $S\'$ of $N^{(t)}$ such that there is an edge of $S\'$ parallel to $v_i$ contained in $\bigcup\limits_{i=0}^{t}\tilde{E}^{(i)}$ for $i\in[1,r]$, where $\{v_1,\cdots,v_r\}$ is the lattice generating set of $S\'$ based on its minimal point, and at least one edge not in $\bigcup\limits_{i=0}^{t}\tilde{E}^{(i)}$, we construct a polytope $S$ of dimension 2 such that $S\'\subseteq S$ by the following conditions:

(a)\; Let $\tilde{\tau}:\;\R^r\;\rightarrow\;\R^n$ be a linear map determined by $\tilde{\tau}(e_{i})=v_{i}=\sum\limits_{j=1}^{n}w_{ji}e_{j}\in\N^{n}$ for $i\in[1,r]$.

(b)\; $\A\'$ is a cluster algebra with principal coefficients, whose initial exchange matrix $B\'=\overline{W}^{\top}BW$, where $W=(\tilde{\tau}(e_{1})^{\top},\cdots,\tilde{\tau}(e_{r})^{\top})$, $\overline{W}=(\overline{w}_1^{\top},\cdots,\overline{w}_r^{\top})$ and $\overline{w}_i=\sum\limits_{j=1}^{n}\frac{d_{j}}{d_{s}}w_{ji}e_{j}$ with $s\neq0$ the label of an edge in $S$ parallel to $\tilde{\tau}(e_{i})$ and $\overline{w}_i=\sum\limits_{j=1}^{n}d_{j}w_{ji}e_{j}$ when the label is 0.

(c)\; There is a vector $h_{S}\in\Z^{r}$ and a non-negative projection $\tau:\; N_{h_{S}}|_{\A\'}\longrightarrow S$ such that $\tilde{\tau}$ is induced by $\tau$. That is, {\em the 2-dimensional face $S$ is constructed as the image of $\tau$}.

Let $\mathbf{S}^{(t)}$ be the set consisting of all such 2-dimensional polytopes $S$. Define $\overline{E}^{(t+1)}$ to be a set consisting of edges $l$ of $S\in\mathbf{S}^{(t)}$ which is not in $\bigcup\limits_{i=0}^{t}\tilde{E}^{(i)}$. The labels are assigned according to those of edges in $N_{h_{S}}|_{\A\'}$.

According to the definition of $N_{h}$ and Theorem \ref{properties of N_h for skew-symmetrizable}, $N^{(i)}$ is a sub-polytope of $N^{(i+1)}$ and $N^{(i+1)}\setminus N^{(i)}\neq \emptyset$ until $N^{(i)}=N_{h}$. Hence $\lim\limits_{i\rightarrow+\infty} N^{(i)}=N_h$. When $N_{h}$ is a finite polytope, calculation along the above way ends in finitely many times. Otherwise the process never ends, however the polytopes are always finite locally. So we can still have such construction formally.

Similarly, we can also calculate any proper face of $N_{h}$ with higher dimension by Theorem \ref{properties of N_h for skew-symmetrizable}.

The later process to determine weights of interior points in $N_{h}$ is just as that in Construction \ref{construction}, but it is easier in calculation since we have known all proper faces now.

\begin{Example}
  In Example \ref{example of essential skeleton}, for the proper face $S$ containing $(1,1,0),(5,3,0)$ and $(5,3,3)$, we find a linear map:
  $$\tilde{\tau}:\quad\R^{2}\quad\longrightarrow\quad\R^{3}$$
  such that $\tilde\tau{(1,0)}=(2,1,0)$ and $\tilde\tau{(0,1)}=(0,0,1)$. It can be checked that there is an isomorphism
  $$\tau:\quad N_{(-2,1)}|_{\A\'}\quad\longrightarrow\quad S$$
  which can induce $\tilde{\tau}$, where the initial exchange matrix $B\'$ of the cluster algebra $\A\'$ is
  \[B\'=\begin{pmatrix}
          0 & 1 \\
          -2 & 0
        \end{pmatrix}=\overline{W}^{\top}BW,\]
  as
  \[W=\begin{pmatrix}
        2 & 0 \\
        1 & 0 \\
        0 & 1
      \end{pmatrix},\qquad\overline{W}=\begin{pmatrix}
                                         1 & 0 \\
                                         1 & 0 \\
                                         0 & 1
                                       \end{pmatrix}\]
\end{Example}
\vspace{4mm}

{\bf Acknowledgements:}\;  {\em This project is supported by the National Natural Science Foundation of China(No.12071422 and No.12131015).}

\emph{We thank Jiarui Fei, Y. Gyoda, Siyang Liu, Fan Qin and R. Schiffler for their advice and pointing out faults in our previous version. Also we thank Peigen Cao and Fan Qin for recommending us the example in \cite{Zh}.}


\begin{thebibliography}{99}
\bibitem{BFZ} A. Berenstein, S. Fomin and A.Zelevinsky, Cluster algebras, III. Upper bounds and double Bruhat cells, Duke Math. J. 126 (2005),1-52.

\bibitem{BZ} A. Berenstein, A. Zelevinsky, Quantum cluster algebras. Adv. in Mathematics, 195 (2005): 405-455.

\bibitem{CGMMRSW} Man Wai Cheung, M. Gross, G. Muller, G. Musiker, D. Rupel, S. Stella and H. Williams, The greedy basis equals the theta basis: A rank two haiku, Journal of Combinatorial Theory, Series A, 145 (2017), 150-171.

\bibitem{CL} Peigen Cao and Fang Li, The enough $g$-pairs property and denominator vectors of cluster algebras, Mathematische Annalen. 377 (2020), 1547-1572.

\bibitem{F} Jiarui Fei, Combinatorics of $F$-Polynomials, IMRN, 2023 (2013), 7578-7615.

\bibitem{FK}  Changjian Fu and B. Keller, On cluster algebras with coeffcients and 2-Calabi-Yau categories, Trans. Amer. Math. Soc. 362 (2010) no. 2, 859-895.

\bibitem{FZ} S. Fomin and A. Zelevinsky, Cluster algebras: Notes for the CDM-03 conference, in: CDM 2003: Current Developments in Mathematics, International Press, 2004.

\bibitem{FZ1} S. Fomin and A. Zelevinsky, Cluster algebras. I. Foundations, J. Amer. Math. Soc. 15 (2002),
no. 2, 497-529 (electronic).

\bibitem{FZ2} S. Fomin and A. Zelevinsky, Cluster algebras, II. Finite type classification. Invent. Math., 154 (2003),no. 1, 63-121.

\bibitem{FZ4} S. Fomin and A. Zelevinsky, Cluster algebras, IV. Coefficients. Compos. Math., 143 (2007), 112-164.

\bibitem{GLS} C. Geiss, B. Leclerc and J. Schr\"{o}er: Factorial cluster algebras. Doc. Math., 18 (2013), 249-274.

\bibitem{GHKK} M. Gross, P. Hacking, S. Keel and M. Kontsevich, Canonical bases for cluster algebras, J. Amer. Math. Soc. 31 (2018), 497-608.

\bibitem{HL} Ming Huang and Fang Li, Unfolding of acyclic sign-skew-symmetric cluster algebras and applications to positivity and F-polynomials, Advances
     in Mathematics 340 (2018): 221-283.


\bibitem{K} B. Keller, On cluster theory and quantum dilogarithm identities, In Representations of algebras and related topics, EMS Ser. Congr. Rep., (2011), 85-116.

\bibitem{K2} B. Keller, Cluster algebras and derived categories, in Derived Categories in Algebraic Geometry, EMS Ser.
Congr. Rep., Eur. Math. Soc., Z\"{u}rich, 2012, 123-183

\bibitem{LLS} K. Lee, L. Li, R. Schiffler, Newton polytopes of rank 3 cluster variables, arXiv:1910.14372.

\bibitem{LLZ} K. Lee, L. Li, A. Zelevinsky, Greedy elements in rank 2 cluster algebras, Selecta Mathematica. New Series, 20 (2012), no. 1, 57-82.

\bibitem{LLZ2} K. Lee, L. Li, A. Zelevinsky, Positivity and tameness in rank 2 cluster algebras, J Algebr Comb 40(2014), 823-840.

\bibitem{M} G. Muller, The existence of a maximal green sequence is not invariant under quiver mutation, Electron. J. Combin. 23 (2016), no. 2, P2.47.

\bibitem{N} T. Nakanishi, Cluster algebras and scattering diagrams, part II. Cluster patterns and scattering diagrams, arXiv:2103.16309.

\bibitem{P} Jie Pan, Polytope realization of cluster structures, arXiv:2312.15327.

\bibitem{Z} G. Ziegler, Lectures on polytopes. In Graduate texts in mathematics, Vol. 152. Berlin: Springer, 1995.

\bibitem{Zh} Yan Zhou, Cluster Structures and Subfans in Scattering Diagrams. Symmetry Integrability and Geometry-methods and Applications 16 (2020): 013.
\end{thebibliography}
\end{document}